\documentclass[11pt]{amsart}
\usepackage[english]{babel}
\usepackage{amsfonts,latexsym,amsthm,amssymb,graphicx}
\usepackage{mathabx}
\usepackage[all]{xy}
\usepackage[usenames]{color}
\usepackage{tikz}
\usetikzlibrary{shapes}
\usepackage{xypic}
\usetikzlibrary{decorations.pathreplacing}
\usepackage{amsmath}
\usepackage{graphicx}
\usepackage[enableskew]{youngtab}
\usepackage[utf8]{inputenc}
\usepackage[english]{babel}
\usepackage{tikz}
\usetikzlibrary{arrows}
\usepackage{float}
\usepackage{amscd}
\usepackage[colorlinks=true, linkcolor=blue, citecolor=blue, urlcolor=blue]{hyperref}

\def\R{{\mathbb R}}

\newcommand{\C}{{\mathbb C}}
\newcommand{\Z}{{\mathbb Z}}

\CompileMatrices

\newtheorem{thm}{Theorem}[section]
\newtheorem{lemma}[thm]{Lemma}
\newtheorem{cor}[thm]{Corollary}

{\theoremstyle{defn} \newtheorem{defn}[thm]{Definition}}

{\theoremstyle{remark} \newtheorem{remark}[thm]{Remark}
\newtheorem{example}[thm]{Example}}

\begin{document}

\title{Combinatorial curve neighborhood of the affine flag manifold of type $A_{n-1}^1$}

\author{Songul Aslan}
\address{ Department of Mathematics, 460 McBryde Hall, Virginia Tech University, Blacksburg VA 24060 USA}\email{sonas@vt.edu}

\maketitle
\begin{abstract}
Let $\mathcal{X}$ be the affine flag manifold of Lie type $A_{n-1}^{(1)}$ where $n \geq 3$ and let $W_{\text{aff}}$ be the associated affine Weyl group. The moment graph for $\mathcal{X}$ encodes the torus fixed points (corresponding to elements of the affine Weyl group $W_{\text{aff}}$) and the torus stable curves in $\mathcal{X}$. Given a fixed point $u\in W_{\text{aff}}$ and a degree $\mathbf{d}=(d_0,d_1,...,d_{n-1})\in \mathbb{Z}_{\geq 0}^{n}$, the combinatorial curve neighborhood is the set of maximal elements in the moment graph of $\mathcal{X}$ which can be reached from $u'\leq u$ by a chain of curves of total degree $\leq \mathbf{d}$. In this paper we give combinatorial formulas and algorithms for calculating these elements in $\mathcal{X}$. \end{abstract}

\tableofcontents

\section{Introduction}

Let $X = G/P$ be a flag manifold defined by a semisimple complex Lie group $G$ and a parabolic subgroup $P$ and let $\Omega \subset X$ be a Schubert variety. Fix an effective degree $\mathbf{d}\in H_{2}(X)$. The (geometric) \textit{curve neighborhood} $\Gamma_{\mathbf{d}}(\Omega)$ is the closure of the union of all rational curves of degree $\mathbf{d}$ in $X$ which intersect $\Omega$. The curve neighborhood $\Gamma_{\mathbf{d}}(\Omega)$ was originally studied by Buch, Chaput, Mihalcea and Perrin \cite{BCMP} and the concept was later developed and used e.g. in \cite{MB01,MT01,MM01,LR18,shifler.withrow,BCMP:qkpos}, in relation
to the study of the quantum cohomology and quantum $K$-theory ring of $X$. It was proved in \cite{BCMP} that $\Gamma_{\mathbf{d}}(\Omega)$ is irreducible whenever $\Omega$ is irreducible. In particular,  if $\Omega$ is a Schubert variety in X then $\Gamma_{\mathbf{d}}(\Omega)$ is also a Schubert variety.  In \cite{MB01}, Buch and Mihalcea provided an explicit combinatorial formula for the Weyl group element corresponding to $\Gamma_{\mathbf{d}}(\Omega)$ when $\Omega$ is a Schubert variety in $X$. It has been also shown that the calculation of the curve neighborhoods is encoded in the \textit{moment graph} of $X$ which is a graph encoding the $T$-fixed points and the $T$-stable curves in $X$ where $T$ is a maximal torus of $G$. In \cite{W19} Withrow has studied curve neighborhoods to compute a presentation for the small quantum cohomology ring of a particular Bott- Samelson variety in Type A. Moreover, Mihalcea and Shifler use the technique of curve neighborhoods to prove a Chevalley formula in the equivariant quantum cohomology of the odd symplectic Grassmannian, see \cite{LR18}. The curve neighborhood $\Gamma_{\mathbf{d}}(\Omega)$ has also been studied in the case when $X$ is an affine flag manifold by Mare and Mihalcea, see \cite{MM01}. They defined 
an affine version of the quantum cohomology ring and gave a combinatorial description of the curve neighborhood for ``small'' degrees. Furthermore, a combinatorial formula for the curve neighborhood of affine flag manifold of Lie type $A_{1}^{1}$ has been recently given by Norton and Mihalcea; see \cite{MT01}. The goal of this paper is to give a combinatorial formula for the curve neighborhoods of Schubert varieties in the affine flag manifold of Lie type $A_{n-1}^{(1)}$ for any degree, where $n\geq 3$. 

\subsection{Statement of results} Next we will introduce some notation and recall some definitions before we give an overview of our results. Let $G$ be the special linear group $SL_{n}(\C)$. Let $B\subset G$ be the Borel subgroup, the set of upper triangular matrices, and let $T\subset B$ be the maximal torus, the set of diagonal matrices. We denote by $\mathfrak{g}$ the Lie algebra $\mathfrak{sl}_{n}(\C)$ of $G$.  Let $\Pi$ be the standard root system associated to the triple $(T,B,G)$. Also, let $\Delta=\{\alpha_{1},\alpha_{2},...,\alpha_{n-1}\} \subset \Pi$ be the set of simple roots. This determines a partition of $\Pi=\Pi^{+}\sqcup \Pi^{-}$ such that  $\Pi^{+}=\Pi \cap \bigoplus_{i=1}^{n-1}\mathbb{Z}_{\geq 0}\,\alpha_{i}$ is the set of positive roots and $\Pi^{-}=\Pi \cap \bigoplus_{i=1}^{n-1}\mathbb{Z}_{\leq 0}\,\alpha_{i}$ is the set of negative roots (these roots are described in chapter 2). Denote by $W$ the Weyl group. The Weyl group $W$ is generated by the simple reflections, $s_1,s_2,...,s_{n-1},$ where $s_{i}:=s_{\alpha_{i}}$, $i=1,2,...,n-1$. 
Now, let $\mathfrak{g}_{\text{aff}}$ be the affine Kac-Moody Lie algebra associated to the Lie algebra $\mathfrak{g}.$ Denote by $\Pi_{\text{aff}}$ the affine root system associated to $\mathfrak{g}_{\text{aff}}.$ Let $\Delta_{\text{aff}}=\{\alpha_{0},\alpha_{1},...,\alpha_{n-1}\}\subset \Pi_{\text{aff}}$ be the affine simple roots. This determines a partition of $\Pi_{\text{aff}}$ into positive and negative affine roots: $\Pi_{\text{aff}}=\Pi_{\text{aff}}^{+}\sqcup \Pi_{\text{aff}}^{-}$ where  $\Pi_{\text{aff}}^{-}=-\Pi_{\text{aff}}^{+}$. We denote by $\Pi_{\text{aff}}^{\text{re},\,+}$ the set of affine positive real roots. Let $W_{\text{aff}}$ be the affine Weyl group of type $A_{n-1}^{(1)}$ which is generated by the simple reflections $s_{0},s_{1},..,s_{n-1}$ where $s_{i}:=s_{\alpha_{i}}$, $i=0,1,...,n-1$.  For $w\in W_{\text{aff}},$ the \textit{length} of $w$ is denoted by $\ell{(w)}$. 
Now, let $\mathcal{X}$ be the affine flag manifold in type $A_{n-1}^{(1)}$. The (undirected) \textit{moment graph} for $\mathcal{X}$ is defined by the following: 
\begin{itemize} 

\item The set of vertices is $W_{\text{aff}}.$
\item There is an edge between $u, v \in W_{\text{aff}}$ in the moment graph  if and only if there exists an affine positive real root $\alpha$ such that $v = us_{\alpha}$. This situation is denoted by $$u \stackrel{\alpha} \longrightarrow v$$  We say that the \textit{degree} of this edge is $\alpha.$
(In general the degree is the coroot $\alpha^\vee$, but in simply laced cases roots and coroots are the same.)
\end{itemize}
 A \textit{chain} from $u$ to $v$ in the moment graph is a succession of adjacent edges, starting with $u$ and ending with $v$;
$$\pi: u=u_{0} \stackrel{\beta_{0}} \longrightarrow u_{1}\stackrel{\beta_{1}}\longrightarrow ...\stackrel{\beta_{k-2}}\longrightarrow u_{k-1}\stackrel{\beta_{k-1}}\longrightarrow u_{k}=v .$$

The \textit{degree} of the chain $\pi$ is $\text{deg}(\pi)=\beta_{0}+\beta_{1}+...+\beta_{k-1}.$ A chain is called \textit{increasing} if the length is increasing at each step i.e., if $\,\,\ell{(u_{i})}>\ell{(u_{i-1})}$ for all $i.$ There is a partial order on the elements of $W_{\text{aff}}$, the Bruhat partial order, 
defined by $u<v$ if and only if there exists an increasing chain starting with $u$ and ending with $v.$

\begin{example} Assume that $\mathcal{X}$ is the affine flag manifold associated to $A_{2}^{(1)}$. Then the moment graph for $\mathcal{X}$ up to elements that can be obtained by a chain of degree at most $\alpha_{0}+\alpha_{1}+\alpha_{2}$ is given in the following figure with each edge labelled by its degree;

\bigskip

\end{example}

\begin{figure} [h!]\label{fig1}

\begin{center}
\includegraphics[width=120mm,scale=0.5]{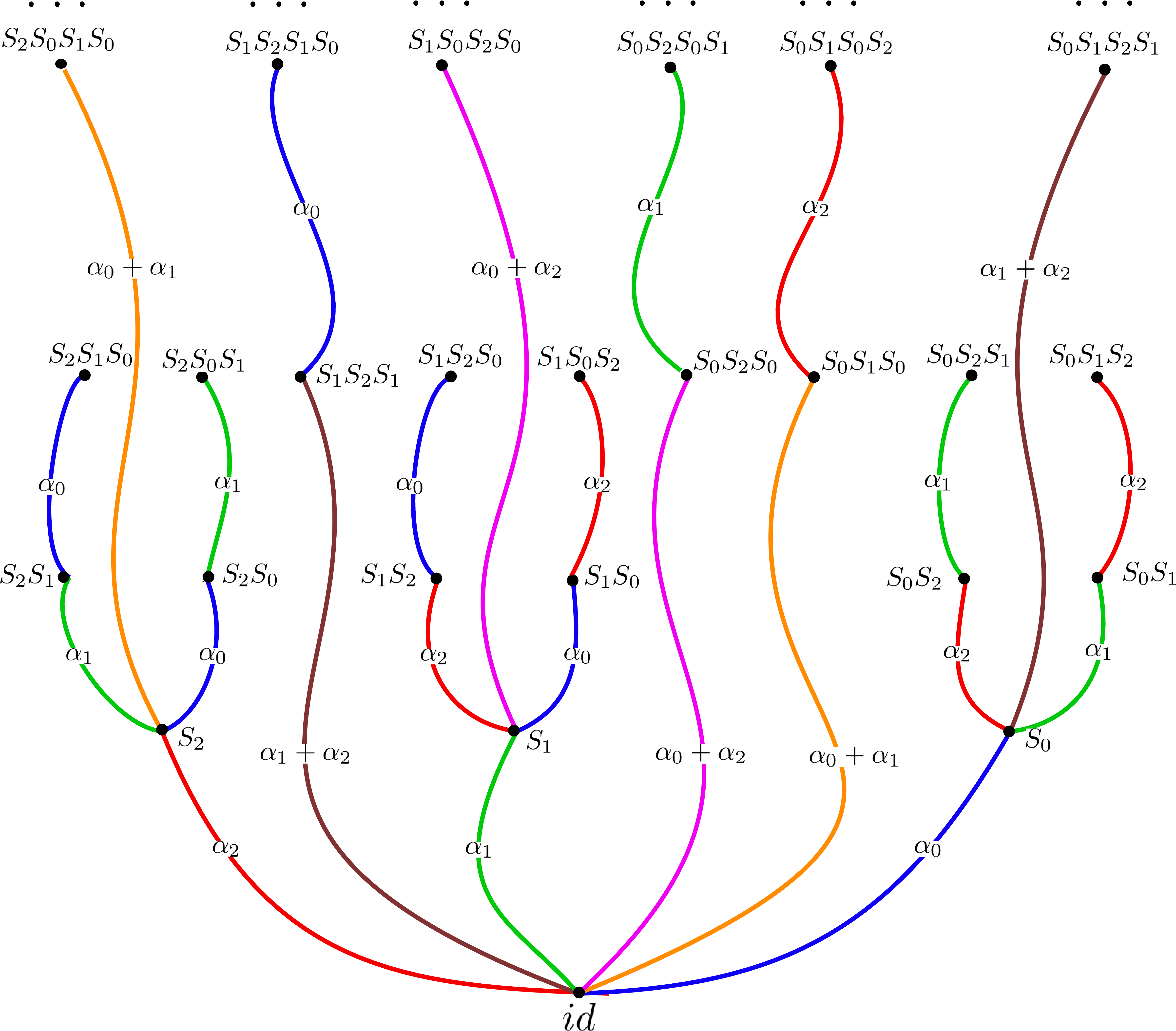}
\end{center}
\caption{The moment graph for the affine flag manifold of type $A_{2}^{(1)}.$}
\end{figure}

\bigskip

\begin{remark} The vertices and edges of this graph correspond to the $T$-fixed points and $T$-stable curves in the affine flag manifold, see \cite[\S12]{kumar} for details.
\end{remark}

A \textit{degree} $\mathbf{d}$ is an $n$-tuple of nonnegative integers $(d_0 , d_1,...,d_{n-1})$. There is a natural partial order on degrees: If $\mathbf{d}=(d_{0},...,d_{n-1})$ and $\mathbf{d'}=(d'_{0},...,d'_{n-1})$ then $\mathbf{d}\geq \mathbf{d'}$ if and only if $d_{i}\geq d'_{i}$ for all $i\in \{0,1,...,n-1\}.$ 

Now, let a degree $\mathbf{d}$ and an element $u$ of $W_{\text{aff}}$ be given. Then, inspired by the geometric definition of curve neighborhoods in \cite{MB01} and \cite{MM01}, the (combinatorial) \textit{curve neighborhood} $\Gamma_{\mathbf{d}}(u)$ is the set of elements $v$ in $W_{\text{aff}}$ such that:

\begin{itemize} 

\item[1)] $v$ can be joined to some $u'\leq u$ (in the moment graph) by a chain of degree $\leq \mathbf{d}$; 

\item[2)]  The elements $v$ are maximal among all elements satisfying $(1)$. 

\end{itemize}

In other words, to compute $\Gamma_{\mathbf{d}}(u)$ we first consider all the paths in the moment graph that start with $u'$ such that $u'\leq u$ and have a total degree at most $\mathbf{d}$. Then the curve neighborhood is given by the set of the elements with the maximal length that can be reached by the paths.

\begin{example} Assume that $\mathcal{X}$ is the affine flag manifold associated to $A_{2}^{(1)}$. The moment graph for $\mathcal{X}$ is shown in Figure \ref{fig1}. Now, let's consider all the paths in the moment graph that start with the identity element and have a degree at most $\alpha_{0}+\alpha_{1}+\alpha_{2}$. Then the elements which can be reached by the paths and have the maximal length are the top elements in the moment graph so 
$$\Gamma_{(\alpha_{0}+\alpha_{1}+\alpha_{2})}(id)= \{s_2s_0s_1s_0,s_1s_2s_1s_0,s_1s_0s_2s_0,s_0s_2s_0s_1,s_0s_1s_0s_2,s_0s_1s_2s_1\}.$$

\end{example}
Now, we will state our first result which allows us to reduce the calculation of $\Gamma_{\mathbf{d}}(w)$ for any given degree $\mathbf{d}$ and $w \in W_{\text{aff}}$ to the calculation of $ \Gamma_{\mathbf{d}}(id)$. The proof of the result can be found in Theorem \ref{thm22}.

\begin{thm} \label{thm0.3} Let $w \in W_{\text{aff}}$ and $\mathbf{d}$ be any degree. Then 
$$\Gamma_{\mathbf{d}}(w)=\text{max}\left \{w\cdot u: u \in \Gamma_{\mathbf{d}}( id) \right \}.$$\end{thm}

Next, we will consider the curve neighborhood of the identity element at degree $\mathbf{d}=(d_{1},d_{2},...,d_{n-1})$ where $d_{i}=0$ for some $i=1,2,...,n-1.$ Here, using the automorphism of the Dynkin diagram $\varphi$, $\mathbf{d}$ can be identified with a finite degree, see Sections \ref{dynkin} and \ref{case of finite degrees} for details. In this case, by  \cite{MB01}, the curve neighborhood $ \Gamma_{\mathbf{d}}(id)$ consists of a unique element;
\begin{equation}\label{equation0.35}\Gamma_{\mathbf{d}}(id)=\{z_{\mathbf{d}}\} \end{equation}
where $z_{\mathbf{d}}:=s_{\alpha}\cdot z_{\mathbf{d}-\alpha}$ such that 
$\alpha$ is a maximal root which is smaller than and equal to $\mathbf{d}$, we refer to definition \ref{def12.5} for details. One might also see Theorem \ref{thm12.7} where we show how $z_{d}$ can be simplified.

Our next result is a corollary of Equation \ref{equation0.35}. Let $\alpha \in \Pi_{\text{aff}}^{\text{re},\,+}$ such that $\alpha<c$. Then

\begin{equation}\label{equation0.36}\Gamma_{\alpha}(id)=\{s_{\alpha}\}. \end{equation}

\begin{example} Assume that the affine Weyl group is associated to $A_{4}^{(1)}$ and $\alpha=\alpha_{0}+\alpha_{4}.$ Then $\Gamma_{\alpha}( id)=\{s_{\alpha_{0}+\alpha_{4}}\}$ by Equation \ref{equation0.36}. 

\end{example}

The rest of our results are about the instances where the degrees are not finite and proven by using some techniques which are different than those one can find in \cite{MB01} and \cite{MM01}.

Now, we will state the result for the curve neighborhood of the identity element at $c$ which is the degree corresponding to the imaginary coroot. In this case the neighborhood is described in terms of the translations $t_\gamma$ where $\gamma$ is a (co)root; see Section \ref{affine Weyl group} for details. The proof of this result can be found in Theorem \ref{thm14}.

\begin{thm} \label{thm0.55} We have 
\begin{enumerate}
\item[1)] $\Gamma_{c} ( id)=\{t_{\gamma}:\gamma \in \Pi \}$. 
\item[2)] $|\Gamma_{c}( id)|=n(n-1)$.
\item[3)] For all $w\in \Gamma_{c}( id)$,  $\ell(w)=2(n-1).$ 

\end{enumerate}
\end{thm}

\begin{example} Suppose that the affine Weyl group $W_{\text{aff}}$ is associated to $A_{2}^{(1)}.$ Then by Theorem \ref{thm0.55} 
$$\Gamma_{c}( id)=\{t_{\alpha_{1}},t_{\alpha_{2}},t_{\alpha_{1}+\alpha_{2}}, t_{-\alpha_{1}}, t_{-\alpha_{2}},t_{-(\alpha_{1}+\alpha_{2})}\}.$$ 
Moreover, for any $w\in \Gamma_{c}( id)$, $\ell(w)=2(n-1)=2\cdot 2=4$. 

\end{example}

In the next theorem, we will generalize the case in Theorem \ref{thm0.55} by considering a constant $m$ multiple of $c$ where $m\geq 2$ is a positive integer. See Theorem \ref{thm19} for the proof. 

\begin{thm} \label{thm0.57}Let $m\geq 2$ be a positive integer.  Then we have 

\begin{enumerate}

\item[1)] $\Gamma_{mc} ( id)=\{t_{m\gamma}:\gamma \in \Pi \}$. 

\item[2)]  $|\Gamma_{mc}( id)|=n(n-1)$.

\item[3)] For all $w\in \Gamma_{mc}( id)$,  $\ell(w)=2m(n-1).$

\end{enumerate}

\end{thm}

\begin{example} Suppose that the affine Weyl group $W_{\text{aff}}$ is associated to $A_{2}^{(1)}.$ Then by Theorem \ref{thm0.57}
$$\Gamma_{10c}( id)=\{t_{10\alpha_{1}},t_{10\alpha_{2}},t_{10(\alpha_{1}+\alpha_{2})}, t_{-10\alpha_{1}}, t_{-10\alpha_{2}},t_{-10(\alpha_{1}+\alpha_{2})}\}.$$ 
Moreover, for any $w\in \Gamma_{10c}( id)$, $\ell(w)=2m(n-1)=2\cdot 10 \cdot 2=40$. 

\end{example}

In the following theorem we will provide the result for the curve neighborhood of the identity element at degree $c+\alpha$ where $\alpha \in \Pi_{\text{aff}}^{\text{re},\,+}$ and $\alpha<c$. This plays a key role in calculating the most general case in this paper. One can find the proof in Theorem \ref{thm17.9}.

\begin{thm}\label{thm0.59} Let $\alpha\in \Pi_{\text{aff}}^{\text{re},\,+}$ be such that $\alpha<c$. Then

\begin{enumerate}
\item[1)] $\Gamma_{c+\alpha}( id)=\{t_{\beta'} s_{\alpha}:\beta'\in \Pi_{\text{aff}}^{\text{re},\,+}(\alpha)\} \cup \{t_{\beta'-c} s_{\alpha}:\beta'\in \Pi_{\text{aff}}^{\text{re},\,+}(\alpha)\}$
\item[2)] $|\Gamma_{c+\alpha}( id)|=|\{\beta': \beta' \in \Pi_{\text{aff}}^{\text{re},\,+}(\alpha)\}|$
\item[3)] For all $w\in \Gamma_{c+\alpha}( id)$,  $\ell(w)=2(n-1)+\ell{(s_{\alpha})}$ 

\end{enumerate}

where $ \Pi_{\text{aff}}^{\text{re},\,+}(\alpha):=\{\beta' \in \Pi_{\text{aff}}^{\text{re},\,+}: \beta'<c \,\,\text{and either}\,\,  \beta'\leq c-\alpha \,\, \text{or both} \,\,\beta'>\alpha \,\,\text{and} \,\, \beta'\perp \alpha \}$.

\end{thm}

\begin{example}  Let $W_{\text{aff}}$ be the Weyl group of type $A_{4}^{(1)}.$ We compute $\Gamma_{c+\alpha}( id)$ where $\alpha=\alpha_{0}+\alpha_{4}.$ Note that, we have six positive real roots which are smaller than $c-\alpha=\alpha_{1}+\alpha_{2}+\alpha_{3};$ $\beta'_{1}=\alpha_{1}, \beta'_{2}=\alpha_{2}, \beta'_{3}=\alpha_{3}, \beta'_{4}=\alpha_{1}+\alpha_{2}, \beta'_{5}=\alpha_{2}+\alpha_{3} , \beta'_{6}=\alpha_{1}+\alpha_{2}+\alpha_{3}.$ Also, we have only one positive root which  is smaller than $c$, strictly bigger than $\alpha$ and perpendicular to $\alpha$; $\beta'_{7}=\alpha_{0}+\alpha_{1}+\alpha_{3}+\alpha_{4}.$ So by Theorem \ref{thm0.59}, we get
$$ \Gamma_{c+\alpha}( id)=\{t_{\beta'_{1}}s_{\alpha}, t_{\beta'_{2}}s_{\alpha},t_{\beta'_{3}}s_{\alpha},t_{\beta'_{4}}s_{\alpha},t_{\beta'_{5}}s_{\alpha},t_{\beta'_{6}}s_{\alpha},t_{\beta'_{7}-c}s_{\alpha} \}.$$

Moreover, for any $w\in \Gamma_{c+\alpha}( id),$ we have 
$\ell(w)=2(n-1)+\ell{(s_{\alpha})}=2(n-1)+2\, |\text{supp}(\alpha)|-1=2\cdot4+2\cdot2-1=11$. 

\end{example}

Next, we will consider the generalization of the case in the previous theorem to the instance where the degree is given by $mc+\alpha$ such that $m \in \Z^{>0},\,\alpha \in \Pi_{\text{aff}}^{\text{re},\,+}$ and $\alpha<c$. See Theorem \ref{thm18} for the proof.

\begin{thm}\label{thm0.6} Let $\alpha<c$ be an affine positive real root and $m$ be a positive integer. Then

\begin{enumerate}
\item[1)] $\Gamma_{mc+\alpha}(id)=\{t_{m\beta'} s_{\alpha}:\beta'\in \Pi_{\text{aff}}^{\text{re},\,+}(\alpha)\} \cup \{t_{m(\beta'-c)} s_{\alpha}:\beta'\in \Pi_{\text{aff}}^{\text{re},\,+}(\alpha)\}$
 
\item[2)] $|\Gamma_{mc+\alpha}( id)|=|\{\beta': \beta' \in \Pi_{\text{aff}}^{\text{re},\,+}(\alpha)\}|$

\item[3)] For all $w\in \Gamma_{mc+\alpha}( id)$,  $\ell(w)=2m(n-1)+\ell{(s_{\alpha})}$ 
\end{enumerate}
where $ \Pi_{\text{aff}}^{\text{re},\,+}(\alpha)$ is defined in Theorem \ref{thm0.59}.

\end{thm}

\begin{example}Suppose that the affine Weyl group $W_{\text{aff}}$ is of type $A_{4}^{(1)}.$ We compute $\Gamma_{12c+\alpha}( id)$ where $\alpha=\alpha_{0}+\alpha_{4}.$  Then, by Theorem \ref{thm0.6}, we get 
$$ \Gamma_{12c+\alpha}( id)=\{t_{12\beta'_{1}}s_{\alpha}, t_{12\beta'_{2}}s_{\alpha},t_{12\beta'_{3}}s_{\alpha},t_{12\beta'_{4}}s_{\alpha},t_{12\beta'_{5}}s_{\alpha},t_{12\beta'_{6}}s_{\alpha},t_{12(\beta'_{7}-c)}s_{\alpha}\}$$
where $\beta'_{1}=\alpha_{1}, \beta'_{2}=\alpha_{2}, \beta'_{3}=\alpha_{3}, \beta'_{4}=\alpha_{1}+\alpha_{2}, \beta'_{5}=\alpha_{2}+\alpha_{3} , \beta'_{6}=\alpha_{1}+\alpha_{2}+\alpha_{3}$ and $\beta'_{7}=\alpha_{0}+\alpha_{1}+\alpha_{3}+\alpha_{4}$, see Example \ref{example17.93}. Moreover, for any $w\in \Gamma_{12c+\alpha}( id),$ we have $\ell(w)=2m(n-1)+\ell{(s_{\alpha})}=2m(n-1)+2\, \text{supp}(\alpha)-1=2\cdot12\cdot4+2\cdot2-1=99.$

\end{example}

Last, we will state the result for the most general case: the curve neighborhood of the identity element at any degree $\mathbf{d}$ which is strictly bigger than $c$. The proof can be found in Theorem  \ref{thm20}.

\begin{thm} \label{thm0.62} Let $\mathbf{d}=(d_0,d_1,...,d_{n-1})>c$ be a degree and $m=\text{min}\{d_0,d_1,...,d_{n-1}\}$. Also, assume that $z_{\mathbf{d}-mc}=s_{\gamma_{1}} s_{\gamma_{2}}...s_{\gamma_{k}}$ for some k, where this expression is obtained in Theorem \ref{thm12.7}. Then 

\begin{enumerate}

\item[1)] $\Gamma_{\mathbf{d}}( id)=\{t_{m\beta'}z_{\mathbf{d}-mc}: \beta' \in \Pi_{\text{aff}}^{\text{re},+}(\mathbf{d}-mc)\}\cup \{t_{m(\beta'-c)}z_{\mathbf{d}-mc}: \beta' \in \Pi_{\text{aff}}^{\text{re},+}(d-mc)\}$

\item[2)] $|\Gamma_{\mathbf{d}}( id)|=|\{\beta': \beta' \in \Pi_{\text{aff}}^{\text{re},\,+}(\mathbf{d}-mc)\}|$

\item[3)] For all $w\in \Gamma_{\mathbf{d}}(id)$, $\ell(w)=2m(n-1)+\ell(z_{\mathbf{d}-mc})$ 
\end{enumerate}
where $ \Pi_{\text{aff}}^{\text{re},\,+}(\mathbf{d}-mc)$ is the set of $\beta' \in \Pi_{\text{aff}}^{\text{re},\,+}$ such that $ \beta'<c$ and either  $\beta' \cap \gamma_{i}=\emptyset $ or both $\beta'>\gamma_{i} $ and $\beta'\perp \gamma_{i}$ for any $ i=1,...,k$. 

\end{thm}

\begin{example}  Let $W_{\text{aff}}$ be the affine Weyl group associated to $A_{3}^{(1)}$ and $\mathbf{d}=(6,5,8,5).$ Then $\mathbf{d}=5c+(1,0,3,0)$ so $m=5$ and $\mathbf{d}-5c=(1,0,3,0)$. Also, $z_{\mathbf{d}-5c}=s_{\alpha_{0}}\cdot s_{\alpha_{2}}\cdot s_{\alpha_{2}}\cdot s_{\alpha_{2}}=s_{\alpha_{0}}s_{\alpha_{2}}.$ Hence by Theorem \ref{thm0.62} we get 
\begin{equation*} \displaystyle \Gamma_{\mathbf{d}}( id)= \{t_{5\beta'_{1}} z_{\mathbf{d}-5c}, t_{5\beta'_{2}}z_{\mathbf{d}-5c},t_{5\beta'_{4}} z_{\mathbf{d}-5c},t_{5(\beta'_{3}-c)} z_{\mathbf{d}-5c} \}
\end{equation*}
where $\beta'_{1}=\alpha_{1}$, $\beta'_{2}=\alpha_{3}$, $\beta'_{3}=\alpha_{0}+\alpha_{1}+\alpha_{3}$, $\beta'_{4}=\alpha_{1}+\alpha_{2}+\alpha_{3}$. Moreover, for all $w\in \Gamma_{\mathbf{d}}( id)$, $\ell(w)=2m(n-1)+\ell(z_{\mathbf{d}-mc})=2\cdot 5\cdot3+2=32$.

\end{example}

\section{Preliminaries}

In this chapter we will set up notations and recall some basic facts about the affine root system in type $A_{n-1}^{(1)}$. We refer to \cite{Bourbaki}, \cite[\S 1]{kumar}, \cite[\S 3]{LS01}, and chapters $1$ through $8$ in \cite{kac1994infinite} for further details. 

\subsection{Weyl Group}
Let $G$ be the Lie group $SL_{n}(\C)$. Let $B\subset G$ be the Borel subgroup, the set of upper triangular matrices, and $T\subset B$ be the maximal torus, the set of diagonal matrices. We denote by $\mathfrak{g}$ the Lie algebra $\mathfrak{sl}_{n}(\C)$ of $G$.  Let $E$ be the subspace of $\R^{n}$ which consists of $n$-tuples for which the coordinates sum to $0.$ Let $\varepsilon_{1},\varepsilon_{2},...,\varepsilon_{n}$ be the standard basis for $\R^{n}$ and let $\alpha_{i}:=\varepsilon_{i}-\varepsilon_{i+1}$ for $1\leq i\leq n-1$. Let $\Delta=\{\alpha_{1},\alpha_{2},...,\alpha_{n-1}\}\subset \mathfrak{h^{*}} $ be the simple roots and $\Delta^{\vee}=\{\alpha_{1}^{\vee},\alpha_{2}^{\vee},...,\alpha_{n-1}^{\vee}\}\subset \mathfrak{h}$ be the corresponding coroot simple roots, where $\mathfrak{h}$ is the Cartan subalgebra of $\mathfrak{g}$. Let $Q=\bigoplus_{i=1}^{n-1}\mathbb{Z}\alpha_{i}\subset \mathfrak{h}^{*}$ and $Q^{\vee}=\bigoplus_{i=1}^{n-1}\mathbb{Z}\alpha_{i}^{\vee}\subset \mathfrak{h}$ be the root and coroot lattice. Let $\Pi=\{\varepsilon_{i}-\varepsilon_{j}: i\neq j\}$ be the standard root system associated to the triple $(T,B,G)$ corresponding to $\Delta,$ where $\Pi=\Pi^{+}\sqcup \Pi^{-}$ such that  $\Pi^{+}=\Pi \cap \bigoplus_{i=1}^{n-1}\mathbb{Z}_{\geq 0}\,\alpha_{i}=\{\varepsilon_{i}-\varepsilon_{j}: i<j\}$ is the set of positive roots and $\Pi^{-}=\Pi \cap \bigoplus_{i=1}^{n-1}\mathbb{Z}_{\leq 0}\,\alpha_{i}=\{\varepsilon_{i}-\varepsilon_{j}: i>j\}$ is the set of negative roots. Due to the fact that $\alpha_{i}=\alpha_{i}^{\vee}$ for all $i$ in type $A$, we may identify $Q$ with $Q^{\vee},$ and $ \mathfrak{h^{*}} $ with $\mathfrak{h}.$ Let $\langle,\rangle: \mathfrak{h}^{*}\times \mathfrak{h}\rightarrow \C$ be the natural pairing. Denote by $W$ the Weyl group. The Weyl group $W$ is generated by the simple reflections, $s_1,s_2,...,s_{n-1},$ where $s_{i}:=s_{\alpha_{i}}$, $i=1,2,...,n-1.$ For $w\in W,$ the \textit{length} of $w$, denoted $\ell{(w)},$ is the smallest positive integer $k$ such that $w=s_{i_1}s_{i_2}...s_{i_k}$, $1\leq i_{j}\leq n-1.$ Such an expression is called a \textit{reduced expression} of $w.$ We set $\ell{(id)}=0.$ $W$ acts on $ \mathfrak{h^{*}} $ by 
$$s_{i}(\mu)=\mu-\langle\mu,\alpha_{i}^{\vee}\rangle\alpha_{i} \,\,\,\,\,\,\text{for}\,\,\mu \in  \mathfrak{h}^{*}$$
For all $w \in W, \mu, \lambda \in  \mathfrak{h^{*}}$, we have $\langle w\cdot \mu,w\cdot \lambda  \rangle =\langle\mu,\lambda  \rangle .$  For each $\alpha \in \Pi$ there is a $w\in W$ and $1\leq i\leq n-1$ such that $\alpha=w\cdot \alpha_{i}$. The reflection of $\alpha$ is given by $s_{\alpha}=ws_{i}w^{-1}$ which is independent  of the choice of $w$ and $i.$ Moreover, it is well-known that $\ell{(w)}=|\{\alpha \in  \Pi^{+}:w\cdot \alpha<0\}|$ for all $w \in W$.
We can identify $W$ with the symmetric group $S_{n}$ via the map $s_{i}\mapsto (i,i+1)$.

\subsection{Affine Weyl Group}\label{affine Weyl group}

Let $\mathfrak{g}_{\text{aff}}$ be the affine Kac-Moody Lie algebra associated to the Lie algebra $\mathfrak{g}.$ We have $\mathfrak{g}_{\text{aff}}=\mathcal{L}(\mathfrak{g})\oplus \C c\oplus \C d,$ where $\mathcal{L}(\mathfrak{g}):=\mathfrak{g}\otimes \C[t,t^{-1}]\,(t\in \C^{*})$ is the space of all Laurent polynomials in $\mathfrak{g}$ and $c$ is a central element with respect to the Lie bracket in $\mathfrak{g_{\text{aff}}}.$ We denote by $\mathfrak{h}_{\text{aff}}$ the Cartan subalgebra of $\mathfrak{g_{\text{aff}}}$ which is given by $\mathfrak{h}_{\text{aff}}:=\mathfrak{h}\oplus \C c\oplus \C d.$ Let $\langle,  \rangle : \mathfrak{h}_{\text{aff}}^{*}\times \mathfrak{h}_{\text{aff}}\rightarrow \C$ be the natural pairing. The affine root system $\Pi_{\text{aff}}$ associated to $\mathfrak{g}_{\text{aff}}$ consists of 

\begin{itemize}

\item $m\delta+\alpha,$ where $\alpha \in \Pi$ and $m \in \Z$, which are called the \textit{affine real roots.}

\item $m\delta$, $m\in \Z \setminus \{0\}$, which are called the \textit{imaginary roots.}

\end{itemize}

For any root $\alpha \in \Pi \subset \Pi_{\text{aff}}$ we identify $\alpha$ as a linear function on $\mathfrak{h}^{*}$ whose restriction to $\mathfrak{h}$ is $\alpha$ and $\langle\alpha,c  \rangle =\langle \alpha,d  \rangle =0$ and the imaginary root $\delta\in \mathfrak{h}^{*}$ is defined by $\delta |\mathfrak{h}\oplus \C c=0$ and $\langle \delta,d  \rangle =1$. Let $\Delta_{\text{aff}}=\{\alpha_{0}:=\delta-\theta,\alpha_{1},...,\alpha_{n-1}\}\subset \Pi_{\text{aff}}$ be the affine simple roots where $\theta=\alpha_{1}+...+\alpha_{n-1} \in \Pi$ is the highest root. This determines a partition of $\Pi_{\text{aff}}$ into positive and negative affine roots: $\Pi_{\text{aff}}=\Pi_{\text{aff}}^{+}\sqcup \Pi_{\text{aff}}^{-}$ where  $\Pi_{\text{aff}}^{-}=-\Pi_{\text{aff}}^{+}$ and $\Pi_{\text{aff}}^{+}$ consists of the elements $m\delta+\alpha$ such that either $m>0$ or both $\alpha \in \Pi^{+}$ and $m=0.$ Denote by $\Pi_{\text{aff}}^{\text{re},\,+}$ the set of affine positive real roots. We denote by $Q_{\text{aff}}=\bigoplus_{i=0}^{n-1}\mathbb{Z}\alpha_{i}\subset \mathfrak{h}_{\text{aff}}^{*}$ and $Q_{\text{aff}}^{\vee}=\bigoplus_{i=0}^{n-1}\mathbb{Z}\alpha_{i}^{\vee}\subset \mathfrak{h}_{\text{aff}}$ the affine root and coroot lattices. For $a,b \in \bigoplus_{i=0}^{n-1}\mathbb{Z}\alpha_{i}$ we say that $a\leq b$ if $b-a$ is a linear combination of non-negative coefficients of $\alpha_{0},\alpha_{1},...,\alpha_{n-1}.$ Denote by $a<b$ if $a\leq b$ and $a\neq b.$ Furthermore, in type A we identify $\delta$ with $c$ so we set $c=\alpha_{0}+\theta=\sum_{i=0}^{n-1}\alpha_{i}.$ 

Let $W_{\text{aff}}=W\ltimes Q^{\vee}$ be the affine Weyl group corresponding to $W$. We denote by $t_{\mu}$ the image of $\mu \in Q^{\vee}$ in $W_{\text{aff}}.$ For all $w \in W$ and $\mu \in Q^{\vee},$ we have $t_{w\cdot \mu}=wt_{\mu}w^{-1}.$ $W_{\text{aff}}$ is generated by the simple reflections, $s_0,s_1,...,s_{n-1},$ where $s_{i}:=s_{\alpha_{i}}$, $i=0,1,...,n-1$ and 
$$s_{i}(\mu)=\mu-\langle \mu,\alpha_{i}^{\vee}  \rangle \, \alpha_{i} \,\,\,\,\text{for}\,\,\mu \in  \mathfrak{h}_{\text{aff}}^{*}\,\,\text{and}\, 0\leq i\leq n-1.$$ 
The affine root system $\Pi_{\text{aff}}$ is $W_{\text{aff}}$-invariant. For an affine real root $\alpha,$ we have $\alpha=w\cdot \alpha_{i}$ for some $w\in W_{\text{aff}}$ and $0\leq i\leq n-1.$ Then the associated reflection $s_{\alpha}$ is independent of the choice of $w$ and $i$. Also, $s_{\alpha}=ws_{i}w^{-1}$ and $s_{\alpha}(\mu)=\mu-\langle\mu,\alpha^{\vee}  \rangle \, \alpha \,\,\,\,\text{for}\,\,\mu \in  \mathfrak{h}_{\text{aff}}^{*}$. For an affine real root $\beta=\alpha+m\delta$ we have $s_{\beta}=s_{\alpha}t_{m\alpha^{\vee}}=s_{\alpha}t_{m\alpha}$ and, in particular $s_{0}=s_{\theta}t_{-\theta^{\vee}}=s_{\theta}t_{-\theta},$ see \cite[\S 3]{LS01} for further details.

Let $S=\{ s_{0},s_{1},s_{2},...,s_{n-1}\}.$ Then $(W_{\text{aff}},S)$ is a Coxeter system with the following relations;
$$s_{i}^{2}=1\,\,\,\,\,\,\text{for all}\,\,i$$
$$s_{i}s_{i+1}s_{i}=s_{i+1}s_{i}s_{i+1} \,\,\,\,\,\,\text{for all}\,\,i$$
$$s_{i}s_{j}=s_{j}s_{i} \,\,\,\,\,\,\text{for}\,\,i,j\,\,\text{not adjacent,}\,\,i\neq j$$
where the indices are taken modulo $n;$ see \cite[p.263]{BB01}. Each $w\neq e$ in $W_{\text{aff}}$ can be written in the form $w=s_{i_{1}}s_{i_{2}}...s_{i_{k}}$ for some $s_{i_{j}}$ (not necessarily distinct) in $S.$ If $k$ is as small as possible, call it the \textit{length} of $w$, written $l(w),$ and call any expression of $w$ as a product of $k$ elements of $S$ a \textit{reduced expression.} We set $\ell{(e)}=0.$ For an $\alpha \in \Pi_{\text{aff}}$ if $\alpha=\sum_{i=0}^{n-1} a_{i}\alpha_{i}$ then the support of $\alpha$ is $\text{supp}(\alpha)=\{\alpha_{i}: a_{i}\neq 0 \}$. We have the following equations:

\begin{itemize}

\item If $\alpha \in \Pi_{\text{aff}}^{\text{re},\,+}$ such that $\alpha <c$ then $\ell{(s_{\alpha})}=2|\text{supp}(\alpha)|-1$.
\item If $w \in W_{\text{aff}}$ then $\ell{(w)}=|\{\alpha \in  \Pi_{\text{aff}}^{+}:w\cdot \alpha<0\}|.$

\item If $x=wt_{\lambda}\in W_{\text{aff}}$ then by Lemma $3.1$ in \cite{LS01}, we have 
\begin{equation}\label{equation5.5} \ell{(x)}=\sum_{\gamma \in \Pi^{+}} |\chi{(w\cdot \gamma <0)}+\langle\lambda,\gamma  \rangle |
\end{equation}
where $\chi{(P)}=1$ if $P$ is true and $\chi{(P)}=0$ otherwise. 

\end{itemize}
Given $u,v \in W_{\text{aff}}$, we say that the product $uv$ is \textit{reduced} if $\ell{(uv)}=\ell{(u)}+\ell{(v)}.$ 

Next,  we will set up some notations for affine positive real roots which are smaller than $c$; let $p_{i,i}:=\alpha_{i}$ for $0\leq i \leq n-1$, and $p_{i,j}:=\alpha_{i}+\alpha_{i+1}+...+\alpha_{j-1}+\alpha_{j}$ for $0\leq i<j\leq n-1$ such that $j-i<n-1$ and $p_{i,j}:=\alpha_{0}+\alpha_{1}+...+\alpha_{j}+\alpha_{i}+\alpha_{i+1}+...+\alpha_{n-1}$ where $0\leq j<i-1\leq n-2.$ Note that, $s_{p_{i,i}}=s_{i}$ and $s_{p_{i,j}}=s_{i}s_{i+1}...s_{j-1}s_{j}s_{j-1}...s_{i-1}s_{i}=s_{j}s_{j-1}...s_{i+1}s_{i}s_{i+1}...s_{j-1}s_{j}$ if $i<j,$ and 
$$\begin{array}{clclc}s_{p_{i,j}}&=&s_{j}s_{j-1}...s_{1}s_{0}s_{i}s_{i+1}...s_{n-2}s_{n-1}s_{n-2}...s_{i+1}s_{i}s_{0}s_{1}...s_{j-1}s_{j}\\
&=&s_{i}s_{j}s_{j-1}...s_{1}s_{0}s_{i+1}...s_{n-2}s_{n-1}s_{n-2}...s_{i+1}s_{0}s_{1}...s_{j-1}s_{j}s_{i}\\
\end{array}$$
if $i>j$, are some of the reduced expressions for the reflections which will be used throughout this paper.

\subsection{Dynkin Diagram}\label{dynkin}
In this section, we will consider the Dynkin diagram for $\Pi_{\text{aff}}$ which is given in the figure below.
\medskip
\begin{figure}[h!] \label{fig2}
\begin{center}
\includegraphics[width=120mm,scale=0.5]{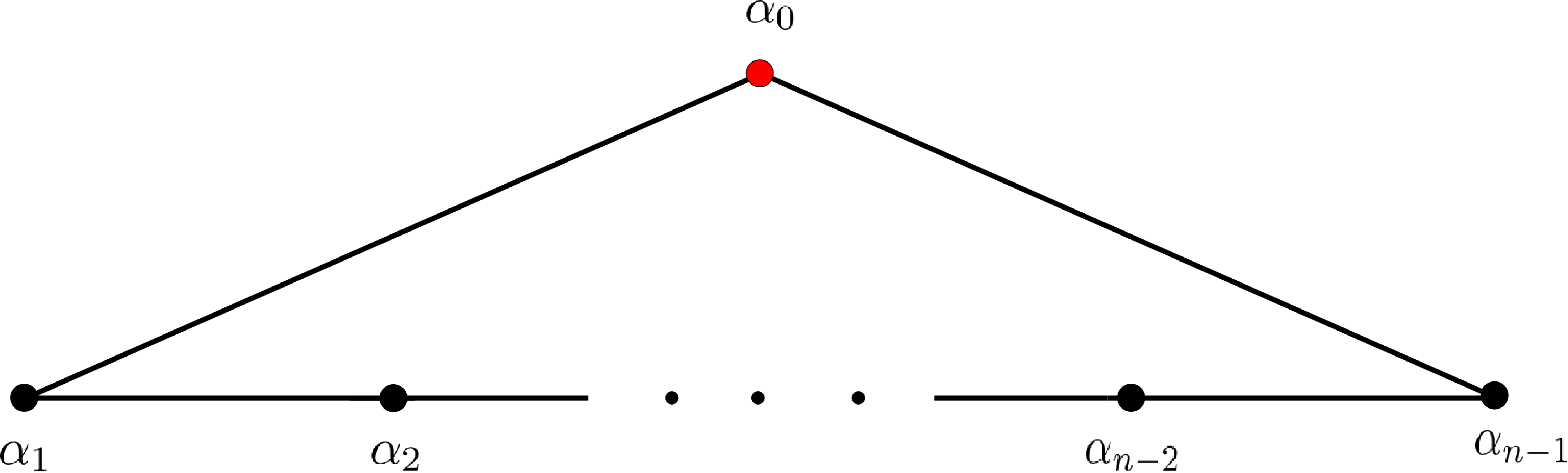}
\end{center}
\caption{Dynkin diagram of affine root system in type $A_{n-1}^{(1)}$.}
\end{figure}

Let $\varphi: \Delta_{\text{aff}}\rightarrow \Delta_{\text{aff}}$ be a map given by $\varphi(\alpha_{i})=\alpha_{i+1}$ if $0\leq i \leq n-2$ and $\varphi(\alpha_{n-1})=\alpha_{0}$. Note that, $\langle \alpha_{i},\alpha_{j}^{\vee}  \rangle =\langle\varphi(\alpha_{i}),\varphi(\alpha_{j})^{\vee}  \rangle $ for all $i,j,$ since the Dynkin diagram for $\Pi_{\text{aff}}$ is simply laced and circular. Thus $\varphi$ is an automorphism of the diagram. It is clear that the order of $\varphi$ is $n$ and $\varphi$ preserves the partial order "$\leq$" on the roots.

\subsection{Hecke Product} \label{hecke product}
We describe some of the curve neighborhoods in terms of the Hecke product. For that reason, we recall the definition of the Hecke product and some of its properties in this section. We refer to \cite[\S{3}]{MB01} for further details. For $u \in W_{\text{aff}}$ and $i\in \{0,1,...,n-1\}$, define
\begin{displaymath}
u\cdot s_{i} = \left\{ \begin{array}{ll}
 us_{i} & \textrm{If $\ell{(us_{i})}>\ell{(u)}$}\\
 u & \textrm{otherwise.}\\
  \end{array} \right.
\end{displaymath}
Let $u,v \in W_{\text{aff}}$ and let $v=s_{i_{1}}s_{i_{2}}...s_{i_{k}}$ be any reduced expression for $v.$ Define the \textit{Hecke product} of $u$ and $v$ by
$$u\cdot v=u\cdot s_{i_{1}}\cdot s_{i_{2}}\cdot ...\cdot s_{i_{k}},$$
where the simple reflections are multiplied to $u$ in left to right order. This product is independent of the chosen reduced expressions for $v.$ It provides a monoid structure on the affine Weyl group $W_{\text{aff}}.$ Furthermore, we have the following properties of the Hecke product: Let $u,v,v',w \in W_{\text{aff}}.$

\begin{itemize}

\item[a)] The Hecke product is associative, i.e. $(u\cdot v)\cdot w=u\cdot (v\cdot w).$

\item[b)] If $v\leq v'$ then $u\cdot v \cdot w\leq u\cdot v'\cdot w.$

\item[c)] We have $u\leq u\cdot v,$ $v\leq u\cdot v,$ $uv\leq u\cdot v,$ and $\ell{(u\cdot v)}\leq \ell{(u)}+\ell{(v)}.$

\item[d)] If $uv$ is reduced then $uv=u\cdot v$ and $\ell{(u\cdot v)}= \ell{(u)}+\ell{(v)}.$

\item[e)] The element $v=(w\cdot u)u^{-1}$ satisfies $v\leq w$ and $vu=v\cdot u=w\cdot u.$

\end{itemize}

\section{Preliminary Lemmas and $z_{\mathbf{d}}$}

In this section we give the definition of $z_{\mathbf{d}}$ which is the unique element in $\Gamma_{\mathbf{d}}(id)$ where $\mathbf{d}=(d_0,d_1,...,d_{n-1})\in Q_{\text{aff}}$ such that $d_{i}=0$ for some $i.$ Furthermore, the most general case of the curve neighborhood $\Gamma_{\mathbf{d}}(id)$ where $\mathbf{d}>c$ is described in terms of $z_{\mathbf{d}}$. This section also includes a theorem where we show how one can simplify $z_{\mathbf{d}}$. We refer the reader to \cite[\S 4]{MB01} and \cite{Barligea} for more details.

\begin{defn}\label{def12.5} Let $\mathbf{d}=(d_0,d_1,...,d_{n-1})\in Q_{\text{aff}}$ such that $d_{i}=0$ for some $i.$ If $\mathbf{d}=0$ then set $z_{\mathbf{d}}=id$. Otherwise we set $z_{\mathbf{d}}=s_{\alpha}\cdot z_{\mathbf{d}-\alpha}$ where $\alpha$ is any maximal root which is smaller than and equal to $\mathbf{d}.$
\end{defn}

Moreover, $z_{\mathbf{d}}$ is well-defined by induction on $\mathbf{d}$.

\begin{thm}\label{thm12.7} Let $\mathbf{d}=(d_0,d_1,...,d_{n-1})$ be a degree such that $d_{i}=0$ for some $i.$ Then $z_{\mathbf{d}}=s_{\gamma_{1}}s_{\gamma_{2}}...s_{\gamma_{r}}$ for some positive real roots, $\gamma_1,\gamma_2,...,\gamma_{r}$ such that either $\text{supp}(\gamma_{i})$ and $\text{supp}(\gamma_{j})$ are disconnected or both $\gamma_{i}\perp \gamma_{j}$ and $\gamma_{i},\gamma_{j}$ are comparable, for any $1\leq i,j\leq r$ such that $i\neq j$. Moreover, $z_{\mathbf{d}}=s_{\gamma_{1}}s_{\gamma_{2}}...s_{\gamma_{r}}$ is reduced and $\ell{(z_{\mathbf{d}})}=\sum_{i=1}^{r}(2\,|\text{supp}(\gamma_{i})|-1).$

\end{thm}

\begin{proof} Let $z_{\mathbf{d}}=s_{\beta_{1}}\cdot s_{\beta_{2}}\cdot ...\cdot s_{\beta_{r}}$ for some positive real roots,  $\beta_1,\beta_2,...,\beta_{r}$. By definition, we have either $\text{supp}(\beta_{i})$ and $\text{supp}(\beta_{j})$ are disconnected or $\beta_{i},\beta_{j}$ are comparable, for any $1\leq i,j\leq r$ such that $i\neq j$, see \cite[\S{4}]{MB01}. So $s_{\beta_{i}}\cdot s_{\beta_{j}}=s_{\beta_{j}}\cdot s_{\beta_{i}}$ for any $1\leq i,j\leq r$, by Lemma \ref{lemma5.7}.  Now, assume that $\beta_{i}$ and $\beta_{j}$ are comparable but not perpendicular for some $1\leq i,j\leq r$ such that $i\neq j$. Here, we can suppose that $\beta_{i}\geq \beta_{j}$, without loss of generality. Now, if $\beta_{i}=p_{k,l}$ for some $0\leq k,l \leq n-1$ then we have three cases; either $\beta_{j}=p_{k,q}$ where $q\neq l$ or $\beta_{j}=p_{q,l}$ where $q\neq k$ or $\beta_{j}=p_{k,l}.$ First, suppose that $\beta_{j}=p_{k,q}$ where $q\neq l$. Note that, both $s_{\beta_{i}}$ and $s_{\beta_{j}}$ have a reduced expression which starts and ends with the simple reflection, $s_{k}$ this implies that $s_{\beta_{i}}\cdot s_{k}=s_{\beta_{i}}$. We will show that $s_{\beta_{i}}\cdot s_{\beta_{j}}=s_{\beta_{i}}\cdot s_{\beta_{j}-\alpha_{k}}$ where $\alpha_{k}$ is the simple root. Here, we will use the convention, $s_{a}:=id$ if $a=0.$ Now, observe that $s_{\beta_{j}}=s_{k}\cdot s_{\beta_{j}-\alpha_{k}}\cdot s_{k}$.  Furthermore, $s_{\beta_{i}}\cdot s_{\beta_{j}-\alpha_{k}}=s_{\beta_{j}-\alpha_{k}}\cdot s_{\beta_{i}}$ since $\beta_{j}-\alpha_{k}<\beta_{i}$ by Lemma \ref{lemma5.7}. Thus 
$$\begin{array}{lll} s_{\beta_{i}}\cdot s_{\beta_{j}}&=&s_{\beta_{i}}\cdot s_{k}\cdot s_{\beta_{j}-\alpha_{k}}\cdot s_{k}=s_{\beta_{i}}\cdot s_{\beta_{j}-\alpha_{k}}\cdot s_{k}=s_{\beta_{j}-\alpha_{k}}\cdot s_{\beta_{i}}\cdot s_{k}\\
&=&s_{\beta_{j}-\alpha_{k}}\cdot s_{\beta_{i}}=s_{\beta_{i}}\cdot s_{\beta_{j}-\alpha_{k}}.\\
\end{array}$$ 
Here, notice that $\beta_{i}$ and $\beta_{j}-\alpha_{k}$ are perpendicular. Second, suppose that $\beta_{j}=p_{q,l}$ where $q\neq k$. Then, one can show that $s_{\beta_{i}}\cdot s_{\beta_{j}}=s_{\beta_{i}}\cdot s_{\beta_{j}-\alpha_{l}}$ where $\alpha_{l}$ is the simple root, by using similar arguments in the previous case. Also, note that $\beta_{i}$ and $\beta_{j}-\alpha_{l}$ are perpendicular. Last, assume that $\beta_{j}=p_{k,l}.$ Then again by the previous two cases one can show that $s_{\beta_{i}}\cdot s_{\beta_{j}}=s_{\beta_{i}}\cdot s_{\beta_{j}-\alpha_{k}-\alpha_{l}}$ where $\alpha_{k}$ and $\alpha_{l}$ are the simple reflections. Moreover, $\beta_{i}$ and $\beta_{j}-\alpha_{k}-\alpha_{l}$ are perpendicular. Hence, we can assume that either $\text{supp}(\beta_{i})$ and $\text{supp}(\beta_{j})$ are disconnected or both $\beta_{i},\beta_{j}$ are comparable and $\beta_{i}\perp \beta_{j}$, for any $1\leq i,j\leq r$ such that $i\neq j$. 

Next, we will show that $z_{\mathbf{d}}=s_{\beta_{1}} s_{\beta_{2}} ...s_{\beta_{r}}$ is reduced. Note that $\beta_{i}<c$ for all $i$, by definition. So $\ell{(s_{\beta_{i}})}=2|\text{supp}(\beta_{i})|-1$ and we have $\ell{(z_{\mathbf{d}})}=\ell{(s_{\beta_{1}} s_{\beta_{2}} ...s_{\beta_{r}})}\leq \sum_{i=1}^{r}\ell{(s_{\beta_{i}})}=\sum_{i=1}^{r} 2|\text{supp}(\beta_{i})|-1.$ Furthermore, for a root $\nu,$ we have $s_{\beta_{1}} s_{\beta_{2}} ...s_{\beta_{r}}(\nu)=\nu-\sum_{i=1}^{r}<\nu,\beta_{i}^{\vee}>\beta_{i}$ since $\beta_{i}\perp \beta_{j},$ for any $1\leq i,j\leq r$ such that $i\neq j$. Now, assume that $\nu$ is an affine positive real root such that $\nu\leq \beta_{i}$ for some $i$. Also, assume that $\nu$ and $\beta_{i}$ are not perpendicular. This implies that $<\nu,\beta_{i}^{\vee}>$ is either equal to $2$ or $1$. if $<\nu,\beta_{i}^{\vee}>=2$ then $\beta_{i}=\nu.$ But then $\nu \perp \beta_{j}$ for all $j\in \{1,2,...,r\} \setminus \{i\},$ which follows by $s_{\beta_{1}} s_{\beta_{2}} ...s_{\beta_{r}}(\nu)=-\nu<0$. Now, suppose that  $<\nu,\beta_{i}^{\vee}>=1$ then we have $\nu \perp \beta_{j}$ for all $j$ such that $\text{supp}(\beta_{i})$ and $\text{supp}(\beta_{j})$ are disconnected since $\nu\leq \beta_{i}.$ So if $\nu$ and $\beta_{j}$ are not perpendicular for some $\beta_{j}$ where $\beta_{j}\neq \beta_{i}$ we have to have $\beta_{j}<\beta_{i}$ which implies that $<\nu,\beta_{j}^{\vee}>$ is either equal to $1$ or $-1$ and there is at most one such root as $\beta_{j}.$ Now, if  $<\nu,\beta_{j}^{\vee}>=1$ then either $\beta_{j}<\nu$ or $\nu<\beta_{j}$ but since $\beta_{j}<\beta_{i}$ and $\beta_{j}\perp \beta_{i}$ we have to have $\beta_{j}<\nu$. If $<\nu,\beta_{j}^{\vee}>=-1$ then $\beta_{j}\cap \nu=\emptyset$. Thus we have either $s_{\beta_{1}}s_{\beta_{2}} ...s_{\beta_{r}}(\nu)=\nu-\beta_{i}- \beta_{j}$ or $s_{\beta_{1}}s_{\beta_{2}} ...s_{\beta_{r}}(\nu)=\nu-\beta_{i}+ \beta_{j}$. Here, observe that in either case $s_{\beta_{1}}s_{\beta_{2}} ...s_{\beta_{r}}(\nu)<0$. Now, let $\beta_{i}=p_{k,l}$ and $\nu=p_{t,q}$. Then $t=k$ or $q=l$. If $k\leq l$ then there are $2l-2k+1=2|\text{supp}(\beta_{i})|-1$ such $\nu$ and if $k>l$ then there are $2n+2l-2k+1=2|\text{supp}(\beta_{i})|-1$ such $\nu$. Hence, if $A$ is the set of all positive real roots, $\nu$ such that $s_{\beta_{1}}s_{\beta_{2}}... s_{\beta_{r}}(\nu)<0$ where $\nu\leq \beta_{i}$ and, $\nu$ and $\beta_{i}$ are not perpendicular for some $i=1,2,..,r$  then $|A|=\sum_{i=1}^{r}2|\text{supp}(\beta_{i})|-1.$ Also, note that $A\subseteq \{\gamma \in \Pi_{\text{aff}}^{\text{re},\,+}: s_{\beta_{1}} s_{\beta_{2}}...s_{\beta_{r}}(\gamma)<0\}$ and we have 
$$\begin{array}{lll} |A|=\sum_{i=1}^{r}2|\text{supp}(\beta_{i})|-1 &\leq& |\{\gamma \in \Pi_{\text{aff}}^{\text{re},\,+}: s_{\beta_{1}}s_{\beta_{2}} ...s_{\beta_{r}}(\nu)<0\}|\\
&=&\ell{(s_{\beta_{1}}s_{\beta_{2}} ...s_{\beta_{l}})}\leq \sum_{i=1}^{r}2|\text{supp}(\beta_{i})|-1\\
\end{array}$$
which follows by $\ell{(s_{\beta_{1}}s_{\beta_{2}} ... s_{\beta_{r}})}= \sum_{i=1}^{r}2|\text{supp}(\beta_{i})|-1$. Thus $s_{\beta_{1}}s_{\beta_{2}}...s_{\beta_{r}}$ is reduced.

\end{proof}

Next, we will give an example for computing $z_{\mathbf{d}}.$
\begin{example}\label{example12.01} Let $\mathbf{d}=(5,0,2,2,3,0,4).$ Observe that $\alpha_{0}+\alpha_{6}$ and $\alpha_{2}+\alpha_{3}+\alpha_{4}$ are maximal roots of $\mathbf{d}$ which have disconnected supports and $\mathbf{d}=4(\alpha_{0}+\alpha_{6})+2(\alpha_{2}+\alpha_{3}+\alpha_{4})+\alpha_{0}+\alpha_{4}$ so by definition we get
$z_{\mathbf{d}}=s_{\alpha_{0}+\alpha_{6}}\cdot s_{\alpha_{0}+\alpha_{6}}\cdot s_{\alpha_{0}+\alpha_{6}}\cdot s_{\alpha_{0}+\alpha_{6}}\cdot s_{\alpha_{2}+\alpha_{3}+\alpha_{4}}\cdot s_{\alpha_{2}+\alpha_{3}+\alpha_{4}}\cdot s_{\alpha_{0}}\cdot s_{\alpha_{4}}.$
But notice that $s_{\alpha_{0}+\alpha_{6}}\cdot s_{\alpha_{0}+\alpha_{6}}=s_{\alpha_{0}+\alpha_{6}}$ and $s_{\alpha_{2}+\alpha_{3}+\alpha_{4}}\cdot s_{\alpha_{2}+\alpha_{3}+\alpha_{4}}=s_{\alpha_{2}+\alpha_{3}+\alpha_{4}}\cdot s_{\alpha_{3}}$, see the proof of Lemma \ref{lemma5.7}, so we have $z_{\mathbf{d}}=s_{\alpha_{0}+\alpha_{6}} \cdot s_{\alpha_{2}+\alpha_{3}+\alpha_{4}}\cdot s_{\alpha_{3}}\cdot s_{\alpha_{0}}\cdot s_{\alpha_{4}}.$ Now, also note that any two reflections that appear in $z_{\mathbf{d}}$ Hecke commute by Lemma \ref{lemma5.7}. Moreover,  $s_{\alpha_{0}+\alpha_{6}} \cdot s_{\alpha_{0}}=s_{\alpha_{0}+\alpha_{6}} $ and $s_{\alpha_{2}+\alpha_{3}+\alpha_{4}}\cdot s_{\alpha_{4}}=s_{\alpha_{2}+\alpha_{3}+\alpha_{4}}$, again see the proof of Lemma \ref{lemma5.7}. Thus $z_{\mathbf{d}}=s_{\alpha_{0}+\alpha_{6}} \cdot s_{\alpha_{2}+\alpha_{3}+\alpha_{4}}\cdot s_{\alpha_{3}}.$ But this multiplication for $z_{\mathbf{d}}$ is reduced; see the proof of Theorem \ref{thm12.7}. So $z_{\mathbf{d}}=s_{\alpha_{0}+\alpha_{6}}  s_{\alpha_{2}+\alpha_{3}+\alpha_{4}} s_{\alpha_{3}}.$ Furthermore, 
$$\ell{(z_{\mathbf{d}})}=2|\text{supp}(\alpha_{0}+\alpha_{6})|-1+2|\text{supp}(\alpha_{2}+\alpha_{3}+\alpha_{4})|-1+ 2|\text{supp}(\alpha_{3})|-1=9.$$

\end{example}

\subsection{Lemma about the Hecke Product}

Here, we will consider a lemma which allows us to manipulate a given Hecke multiplication of two reflections without changing the sum of the corresponding roots. 

\begin{lemma} \label{lemma5.7}Let $\alpha, \beta \in{\Pi_{\text{aff}}^{\text{re},\,+}}$ such that $\alpha, \beta <c$. Then
\begin{itemize}
\item[1)] Assume that $\alpha\leq \beta.$ Then $s_{\alpha}\cdot s_{\beta}=s_{\beta}\cdot s_{\alpha}$ 
 
\item[2)] Suppose that $\gamma=\alpha \cap \beta \neq \emptyset$ and $\gamma \neq \alpha, \beta$. Then
\begin{itemize}
\item[a)] If $\gamma$ is a root then $\alpha-\gamma$ and $\beta-\gamma$ are also roots. Moreover,
 $$s_{\alpha}\cdot s_{\beta}=s_{\alpha}\cdot s_{\beta-\gamma}\cdot s_{\gamma}=s_{\gamma}\cdot s_{\alpha-\gamma}\cdot s_{\beta}$$ 

 \item[b)] If $\gamma$ is not a root then $\gamma=\gamma_{1}+\gamma_{2}$ for some positive real roots $\gamma_{1}$, $\gamma_{2}$. Also, $\alpha-\gamma$,$\alpha-\gamma_{1}$, $\alpha-\gamma_{2}$, $ \beta-\gamma$, $\beta-\gamma_{1}$ and $\beta-\gamma_{2}$ are all roots. Moreover, $$s_{\alpha}\cdot s_{\beta}=s_{\alpha}\cdot s_{\beta-\gamma}\cdot s_{\gamma_{1}}\cdot s_{\gamma_{2}}=s_{\gamma_{1}}\cdot s_{\gamma_{2}}\cdot s_{\alpha-\gamma}\cdot s_{\beta}$$
\end{itemize} 

\item[3)] Assume that $\alpha \cap \beta =\emptyset.$ Then 

\begin{itemize}

\item[a)] If $\alpha+\beta$ is a root then $s_{\alpha}\cdot s_{\beta}\leq s_{\alpha+\beta}.$

\item[b)] If $\alpha+\beta$ is not a root then $s_{\alpha}\cdot s_{\beta}=s_{\beta}\cdot s_{\alpha}.$
\end{itemize}

\end{itemize}

\end{lemma}

\begin{proof} Let  a simple reflection $s_{q}$ and a positive real root $p_{i,j}<c$ be given. We will show that $s_{q}\cdot  s_{p_{i,j}}=s_{p_{i,j}}\cdot s_{q}$ if $\alpha_{q}\in \text{supp}(p_{i,j}).$ First, assume that $q=i.$ Then we have three cases; either $i=j$ or $i<j$ or $i>j$. If $i=j$ then the equality is clear since $p_{i,j}=p_{i,i}=\alpha_{i}$ and $s_{q}=s_{p_{i,j}}=s_{i}$. Now, suppose that $i<j.$ Then by $s_{i}\cdot s_{i}=s_{i}$ we get
$$\begin{array}{lll}s_{q}\cdot s_{p_{i,j}}&=&s_{i}\cdot s_{i}\cdot s_{i+1}\cdot ...\cdot s_{j-1}\cdot s_{j}\cdot s_{j-1}\cdot ...\cdot s_{i-1}\cdot s_{i}\\ &=&s_{i}\cdot s_{i+1}\cdot ...\cdot s_{j-1}\cdot s_{j}\cdot s_{j-1}\cdot ...\cdot s_{i-1}\cdot s_{i}\cdot s_{i}=s_{p_{i,j}}\cdot s_{q}.\\\end{array}$$
The case, $i>j$ is similar. 
If $q=j$ then we have the equalities by the same argument since $s_{p_{i,j}}$ has a reduced expression which starts and ends with $s_{j}$ in all sub cases. Now assume that $q$ is neither $i$ nor $j.$ Here, we have two sub cases; $i<j$ or $i>j.$ First, assume that $i<j.$ Then $i<q<j$ and we have 
$$\begin{array}{lll}
s_{q}\cdot s_{p_{i,j}}&=&s_{q}\cdot s_{i}\cdot s_{i+1}\cdot ...\cdot (s_{q-1}\cdot s_{q})\cdot ...\cdot s_{j-1}\cdot s_{j}\cdot s_{j-1}\cdot ...\cdot (s_{q}\cdot s_{q-1})\cdot ...\cdot s_{i-1}\cdot s_{i} \\ &=&s_{i}\cdot s_{i+1}\cdot ...\cdot (s_{q}\cdot s_{q-1}\cdot s_{q})\cdot ...\cdot s_{j-1}\cdot s_{j}\cdot s_{j-1}\cdot ...\cdot (s_{q}\cdot s_{q-1})\cdot ...\cdot s_{i-1}\cdot s_{i}\\
&=& s_{i}\cdot s_{i+1}\cdot ...\cdot (s_{q-1}\cdot s_{q}\cdot s_{q-1})\cdot ...\cdot s_{j-1}\cdot s_{j}\cdot s_{j-1}\cdot ...\cdot (s_{q}\cdot s_{q-1})\cdot ...\cdot s_{i-1}\cdot s_{i}\\ &=& s_{i}\cdot s_{i+1}\cdot ...\cdot ( s_{q-1}\cdot s_{q})\cdot ...\cdot s_{j-1}\cdot s_{j}\cdot s_{j-1}\cdot ...\cdot (s_{q-1}\cdot s_{q}\cdot s_{q-1})\cdot ...\cdot s_{i-1}\cdot s_{i}\\&=& s_{i}\cdot s_{i+1}\cdot ...\cdot (s_{q-1}\cdot s_{q})\cdot ...\cdot s_{j-1}\cdot s_{j}\cdot s_{j-1}\cdot ...\cdot (s_{q}\cdot s_{q-1}\cdot s_{q})\cdot ...\cdot s_{i-1}\cdot s_{i} \\&=&s_{i}\cdot s_{i+1}\cdot ...\cdot (s_{q-1}\cdot s_{q})\cdot ...\cdot s_{j-1}\cdot s_{j}\cdot s_{j-1}\cdot ...\cdot (s_{q}\cdot s_{q-1})\cdot ...\cdot s_{i-1}\cdot s_{i}\cdot s_{q}\\&=&s_{p_{i,j}}\cdot s_{q}\\
\end{array}$$

The case, $i>j$ is similar.

$1)$ Assume that $\alpha \leq \beta$. Suppose that $\alpha=s_{i_{1}}\cdot s_{i_{2}}\cdot ...\cdot s_{i_{r}}$ is a reduced expression for $\alpha$. Now, since $\alpha \leq \beta$ we have  supp($\alpha)=\{i_{1},i_{2},...,i_{r}\} \subset$ supp($\beta)$. Thus the reflection $s_{i_{k}}$ will Hecke commute with $\beta$ for all $k$ by the argument above. So we will have 
$$\alpha \cdot \beta=s_{i_{1}}\cdot s_{i_{2}}\cdot ...\cdot s_{i_{r}}\cdot \beta=\beta \cdot s_{i_{1}}\cdot s_{i_{2}}\cdot ...\cdot s_{i_{r}}=\beta \cdot \alpha.$$

$2)$ \textit{Case a:} Let $\gamma=\alpha \cap \beta$ be a root and $\gamma \neq \alpha, \beta$. First, we will show that $s_{\alpha}=u\cdot s_{\alpha-\gamma}\cdot u^{-1}$ and $s_{\beta}=u^{-1}\cdot s_{\beta-\gamma}\cdot u$ for some permutation $u$ such that $\gamma=u \cdot u^{-1}.$  We have several cases here;

\textit{Case a1:} Assume that $\alpha=p_{i,j}$ and $\beta=p_{k,l}$ where $0\leq i<k<j<l\leq n-1.$ Then $\gamma=p_{k,j}.$ Furthermore, $\alpha-\gamma=p_{i,k-1}$ and $\beta-\gamma=p_{j+1,l},$ so both are roots. Now, note that 
$$s_{\alpha}=s_{j}\cdot s_{j-1}\cdot...\cdot s_{k+1}\cdot s_{k}\cdot s_{k-1}\cdot ...\cdot s_{i+1}\cdot s_{i}\cdot s_{i+1}\cdot ...\cdot s_{k-1}\cdot s_{k}\cdot s_{k+1}\cdot...\cdot s_{j-1}\cdot s_{j},$$
 $$s_{\beta}=s_{k}\cdot s_{k+1}\cdot ...\cdot s_{j-1}\cdot s_{j}\cdot s_{j+1}\cdot ...\cdot s_{l-1}\cdot s_{l}\cdot s_{l-1}\cdot ...s_{j+1}\cdot s_{j}\cdot s_{j-1}\cdot...\cdot s_{k+1}\cdot s_{k}, $$
 $s_{\gamma}=s_{j}\cdot s_{j-1}\cdot...\cdot s_{k+1}\cdot s_{k}\cdot s_{k+1}\cdot...\cdot s_{j-1}\cdot s_{j}$ and $s_{\alpha-\gamma}=s_{k-1}\cdot ...\cdot s_{i+1}\cdot s_{i}\cdot s_{i+1}\cdot ...\cdot s_{k-1}$ and $s_{\beta-\gamma}=s_{j+1}\cdot ...\cdot s_{l-1}\cdot s_{l}\cdot s_{l-1}\cdot ...s_{j+1}$ are some reduced expressions for the reflections. So we can take $u=s_{j}\cdot s_{j-1}\cdot...\cdot s_{k+1}\cdot s_{k}$ which will imply that $s_{\gamma}=u\cdot u^{-1}$ since $s_{k}\cdot s_{k}=s_{k}.$

\textit{Case a2:} Suppose that $\alpha=p_{i,j}$ and $\beta=p_{k,l}$ where $0\leq k<j<l\leq i\leq n-1.$ Then $\gamma=p_{k,j}.$ Moreover, $\alpha-\gamma=p_{i,n-1}$ and $\beta-\gamma=p_{j+1,l},$ hence both are roots. This case is similar to $a1).$

\textit{Case a3:} Suppose that $\alpha=p_{i,j}$ and $\beta=p_{k,l}$ where $0\leq j\leq k<i<l \leq n-1.$ Then $\gamma=p_{i,l}.$ Also, $\alpha-\gamma=p_{0,j}$ and $\beta-\gamma=p_{k,i-1},$ thus both are roots. Similar to $a1).$

\textit{Case a4:} Suppose that $\alpha=p_{i,j}$ and $\beta=p_{l,k}$ where $0\leq k<j\leq l < i\leq n-1.$ Then $\gamma=p_{i,k}.$ Moreover, $\alpha-\gamma=p_{k+1,j}$ and $\beta-\gamma=p_{l,i-1},$ so both are roots. Note that, 
$$\begin{array}{lll}
s_{\alpha}&=&s_{j}\cdot ...\cdot s_{k}\cdot ...\cdot s_{1}\cdot s_{0}\cdot s_{i}\cdot s_{i+1}\cdot...\cdot s_{n-1}\cdot ...\cdot s_{i-1}\cdot s_{i}\cdot s_{0}\cdot s_{1}\cdot...\cdot s_{k}\cdot...\cdot s_{j}\\
&=&s_{i}\cdot s_{i+1}\cdot...\cdot s_{n-1}\cdot s_{0}\cdot s_{1}\cdot  ...\cdot s_{k}\cdot ...\cdot s_{j}\cdot ...\cdot s_{k}\cdot ...\cdot s_{1}\cdot s_{0}\cdot s_{n-1}\cdot ...\cdot s_{i-1}\cdot s_{i}, \\
\end{array}$$
$$\begin{array}{lll}
s_{\beta}&=&s_{k}\cdot  ...\cdot s_{1}\cdot s_{0}\cdot s_{l}\cdot s_{l+1}\cdot ...\cdot s_{i}\cdot...s_{n-1}\cdot ...\cdot s_{i}\cdot ...\cdot s_{l+1}\cdot s_{l}\cdot s_{0}\cdot s_{1}\cdot ...\cdot s_{k}\\
&=& s_{k}\cdot  ...\cdot s_{0}\cdot s_{n-1}\cdot ...s_{i}\cdot s_{i-1}\cdot  ...\cdot s_{l+1}\cdot s_{l}\cdot s_{l+1}\cdot ...\cdot s_{i-1}\cdot s_{i}\cdot ...\cdot s_{n-1}\cdot s_{0} \cdot ...\cdot s_{k},\\
\end{array}$$
$$\begin{array}{lll} s_{\gamma}&=&s_{k}\cdot ..\cdot s_{1}\cdot s_{0}\cdot s_{i}\cdot s_{i+1}\cdot...\cdot s_{n-1}\cdot ...\cdot s_{i+1}\cdot s_{i}\cdot s_{0}\cdot s_{1}\cdot...\cdot s_{k}\\
&=&s_{i}\cdot s_{i+1}\cdot...\cdot s_{n-1}\cdot s_{0}\cdot s_{1}\cdot ...\cdot s_{k}\cdot ...\cdot s_{1}\cdot s_{0}\cdot s_{n-1}\cdot ...\cdot s_{i+1}\cdot s_{i},\\
\end{array}$$
and $s_{\alpha-\gamma}=s_{p_{k+1,j}}=s_{k+1}\cdot s_{k+2}\cdot ...\cdot s_{j}\cdot ..\cdot s_{k+2}\cdot s_{k+1},$ and 
 $$\begin{array}{lll}s_{\beta-\gamma}&=&s_{p_{l,i-1}}=s_{l}\cdot s_{l+1}\cdot ...\cdot s_{i-1}\cdot...\cdot s_{l+1}\cdot s_{l}\\ 
 &=&s_{i-1}\cdot s_{i-2}\cdot...\cdot s_{l+1}\cdot s_{l}\cdot s_{l+1}\cdot...\cdot s_{i-2}\cdot s_{i-1}.\\  
 \end{array}$$
Thus, we can take $u=s_{i}\cdot s_{i+1}\cdot...\cdot s_{n-1}\cdot s_{0}\cdot s_{1}\cdot  ...\cdot s_{k}$ since $s_{k}\cdot s_{k}=s_{k}.$

\textit{Case a5:} Assume that $\alpha=p_{i,j}$ and $\beta=p_{l,k}$ where $0\leq j<k \leq i< l\leq n-1.$ Then $\gamma=p_{l,j}.$ Furthermore, $\alpha-\gamma=p_{i,l-1}$ and $\beta-\gamma=p_{j+1,k},$ hence both are roots. This case is similar to $a4).$

Now, note that, in all cases above we have $s_{\alpha}=u\cdot s_{\alpha-\gamma}\cdot u^{-1},$ $s_{\beta}=u^{-1}\cdot s_{\beta-\gamma}\cdot u$ where $s_{\gamma}=u\cdot u^{-1}.$ Also, note that $u^{-1}\cdot s_{\beta}=s_{\beta}\cdot u^{-1}$ since each simple root that appears in $u^{-1}$ is in the support of $\beta$ since  $\gamma<\beta.$ Hence,

$\begin{array}{lll}s_{\alpha}\cdot s_{\beta}&=&u\cdot s_{\alpha-\gamma}\cdot u^{-1}\cdot s_{\beta}=u\cdot s_{\alpha-\gamma}\cdot s_{\beta}\cdot u^{-1}\\
&=&u\cdot s_{\alpha-\gamma}\cdot u^{-1}\cdot s_{\beta-\gamma}\cdot u \cdot u^{-1}=s_{\alpha}\cdot s_{\beta-\gamma}\cdot s_{\gamma}\\
\end{array}$

\textit{Case b:} Now, assume that $\gamma=\alpha \cap \beta$ is not a root. First, we will show that $\gamma=\gamma_{1}+\gamma_{2}$ for some $\gamma_{1},\gamma_{2}\in \Pi_{\text{aff}}^{\text{re},\,+}$ and $\alpha-\gamma_{1},\alpha-\gamma_{2},\alpha-\gamma,\beta-\gamma_{1},\beta-\gamma_{2}$ and $\beta-\gamma$ are all roots. Here, we have two cases;

\textit{Case b1:} Suppose that $\alpha=p_{i,j}$ and $\beta=p_{k,l}$ where $0\leq k\leq j<i-1\leq l-1<n-2.$ Note that, $\gamma=\alpha \cap \beta=p_{k,j}+p_{i,l}$ so we can take $\gamma_{1}=p_{k,j}$ and $\gamma_{2}=p_{i,l}.$ Also, note that $\alpha-\gamma_{1}=p_{i,k-1},$ $\alpha-\gamma_{2}=p_{l+1,j},$ $\alpha-\gamma=p_{l+1,k-1},$ $\beta-\gamma_{1}=p_{j+1,l},$ $\beta-\gamma_{2}=p_{k,i-1},$ and $\beta-\gamma=p_{j+1,i-1},$ thus all are roots.

\textit{Case b2:} Assume that $\alpha=p_{i,j}$ and $\beta=p_{k,l}$ where $0< k\leq j<i-1\leq l-1\leq n-2.$ Again, $\gamma=\alpha \cap \beta=p_{k,j}+p_{i,l}$ so we can take $\gamma_{1}=p_{k,j}$ and $\gamma_{2}=p_{i,l}.$ Furthermore, $\alpha-\gamma_{1}=p_{i,k-1},$ $\alpha-\gamma_{2}=p_{l+1,j},$ $\alpha-\gamma=p_{l+1,k-1},$ $\beta-\gamma_{1}=p_{j+1,l},$ $\beta-\gamma_{2}=p_{k,i-1},$ and $\beta-\gamma=p_{j+1,i-1},$ so all are roots.

Now, by $a)$ we can obtain some reduced expressions for $s_{\alpha}$ and $s_{\beta}$ such that $s_{\alpha}=u_{1}\cdot s_{\alpha-\gamma_{1}}\cdot u_{1}^{-1}$ and $s_{\beta-\gamma_{1}}=u_{1}^{-1}\cdot s_{\beta-\gamma_{1}}\cdot u_{1}$ for a permutation $u_{1}$ such that $s_{\gamma_{1}}=u_{1}\cdot u_{1}^{-1}$ since both $\alpha-\gamma_{1}$ and $\beta-\gamma_{1}$ are roots. Similarly, we can write $s_{\alpha-\gamma_{1}}=u_{2}\cdot s_{\alpha-\gamma_{1}-\gamma_{2}}\cdot u_{2}^{-1}$ and $s_{\beta-\gamma_{1}}=u_{2}^{-1}\cdot s_{\beta-\gamma_{1}-\gamma_{2}}\cdot u_{2}$ for a permutation $u_{2}$ such that $s_{\gamma_{2}}=u_{2}\cdot u_{2}^{-1}$ since both $\alpha-\gamma_{1}-\gamma_{2}$ and $\beta-\gamma_{1}-\gamma_{2}$ are roots. Also, note that $u_{1}^{-1}\cdot s_{\beta}=s_{\beta}\cdot u_{1}^{-1}$ and $u_{2}^{-1}\cdot s_{\beta}=s_{\beta}\cdot u_{2}^{-1}$ since each simple root that appears in $u_{1}^{-1}$ and $u_{2}^{-1}$ is in the support of $\beta$. Furthermore, $\text{supp}(\gamma_{1})$ and $\text{supp}(\gamma_{2})$ are disconnected so $s_{\gamma_{1}}\cdot s_{\gamma_{2}}=s_{\gamma_{2}}\cdot s_{\gamma_{1}}$ and $u_{1}\cdot s_{\gamma_{2}}=s_{\gamma_{2}}\cdot u_{1}$. Thus,

$\begin{array}{lll}
s_{\alpha}\cdot s_{\beta}&=&u_{1}\cdot s_{\alpha-\gamma_{1}}\cdot u_{1}^{-1}\cdot s_{\beta}
=u_{1}\cdot u_{2}\cdot s_{\alpha-\gamma_{1}-\gamma_{2}}\cdot u_{2}^{-1}\cdot u_{1}^{-1}\cdot s_{\beta} \\
&=&u_{1}\cdot u_{2}\cdot s_{\alpha-\gamma_{1}-\gamma_{2}}\cdot s_{\beta}\cdot u_{2}^{-1}\cdot u_{1}^{-1} \\
&=&u_{1}\cdot u_{2}\cdot s_{\alpha-\gamma_{1}-\gamma_{2}}\cdot u_{2}^{-1}\cdot u_{1}^{-1}\cdot s_{\beta-\gamma_{2}-\gamma_{1}}\cdot u_{1}\cdot u_{2}\cdot  u_{2}^{-1}\cdot u_{1}^{-1} \\

&=& s_{\alpha}\cdot s_{\beta-\gamma}\cdot s_{\gamma_{2}}\cdot s_{\gamma_{1}} 
= s_{\alpha}\cdot s_{\beta-\gamma}\cdot s_{\gamma_{1}}\cdot s_{\gamma_{2}} \\

\end{array}$

$3)$ \textit{Case a:} Note that, $\text{supp}(\alpha)\cap \text{supp}(\beta)=\emptyset$ in this case and since both $\alpha$ and $\beta$ are smaller than $c,$ the multiplication $s_{\alpha}s_{\beta}$ is reduced so $s_{\alpha}s_{\beta}=s_{\alpha}\cdot s_{\beta}.$ Now, by Corollary \ref{cor13.6} we get $s_{\alpha}\cdot s_{\beta}=s_{\alpha} s_{\beta}\leq s_{\alpha+\beta}.$

\textit{Case b:} Here, $\text{supp}(\alpha)$ and $\text{supp}(\beta)$ are disconnected. Hence, the multiplication, $s_{\alpha}s_{\beta}$ is reduced and by the braid relations any simple reflection that appears in $s_{\alpha}$ will commute with all simple reflections which appear in $s_{\beta}$ so $s_{\alpha}s_{\beta}=s_{\beta}s_{\alpha}.$ Thus, $s_{\alpha}\cdot s_{\beta}=s_{\alpha}s_{\beta}=s_{\beta}s_{\alpha}=s_{\beta}\cdot s_{\alpha}.$ \end{proof}

\begin{cor}\label{cor5.9} Let $w=s_{\beta_{1}}\cdot s_{\beta_{2}}\cdot ...\cdot s_{\beta_{r}}$ be such that $\beta_{i} \in \Pi_{\text{aff}}^{\text{re},\,+}$ for all $i$ and $\sum_{i=1}^{r}\beta_{i}=\alpha$ where $\alpha<c$. Then $w\leq s_{\alpha}$.

\end{cor}

\begin{proof} The proof follows by using Lemma \ref{lemma5.7} part $3)$ repeatedly. 

\end{proof}

\subsection{Lemma about Decomposition}

\begin{lemma}\label{lemma6} Let $\alpha,\beta \in \Pi_{\text{aff}}^{\text{re},\,+}$ such that $\beta<c$ and $\alpha=mc+\beta$ for an integer $m\geq 1.$ Then $s_{\alpha}=(s_{\beta}s_{\beta'})^{m}s_{\beta}$ where $\beta+\beta'=c.$  \end{lemma}

\begin{proof} First, we will show that $s_{\alpha}=s_{\beta}s_{\gamma}s_{\beta}$ where  $\gamma=mc-\beta$. Note that, $\langle \gamma,\beta^\vee  \rangle =\langle mc-\beta,\beta^\vee \rangle=-2$, hence $s_{\beta}(\gamma)=\gamma-\langle \gamma,\beta^\vee \rangle \beta=\gamma+2\beta=\alpha.$ So $s_{\alpha}=s_{s_{\beta}(\gamma)}=s_{\beta}s_{\gamma}s_{\beta}^{-1}=s_{\beta}s_{\gamma}s_{\beta}.$
Next, we will prove the statement by induction on $m$. First, suppose that $m=1.$ Then $s_{\alpha}=s_{\beta}s_{\gamma}s_{\beta}$ where  $\gamma=mc-\beta=c-\beta.$ Now, assume that the statement is true for $m=k$ for a positive integer $k$. We will prove that it is also true for $m=k+1.$ Let $\alpha=(k+1)c+\beta$ where $\beta \in \Pi_{\text{aff}}^{\text{re},\,+}$ and $\beta<c.$ Again, we can write $s_{\alpha}=s_{\beta}s_{\gamma}s_{\beta}$ where  $\gamma=(k+1)c-\beta.$ Note that, $\gamma=(k+1)c-\beta=kc+c-\beta$ and $\beta':=c-\beta$ is a positive root which is smaller than $c.$ So by the induction assumption we can write  $s_{\gamma}=(s_{\beta'}s_{\beta})^{k}s_{\beta'}.$ Thus $s_{\alpha}=s_{\beta}s_{\gamma}s_{\beta}=s_{\beta}(s_{\beta'}s_{\beta})^{k}s_{\beta'}s_{\beta}=(s_{\beta}s_{\beta'})^{k+1}s_{\beta}.$ \end{proof}

\begin{cor}\label{cor7} Let $\alpha \in \Pi_{\text{aff}}^{\text{re},\,+}$ such that $\alpha>c.$ Then $s_{\alpha}=s_{\beta_{1}}s_{\beta_{2}}...s_{\beta_{k}}$ for some $k$ where $\beta_{i} \in \Pi_{\text{aff}}^{\text{re},\,+}$ such that $\beta_{i}<c$ for all $i$ and $\sum_{i=1}^{k}\beta_{i}=\alpha.$
\end{cor}

\begin{proof} Now, suppose that $\alpha=mc+\beta$ for some integer $m$ such that $m\geq 1$ and an affine positive real root $\beta$ which is smaller than $c.$ Then by Lemma \ref{lemma6} we have $s_{\alpha}=(s_{\beta}s_{\beta'})^{m}s_{\beta}$ where  $\beta+\beta'=c.$\end{proof}

\subsection{Lemmas about Lengths}

The main goal of this section is to compute the lengths of some elements of the affine Weyl group $W_{\text{aff}}$ which appear in the curve neighborhoods.

\begin{lemma}\label{lemma10} Let $\beta,\beta',\alpha,\alpha' \in \Pi_{\text{aff}}^{\text{re},\,+}$ such that $\alpha+\alpha'=\beta+\beta'=c.$ Also, suppose that $\beta\neq \alpha$. Then $(s_{\beta}s_{\beta'})^{m}\neq (s_{\alpha}s_{\alpha'})^{m}$ for any positive integer $m.$
\end{lemma}

\begin{proof}  First, assume that $\alpha_{0}\notin \text{supp}(\beta).$ Now, $s_{\beta'}=s_{c-\beta}=s_{-\beta}t_{-\beta}=s_{\beta}t_{-\beta}$. Thus $s_{\beta}s_{\beta'}=s_{\beta}s_{c-\beta}=s_{\beta}s_{\beta}t_{-\beta}=t_{-\beta}.$ So $(s_{\beta}s_{\beta'})^{m}=(t_{-\beta})^{m}=t_{-m\beta}.$ Second, assume that $\alpha_{0}\in \text{supp}(\beta).$ Then $s_{\beta}=s_{c-(c-\beta)}=s_{-(c-\beta)}t_{-(c-\beta)}=s_{(c-\beta)}t_{-(c-\beta)}$. Hence
$$s_{\beta}s_{\beta'}=s_{(c-\beta)}t_{-(c-\beta)}s_{c-\beta}=t_{s_{c-\beta}(-(c-\beta))}
=t_{(c-\beta)}$$
So $(s_{\beta}s_{\beta'})^{m}=(t_{c-\beta})^{m}=t_{m(c-\beta)}.$
Similarly, we have either $(s_{\alpha}s_{\alpha'})^{m}=t_{-m\alpha}$ or $(s_{\alpha}s_{\alpha'})^{m}=t_{m(c-\alpha)}$. So if $(s_{\beta}s_{\beta'})^{m}= (s_{\alpha}s_{\alpha'})^{m}$ then either $t_{-m\beta}=t_{-m\alpha}$, $t_{-m\beta}=t_{m(c-\alpha)},$ $t_{m(c-\beta)}=t_{-m\alpha},$ or $t_{m(c-\beta)}=t_{m(c-\alpha)}$ so we have either $\beta=\alpha$, $\beta-\alpha=c,$ or $\alpha-\beta=c,$ but this is a contradiction. \end{proof}

\begin{lemma}\label{lemma10.2} Let $\beta \in \Pi_{\text{aff}}^{\text{re},\,+}$ such that $\beta<c$. Then $\sum_{\gamma \in \Pi^{+}}|\langle \beta,\gamma  \rangle |=2(n-1).$

\end{lemma}

\begin{proof} First, assume that $\alpha_{0}\notin \text{supp}(\beta).$ Let $\beta=p_{i,j}=\varepsilon_{i}-\varepsilon_{j+1}$ such that $1\leq i<j\leq n-1$ and $\gamma \in \Pi^{+}.$ We have several cases here;
\begin{itemize}

\item If $\gamma=\beta$ then $\langle \beta,\gamma  \rangle =2.$ If $\gamma<\beta$ then $\gamma=p_{k,l}$ such that either $i<k<l<j$ which implies $\langle\beta,\gamma  \rangle =0$, or $k=i<l<j$ or $i<k<l=j$ which implies $\langle\beta,\gamma  \rangle =1$ and there are $2(j-i)$ such $\gamma.$

\item If $\text{supp}(\gamma)$ and $\text{supp}(\beta)$ are disconnected then $\langle\beta,\gamma  \rangle =0.$

\item If $\beta+\gamma$ is a root then $\gamma=p_{k,l}$ such that either $1\leq k<l=i-1$ or $k=j+1<l\leq n-1$ and so $\langle \beta,\gamma  \rangle =-1.$ We have $n-j+i-2$ such $\gamma.$

\item If $\beta \cap \gamma\neq \emptyset$ and $\beta \cap \gamma\neq \beta, \gamma$ then $\gamma=p_{k,l}$ such that either $1\leq k<i<l<j$ or $i<k<j<l\leq n-1$ which implies $\langle \beta,\gamma  \rangle =0.$ 

\item If $\gamma>\beta$ then $\gamma=p_{k,l}$ such that either $1\leq k<i<j<l\leq n-1$ which implies $\langle \beta,\gamma  \rangle =0$, or $k=i<j<l\leq n-1$ or $1\leq k<i<j=l$ which follows by $\langle \beta,\gamma  \rangle =1$ and there are $n-j+i-2$ such $\gamma.$

\end{itemize}
Thus  $\sum_{\gamma \in \Pi^{+}}|\langle \beta,\gamma  \rangle |=2+(n-j+i-2)+2(j-i)+(n-j+i-2)=2(n-1).$ Next, if $\alpha_{0}\in \text{supp}(\beta)$ then the proof is similar. 

\end{proof}

\begin{lemma} \label{lemma11} Let $m$ be a positive integer and $\beta,\beta' \in \Pi_{\text{aff}}^{\text{re},\,+}$ such that $\beta+\beta'=c.$ Then $(s_{\beta}s_{\beta'})^{m}$ is reduced for any positive integer $m$. In particular, $\ell{((s_{\beta}s_{\beta'})^{m})}=2m(n-1)$. 
\end{lemma}

\begin{proof} First, note that, since both $\beta$ and $\beta'$ are smaller than $c$ we have $\ell{(\beta)}=2|\text{supp}(\beta)|-1$ and $\ell{(\beta')}=2|\text{supp}(\beta')|-1$.  Also, $|\text{supp}(\beta)|+|\text{supp}(\beta')|=n.$ So 
$$\ell{((s_{\beta}s_{\beta'})^{m})}\leq m(\ell{(\beta)}+\ell{(\beta')})=m(2|\text{supp}(\beta)|-1+2|\text{supp}(\beta')|-1)=2m(n-1).$$

Now, assume that $\alpha_{0}\notin \text{supp}(\beta).$ Then we have $(s_{\beta}s_{\beta'})^{m}=t_{-m\beta};$ see the proof of Lemma \ref{lemma10}. Moreover, 
$$\ell{(t_{-m\beta})}=\sum_{\gamma \in \Pi^{+}}|\chi{(\gamma<0)}+\langle -m\beta,\gamma  \rangle |$$
by Equation  (\ref{equation5.5}). Here, note that, $\chi{(\gamma<0)=0}$ for all $\gamma \in \Pi^{+}$ and $\langle-m\beta,\gamma  \rangle =-m\langle\beta,\gamma  \rangle $ so 
$$\ell{(t_{-m\beta})}=\sum_{\gamma \in \Pi^{+}}|-m \langle \beta,\gamma  \rangle |= m\sum_{\gamma \in \Pi^{+}}|\langle \beta,\gamma  \rangle |.$$
By Lemma \ref{lemma10.2}, $\sum_{\gamma \in \Pi^{+}}|\langle \beta,\gamma  \rangle |=2(n-1).$ So $\ell{(t_{-m\beta})}=2m(n-1).$ 
Next, assume that $\alpha_{0}\in \text{supp}(\beta).$ Then we have $(s_{\beta}s_{\beta'})^{m}=t_{m(c-\beta)},$ see the proof of Lemma \ref{lemma10}. By the similar arguments above, one can show that $\ell{(t_{m(c-\beta)})}=2m(n-1).$ Hence $(s_{\beta}s_{\beta'})^{m}$ is reduced and $\ell{((s_{\beta}s_{\beta'})^{m})}=2m(n-1)$. \end{proof}

\begin{remark} Alternatively, one can prove Lemma \ref{lemma11} by using the following arguments: For any $\beta^{\vee} \in Q^{\vee}$ and $w \in W$ we have $\ell{(t_{\beta^{\vee}})}=\ell{(t_{w(\beta)})}.$ So one may restrict to the case where $\beta^{\vee}$ is a dominant root. Then we have $\ell{(t_{\beta^{\vee}})}=\langle \beta^{\vee},2\rho\rangle$ where $\rho=\sum_{i=1}^{n-1}\omega_{i}$ is the sum of the fundamental weights, see \cite{LS01}. But note that $\langle \alpha^{\vee}_{i},\omega_{j}\rangle=\delta_{ij}$ for $1\leq i,j\leq n-1$ and if we take $\beta=\theta=\sum_{i=1}^{n-1}\alpha_{i}$ we get $\langle \beta^{\vee},2\rho\rangle=2(n-1).$
\end{remark}

\begin{lemma}\label{lemma12.71} Let $\beta, \gamma_{1},\gamma_{2},...,\gamma_{k} \in \Pi_{\text{aff}}^{\text{re},\,+}$ such that $\beta<c$ and for any $1\leq i,j\leq k$ such that $i\neq j$, either $\text{supp}(\gamma_{i})$ and $\text{supp}(\gamma_{j})$ are disconnected or both $\gamma_{i}\perp \gamma_{j}$ and $\gamma_{i},\gamma_{j}$ are comparable. Also, suppose that either $\gamma_{i}\leq \beta$ or both $\beta \cap \gamma_{i}=\emptyset$ and $\beta \perp \gamma_{i}$ for any $1\leq i\leq k.$ Moreover, let $\gamma \in \Pi^{+}$ such that  $s_{\gamma_{1}}s_{\gamma_{2}}...s_{\gamma_{k}}(\gamma)<0.$ Then $<s_{\gamma_{k}}...s_{\gamma_{2}}s_{\gamma_{1}}(\beta), \gamma>\leq 0.$ \end{lemma}

\begin{proof}First, observe that $<s_{\gamma_{k}}...s_{\gamma_{2}}s_{\gamma_{1}}(\beta), \gamma>=<\beta,s_{\gamma_{1}}s_{\gamma_{2}}...s_{\gamma_{k}}(\gamma)>.$ Now, note that $<\beta,\gamma_{i}>$ is either $2$, $1$ or $0$ since we have either $\gamma_{i}\leq \beta$ or both $\beta \cap \gamma_{i}=\emptyset$ and $\beta \perp \gamma_{i}$. Also, note that $s_{\gamma_{1}}s_{\gamma_{2}}...s_{\gamma_{k}}(\gamma)=\gamma-\sum_{i=1}^{k}<\gamma,\gamma_{i}^{\vee}>\gamma_{i}$ since $\gamma_{i}\perp \gamma_{j}$ for any $1\leq i,j\leq k$ such that $i\neq j.$ Now, if $<\gamma,\gamma_{i}^{\vee}>\leq 0$ for all $i=1,2,...,k$ then $s_{\gamma_{1}}s_{\gamma_{2}}...s_{\gamma_{k}}(\gamma)\geq \gamma >0.$ So $<\gamma,\gamma_{i}^{\vee}>>0$ for some $i.$ This implies that either $<\gamma,\gamma_{i}^{\vee}>=2$ or $<\gamma,\gamma_{i}^{\vee}>=1.$ Now, if $<\gamma,\gamma_{i}^{\vee}>=2$ then $\gamma=\gamma_{i}.$ But then $\gamma \perp \gamma_{j}$ for all $j\in \{1,2,...,k\} \setminus \{i\},$ which follows by $s_{\gamma_{1}}s_{\gamma_{2}}...s_{\gamma_{k}}(\gamma)=-\gamma_{i}$. Then $<\beta,s_{\gamma_{1}}s_{\gamma_{2}}...s_{\gamma_{k}}(\gamma)>=<\beta,-\gamma_{i}>=-<\beta,\gamma_{i}>$. Here, note that if $<\beta,\gamma_{i}>\neq 0$ then $\beta$ is not perpendicular to $\gamma_{i}$ which implies that  $\gamma_{i}\leq \beta$ and $<\beta,\gamma_{i}>$ is either $2$ or $1$ hence we are done with this case. Now, suppose that $<\gamma,\gamma_{i}^{\vee}>=1.$ Then either $\gamma <\gamma_{i}$ or $\gamma_{i} <\gamma$. Now suppose that $\gamma <\gamma_{i}$. Then we have three cases;

\begin{itemize}

\item $\gamma \perp \gamma_{j}$ for all $j\in \{1,2,...,k\} \setminus \{i\}$ which follows by $s_{\gamma_{1}}s_{\gamma_{2}}...s_{\gamma_{k}}(\gamma)=\gamma-\gamma_{i}<0.$ Then $<\beta,s_{\gamma_{1}}s_{\gamma_{2}}...s_{\gamma_{k}}(\gamma)>=<\beta,\gamma-\gamma_{i}>=<\beta,\gamma>-<\beta,\gamma_{i}>.$  If $<\beta,\gamma_i>=2$ then $\beta=\gamma_{i}$ and $<\beta,\gamma>=1$ in this case. If $<\beta,\gamma_i>=1$ then either $<\beta,\gamma>=1$ or $<\beta,\gamma>=0$. Last, if $<\beta,\gamma_i>=0$ then $<\beta,\gamma>=0.$ In all cases, we get $<\beta,\gamma>-<\beta,\gamma_{i}>\leq 0.$

\item For an index $q$, $<\gamma,\gamma_{q}^{\vee}>=-1$ and $\gamma \perp \gamma_{j}$ for all $j\in \{1,2,...,k\} \setminus \{i,q\}.$ Thus $s_{\gamma_{1}}s_{\gamma_{2}}...s_{\gamma_{k}}(\gamma)=\gamma-\gamma_{i}+\gamma_{q}<0.$ Now, we have either $\text{supp}(\gamma_{i})$ and $\text{supp}(\gamma_{q})$ are disconnected or $\gamma_{q}<\gamma_{i}$ or $\gamma_{i}<\gamma_{q}$. Now, by the facts that $\gamma <\gamma_{i}$ and  $<\gamma,\gamma_{q}^{\vee}>=-1$ we have to have $\gamma_{q}<\gamma_{i}.$ So $$<\beta,s_{\gamma_{1}}s_{\gamma_{2}}...s_{\gamma_{k}}(\gamma)>=<\beta,\gamma-\gamma_{i}+\gamma_{q}>=<\beta,\gamma>-<\beta,\gamma_{i}>+<\beta,\gamma_{q}>.$$ If $<\beta,\gamma_i>=2$ then $\beta=\gamma_{i}$ which follows by $<\beta,\gamma>=1$ and $<\beta,\gamma_{q}>=0$. If $<\beta,\gamma_i>=1$ then $<\beta,\gamma_{q}>=0 $ since $\gamma_{q}$ is smaller than $\gamma_{i}$ and $\gamma_{q}\perp \gamma_{i}$. Also, we have either $<\beta,\gamma>=1$ or $<\beta,\gamma>=0$ in this case since $\gamma<\gamma_{i}\leq \beta$. Last, if $<\beta,\gamma_i>=0$ then $<\beta,\gamma>=<\beta,\gamma_{q}>=0$ since both $\gamma_{q}$ and $\gamma$ are smaller than $\gamma_{i}.$ Note that, in all cases, we get $<\beta,\gamma>-<\beta,\gamma_{i}>+<\beta,\gamma_{q}>\leq 0.$

\item For an index $r\neq i$, $<\gamma,\gamma_{q}^{\vee}>=1$ and $\gamma \perp \gamma_{j}$ for all $j\in \{1,2,...,k\} \setminus \{i,r\}.$ Then either $\gamma<\gamma_{r}$ or $\gamma_{r}<\gamma$. But since either $\text{supp}(\gamma_{i})$ and $\text{supp}(\gamma_{r})$ are disconnected or both $\gamma_{i}\perp \gamma_{r}$ and $\gamma_{i}$ and $\gamma_{r}$ are comparable we have to have $\gamma_{r}<\gamma$ which implies that $\gamma_{r} <\gamma_{i}$. Thus $s_{\gamma_{1}}s_{\gamma_{2}}...s_{\gamma_{k}}(\gamma)=\gamma-\gamma_{i}-\gamma_{r}.$ So $$<\beta,s_{\gamma_{1}}s_{\gamma_{2}}...s_{\gamma_{k}}(\gamma)>=<\beta,\gamma-\gamma_{i}+\gamma_{r}>=<\beta,\gamma>-<\beta,\gamma_{i}>+<\beta,\gamma_{r}>.$$ If $<\beta,\gamma_i>=2$ then $\beta=\gamma_{i}$ which follows by $<\beta,\gamma>=1$ and $<\beta,\gamma_{r}>=0$. If $<\beta,\gamma_i>=1$ then $<\beta,\gamma_{r}>=0 $ since $\gamma_{r}$ is smaller than $\gamma_{i}$ and $\gamma_{r}\perp \gamma_{i}$. Also, we have either $<\beta,\gamma>=1$ or $<\beta,\gamma>=0$ in this case since $\gamma<\gamma_{i}\leq \beta$. Last, if $<\beta,\gamma_i>=0$ then $<\beta,\gamma>=<\beta,\gamma_{r}>=0$ since both $\gamma_{r}$ and $\gamma$ are smaller than $\gamma_{i}.$ Note that, in all cases, we get $<\beta,\gamma>-<\beta,\gamma_{i}>+<\beta,\gamma_{r}>\leq 0.$
\end{itemize}

The case where $\gamma_{i} <\gamma$ is similar.

\end{proof}

\begin{lemma}\label{lemma12.72} Let $\beta,\beta' \in \Pi_{\text{aff}}^{\text{re},\,+}$ such that $\beta+\beta'=c$ and let $\mathbf{d}=(d_0,d_1,...,d_{n-1})$ be a degree where $d_{i}=0$ for some $i.$ Also, assume that $z_{\mathbf{d}}=s_{\gamma_{1}}s_{\gamma_{2}}...s_{\gamma_{k}}$ for some k where this expression is obtained in Theorem \ref{thm12.7}. If either $\beta' \cap \gamma_{i}=\emptyset$ or both $\beta' > \gamma_{i}$ and $\beta' \perp \gamma_{i}$ for any $1\leq i\leq k$ then $w=(s_{\beta}s_{\beta'})^{m}z_{\mathbf{d}}$ is reduced for any positive integer $m.$
In particular, $\ell{((s_{\beta}s_{\beta'})^{m} z_{\mathbf{d}})}=2m(n-1)+\ell{(z_{\mathbf{d}})}.$ 

\end{lemma}

\begin{proof} Note that by Lemma \ref{lemma11} we have $\ell{((s_{\beta}s_{\beta'})^{m})}=2m(n-1)$ so $\ell{((s_{\beta}s_{\beta'})^{m}{z_{\mathbf{d}}}})\leq \ell{((s_{\beta}s_{\beta'})^{m})}+\ell{(z_{\mathbf{d}})}=2m(n-1)+\ell{(z_{\mathbf{d}})}$. We will show that $\ell{((s_{\beta}s_{\beta'})^{m}{z_{\mathbf{d}}}})=2m(n-1)+\ell{(z_{\mathbf{d}})}$ to prove that $w=(s_{\beta}s_{\beta'})^{m}z_{\mathbf{d}}$ is reduced for any positive integer $m$. 

First, we can suppose that $d_{0}=0$ by the automorphism of the Dynkin diagram, $\varphi$; see Section (\ref{dynkin}) which implies that $\alpha_{0}\notin \text{supp}(\gamma_{i})$ for any $i=1,2,...,k$ so $\gamma_{i}\in \Pi^{+}$ for all $i.$ Now, suppose that $\alpha_{0}\notin \text{supp}(\beta).$ Then by Lemma \ref{lemma10}, we have $(s_{\beta}s_{\beta'})^{m}=t_{-m\beta}.$ So $w=(s_{\beta}s_{\beta'})^{m} z_{\mathbf{d}}=(s_{\beta}s_{\beta'})^{m}s_{\gamma_{1}}s_{\gamma_{2}}...s_{\gamma_{k}}=t_{-m\beta}s_{\gamma_{1}}s_{\gamma_{2}}...s_{\gamma_{k}}.$ We also have $t_{-m\beta}s_{\gamma_{1}}s_{\gamma_{2}}...s_{\gamma_{k}}=s_{\gamma_{1}}s_{\gamma_{2}}...s_{\gamma_{k}}t_{s_{\gamma_{k}}...s_{\gamma_{2}}s_{\gamma_{1}}(-m\beta)}$. Now, by Lemma $3.1$ in \cite{LS01}, we get
 
$$\begin{array}{lll}\ell{(w)}&=&\sum_{\gamma \in \Pi^{+}}|\chi{(s_{\gamma_{1}}...s_{\gamma_{k}}(\gamma)<0)}+<s_{\gamma_{k}}...s_{\gamma_{1}}(-m\beta),\gamma>|\\
&=&\sum_{\gamma \in \Pi^{+}}|\chi{(s_{\gamma_{1}}...s_{\gamma_{k}}(\gamma)<0)}-m<s_{\gamma_{k}}...s_{\gamma_{1}}(\beta),\gamma>|.

\end{array}$$
Now, note that  $\chi{(s_{\gamma_{1}}...s_{\gamma_{k}}(\gamma)<0)}=0$ if $s_{\gamma_{1}}...s_{\gamma_{k}}(\gamma)>0$  and $\chi{(s_{\gamma_{1}}...s_{\gamma_{k}}(\gamma)<0)}=1$ otherwise. Thus 
$$\ell{(w)}=m\sum_{s_{\gamma_{1}}...s_{\gamma_{k}}(\gamma)>0} |<s_{\gamma_{k}}...s_{\gamma_{1}}(\beta),\gamma>|+\sum_{s_{\gamma_{1}}...s_{\gamma_{k}}(\gamma)<0}|1-m<s_{\gamma_{k}}...s_{\gamma_{1}}(\beta),\gamma>|$$
where $\gamma \in \Pi^{+}.$ Moreover, either $ \gamma_{i}\leq \beta$ or both $\beta \cap \gamma_{i}=\emptyset$ and $\beta \perp \gamma_{i}$ for any $1\leq i\leq k$ since we have either $\beta' \cap \gamma_{i}=\emptyset$ or both $\beta' > \gamma_{i}$ and $\beta' \perp \gamma_{i}$ for any $1\leq i\leq k$. Here, notice that, for $\gamma \in \Pi^{+}$ if $s_{\gamma_{1}}s_{\gamma_{2}}...s_{\gamma_{k}}(\gamma)<0$ then $<s_{\gamma_{k}}...s_{\gamma_{2}}s_{\gamma_{1}}(\beta), \gamma>\leq 0$ by Lemma \ref{lemma12.71}, which implies that $|1-m<s_{\gamma_{k}}...s_{\gamma_{1}}(\beta),\gamma>|=1+m|<s_{\gamma_{k}}...s_{\gamma_{1}}(\beta),\gamma>|$. Hence, 
$$\begin{array}{lll}\ell{(w)}&=&m\sum_{s_{\gamma_{1}}...s_{\gamma_{k}}(\gamma)>0} |<s_{\gamma_{k}}...s_{\gamma_{1}}(\beta),\gamma>|+\sum_{s_{\gamma_{1}}...s_{\gamma_{k}}(\gamma)<0}|1-m<s_{\gamma_{k}}...s_{\gamma_{1}}(\beta),\gamma>|\\
&&\\
&=& m\sum_{s_{\gamma_{1}}...s_{\gamma_{k}}(\gamma)>0} |<s_{\gamma_{k}}...s_{\gamma_{1}}(\beta),\gamma>|+\sum_{s_{\gamma_{1}}...s_{\gamma_{k}}(\gamma)<0}(1+m|<s_{\gamma_{k}}...s_{\gamma_{1}}(\beta),\gamma>|)\\
&&\\
&=&m\sum_{\gamma \in \Pi^{+}}|<s_{\gamma_{k}}...s_{\gamma_{1}}(\beta),\gamma>|+|\{\gamma \in \Pi^{+}:s_{\gamma_{1}}...s_{\gamma_{k}}(\gamma)<0\}|.

\end{array}$$
Here, the fact that $\gamma_{i}\leq \beta$ for all $i$ and the assumptions on $\gamma_{i}$'s force us to have $s_{\gamma_{k}}...s_{\gamma_{1}}(\beta)=\pm \beta$, $s_{\gamma_{k}}...s_{\gamma_{1}}(\beta)=\beta-\gamma_{i}$ or $s_{\gamma_{k}}...s_{\gamma_{1}}(\beta)=\beta-\gamma_{i}-\gamma_{q}$ for some $1\leq i,q\leq k.$ In all these cases $|s_{\gamma_{k}}...s_{\gamma_{1}}(\beta)|$ is a positive real root which is smaller than $c.$ Thus, we get $\sum_{\gamma \in \Pi^{+}}|<s_{\gamma_{k}}...s_{\gamma_{1}}(\beta),\gamma>|=2(n-1)$ by Lemma \ref{lemma10.2}. Furthermore, $|\{\gamma \in \Pi^{+}:s_{\gamma_{1}}...s_{\gamma_{k}}(\gamma)<0\}|=\ell{(s_{\gamma_{1}}...s_{\gamma_{k}})}=\ell{(z_{\mathbf{d}})}$ since $\gamma_{i}\in \Pi^{+}$ for all $i.$ 

If $\alpha_{0} \in \text{supp}(\beta)$ then the proof is similar.

\end{proof}

\begin{cor}\label{cor12.8} Let $\beta,\beta',\alpha \in \Pi_{\text{aff}}^{\text{re},\,+}$ such that $\beta+\beta'=c$ and $\alpha<c$. If either $\beta'\leq c-\alpha$ or both $\beta'>\alpha$ and $\beta' \perp \alpha$ then $(s_{\beta}s_{\beta'})^{m}s_{\alpha}$ is reduced for any positive integer $m.$
In particular, $\ell{((s_{\beta}s_{\beta'})^{m}s_{\alpha}})=2m(n-1)+\ell{(s_{\alpha})}$.

\end{cor}

\begin{proof} Let $d=\alpha.$ Then $z_{\mathbf{d}}=s_{\alpha}.$ Now, if $\beta'\leq c-\alpha$ then $\beta' \cap \alpha=\emptyset$ so by Lemma \ref{lemma12.72} $(s_{\beta}s_{\beta'})^{m}s_{\alpha}$ is reduced for any positive integer $m$ and $\ell{((s_{\beta}s_{\beta'})^{m}s_{\alpha}})=2m(n-1)+\ell{(s_{\alpha})}$.

\end{proof}

\section{The Moment Graph and Curve Neighborhoods}\label{curve neighborhoods}

In this chapter we recall some basic facts about the affine flag manifold of type $A_{n-1}^{(1)}$ and give the definition of the moment graph and the (combinatorial) curve neighborhoods. We refer to\cite{kumar} especially $\S 12$ and $\S 13$ for further details. Let $G$ be the Lie group $SL_{n}(\C)$. Let $B\subset G$ be the Borel subgroup, the set of upper triangular matrices. Now, let $G(\C[t,t^{-1}])$ be the group of Laurent polynomial loops from $\C^{*}$ to $G$ and $\mathcal{G}:=\C^{*}\ltimes G(\C[t,t^{-1}])$ where $\C^{*}$ acts by loop rotation. Let $\mathcal{B}$ be the standard Iwahori subgroup of $\mathcal{G}$ determined by $B$. The affine flag manifold in type $A_{n-1}^{(1)}$ is given by $\mathcal{X}:=\mathcal{G}/ \mathcal{B}$. The Schubert varieties and the curve neighborhoods of Schubert varieties in this context were defined in \cite[\S5]{MM01}. In the following we give a different, but equivalent definition, based on the moment graph, which generalizes the observations from \cite[\S5.2]{MB01}. 

The (undirected) $\mathbf{moment}$ $\mathbf{ graph}$ for $\mathcal{X}$ is the graph given by the following data:

\begin{itemize}

\item The set of $V$ of $\mathbf{vertices}$ is the group $W_{\text{aff}};$ 

\item Let $u, v\in V$ be vertices. Then there is an edge from $u$ to $v$ iff there exists an affine root $\alpha$ such that $v=us_{\alpha}.$ We denote this situation by 
$$u \stackrel{\alpha} \longrightarrow v$$ and we say that the $\mathbf{degree}$ of this edge is $\alpha.$

\end{itemize}

A  \textit{chain} between $u$ and $v$ in the moment graph is a succession of adjacent edges starting with $u$ and ending with $v$

$$\pi: u=u_{0} \stackrel{\beta_{0}} \longrightarrow u_{1}\stackrel{\beta_{1}}\longrightarrow ...\stackrel{\beta_{k-2}}\longrightarrow u_{k-1}\stackrel{\beta_{k-1}}\longrightarrow u_{k}=v .$$

The \textit{degree} of the chain $\pi$ is $\text{deg}(\pi)=\beta_{0}+\beta_{1}+...+\beta_{k-1}.$ A chain is called \textit{increasing} if $\ell{(u_{i})}>\ell{(u_{i-1})}$ for all $i.$ Define the (Bruhat) partial ordering on the elements of $W_{\text{aff}}$ by $u<v$ iff there exists an increasing chain starting with $u$ and ending with $v.$

A $\mathbf{degree}$ is a tuple of nonnegative integers $\mathbf{d}=(d_{0},...,d_{n-1}).$ Notice that it has $n$ components, corresponding to the $n$ affine simple roots $\alpha_{0},...,\alpha_{n-1}.$ There is a natural partial order on degrees: If $\mathbf{d}=(d_{0},...,d_{n-1})$ and $\mathbf{d'}=(d'_{0},...,d'_{n-1})$ then $\mathbf{d}\geq \mathbf{d'}$ iff $d_{i}\geq d'_{i}$ for all $i\in \{0,1,...,n-1\}.$

\begin{defn} Fix a degree $\mathbf{d}$ and $u\in W_{\text{aff}}.$ The $\mathbf{(combinatorial)}$ $\mathbf{curve}$ $\mathbf{neighborhood}$ is the set $\Gamma_{\mathbf{d}}(u)$ consisting of elements $v\in W_{\text{aff}}$ such that:

\begin{itemize}

\item[1)] There exists a chain of some degree $\mathbf{d'}\leq \mathbf{d}$ from $u'\leq u$ to $v$ in the moment graph of $\mathcal{X};$

\item[2)] The elements $v$ are maximal among all of those satisfying the condition in $1).$

\end{itemize}
\end{defn}

\begin{remark} \label{remark12.9}  Let $id \stackrel{\beta_{1}} \longrightarrow s_{\beta_{1}} \stackrel{\beta_{2}} \longrightarrow s_{\beta_{1}}s_{\beta_{2}} ...\stackrel{\beta_{q}} \longrightarrow w=s_{\beta_{1}}s_{\beta_{2}}...s_{\beta_{q}}$
be a path in the moment graph such that $\sum_{i=1}^{q} \beta_{i} <\mathbf{d}$.  Then there are some $\beta_{q+1},\beta_{q+2},...,\beta_{p} \in \Pi_{\text{aff}}^{\text{re},\,+}$ such that $\sum_{i=1}^{p} \beta_{i}=\mathbf{d}.$ Note that, by Corollary \ref{cor7} we can assume that $\beta_{i}<c$ for all $i$. Also, observe that $w=s_{\beta_{1}}s_{\beta_{2}}...s_{\beta_{q}}\leq s_{\beta_{1}}\cdot s_{\beta_{2}}\cdot ...\cdot s_{\beta_{q}}\leq s_{\beta_{1}}\cdot s_{\beta_{2}}\cdot ...\cdot s_{\beta_{q}}\cdot s_{\beta_{q+1}}\cdot s_{\beta_{q+2}}\cdot ...\cdot s_{\beta_{p}}.$
Now, if one can show that $s_{\beta_{1}}\cdot s_{\beta_{2}}\cdot ...\cdot s_{\beta_{q}}\cdot s_{\beta_{q+1}}\cdot s_{\beta_{q+2}}\cdot ...\cdot s_{\beta_{p}}\leq u$ for some $u \in W_{\text{aff}}$ such that $u$ can be reached by a path of degree at most $\mathbf{d}$ that starts with $id$ then it will follow that $u\in \Gamma_{\mathbf{d}}(id).$ 

\end{remark}

\section{Calculation of the Curve Neighborhoods }

In this chapter we will prove our results. In Section \ref{reduction} we will state Theorem \ref{thm22} which implies that it is sufficient to calculate $ \Gamma_{\mathbf{d}}(id)$ to compute $\Gamma_{\mathbf{d}}(w)$ for any given degree $\mathbf{d}$ and $w \in W_{\text{aff}}.$ For that reason we will focus on the (combinatorial) curve neighborhood of the identity element at some degree $\mathbf{d}$.

\subsection{Curve Neighborhood $\Gamma_{\mathbf{d}}(id)$ for Finite Degrees}\label{case of finite degrees}

Let $\mathbf{d}=(d_0,d_1,...,d_{n-1})\in Q_{\text{aff}}$. If $d_{0}=0$ then $\mathbf{d}\in  Q.$ If $d_{0}\neq 0$ but $d_{i}=0$ for some $i\neq 0$ then for an integer $1\leq k\leq n-1$ we will have  $\varphi^{k}(\mathbf{d})=\mathbf{d'}$ where $\mathbf{d'}=(d'_0,d'_1,...,d'_{n-1})$ such that $d'_{0}=0$ and $\varphi$ is the automorphism of the Dynkin diagram; see Section (\ref{dynkin}). So we will consider $\mathbf{d}$ as an element of the finite root lattice, $Q,$ whenever $d_{i}=0$ for some $i.$

\begin{thm}\label{thm13} Let $\mathbf{d}=(d_0,d_1,...,d_{n-1})$ be a degree such that $d_{i}=0$ for some $i.$ Then  $\Gamma_{\mathbf{d}} ( id)=\{z_{\mathbf{d}}\}.$
\end{thm}

\begin{proof} Here, we are in the finite case since $d_{i}=0$ for some $i$. We get the conclusion by the finite case, see  \cite{MB01}.   \end{proof}

\begin{example} Let $W_{\text{aff}}$ be the affine Weyl group associated to $A_{6}^{(1)}$ and $\mathbf{d}=(5,0,2,2,3,0,4).$ Then by Theorem \ref{thm13}, we get $\Gamma_{\mathbf{d}} ( id)=\{z_{\mathbf{d}}\}$ where $z_{\mathbf{d}}=s_{\alpha_{0}+\alpha_{6}}  s_{\alpha_{2}+\alpha_{3}+\alpha_{4}} s_{\alpha_{3}},$ see Example \ref{example12.01}. 

\end{example}

\begin{thm} \label{thm13.5} Let $\alpha \in \Pi_{\text{aff}}^{\text{re},\,+}$ such that $\alpha<c$. Then $\Gamma_{\alpha} (id)=\{s_{\alpha}\}$.

\end{thm}

\begin{proof} Let $\alpha=\sum_{i=0}^{n-1}a_{i}\alpha_{i} \in \Pi_{\text{aff}}^{\text{re},\,+}.$ Then $a_{i}=0$ for some $i$ since $\alpha<c$. So $\Gamma_{\alpha} (id)=\{z_{\alpha}\}$. Observe that $z_{\alpha}=s_{\alpha}$ by definition.     
\end{proof}

\begin{example} Assume that the affine Weyl group is associated to $A_{4}^{(1)}$ and $\alpha=\alpha_{0}+\alpha_{4}.$ Then $\Gamma_{\alpha}( id)=\{s_{\alpha_{0}+\alpha_{4}}\}$ by Theorem \ref{thm13.5}. 

\end{example}

\begin{cor} \label{cor13.6} Let $\beta \in \Pi_{\text{aff}}^{\text{re},\,+}$ such that $\beta<c$. If $\,\,\sum_{i=1}^{r}\beta_{i} \leq \beta$ for some $\beta_{i} \in \Pi_{\text{aff}}^{\text{re},\,+}$ then $s_{\beta_{1}}s_{\beta_{2}}...s_{\beta_{r}}\leq s_{\beta}.$

\end{cor}

\begin{proof} Let $w=s_{\beta_{1}}s_{\beta_{2}}...s_{\beta_{r}}$. Then $id \stackrel{\beta_{1}} \longrightarrow s_{\beta_{1}} \stackrel{\beta_{2}} \longrightarrow s_{\beta_{1}}s_{\beta_{2}} ...\stackrel{\beta_{r}} \longrightarrow w=s_{\beta_{1}}s_{\beta_{2}}...s_{\beta_{r}}$
is a path in the moment graph which has degree $\sum_{i=1}^{r}\beta_{i}\leq \beta$. Thus $w\leq s_{\beta}$ since $\Gamma_{\beta} ( id)=\{s_{\beta}\}$ by Theorem \ref{thm13.5}.

\end{proof}

\subsection{Curve Neighborhood $\Gamma_{c}(id)$} \label{case of c}

Here we will consider one of our basic yet crucial results. The result will play a key role in proving the rest of the results.

\begin{thm} \label{thm14}  We have 
\begin{enumerate}
\item[1)] $\Gamma_{c} ( id)=\{t_{\gamma}:\gamma \in \Pi \}$. 
\item[2)] $|\Gamma_{c}( id)|=n(n-1)$.
\item[3)] For all $w\in \Gamma_{c}( id)$, $\ell(w)=2(n-1)$.
\end{enumerate}
 
\end{thm}

\begin{proof} $1)$ Let $id \stackrel{\beta_{1}} \longrightarrow s_{\beta_{1}} \stackrel{\beta_{2}} \longrightarrow s_{\beta_{1}}s_{\beta_{2}} ...\stackrel{\beta_{r}} \longrightarrow w=s_{\beta_{1}}s_{\beta_{2}}...s_{\beta_{r}}$
be a path in the moment graph such that $\sum_{i=1}^{r} \beta_{i} \leq c$. Note that, by Remark \ref{remark12.9} we can assume that $w=s_{\beta_{1}} \cdot s_{\beta_{2}}\cdot...\cdot s_{\beta_{r}}.$ where $\sum_{i=1}^{r} \beta_{i} = c$ and $\beta_{i}<c$ for all $i$. Now, note that $\sum_{i=2}^{r}\beta_{i}=c-\beta_{1}$ is an affine positive real root. Moreover, $s_{\beta_{2}}\cdot s_{\beta_{3}}\cdot ...\cdot s_{\beta_{r}}\leq s_{c-\beta_{1}}$ by Corollary  \ref{cor5.9}. Thus $w=s_{\beta_{1}}\cdot (s_{\beta_{2}}\cdot ...\cdot s_{\beta_{r}}) \leq s_{\beta_{1}}\cdot s_{c-\beta_{1}}$. But by Lemma \ref{lemma11}, $s_{\beta_{1}}s_{c-\beta_{1}}$ is reduced so $s_{\beta_{1}}\cdot s_{c-\beta_{1}}= s_{\beta_{1}}s_{c-\beta_{1}}$. We also know that $s_{\beta_{1}}s_{c-\beta_{1}}=t_{\gamma}$ for some $\gamma \in \Pi,$ see the proof of Lemma \ref{lemma10}.
 
 $2)$ Let $\beta,\beta',\mu,\mu' \in \Pi_{\text{aff}}^{\text{re},\,+}$ such that $\mu+\mu'=\beta+\beta'=c.$ Then $s_{\beta}s_{\beta'}\neq s_{\mu}s_{\mu'}$ if $\beta\neq \mu$ by Lemma \ref{lemma10}. Also, note that  $|\{\beta \in \Pi_{\text{aff}}^{\text{re},\,+}: \beta<c\}|=n(n-1)$. 
 
$3)$ Let $w=s_{\beta}s_{\beta'} \in \Gamma_{c}( id)$ for some positive real roots, $\beta,\beta'$ such that $\beta+\beta'=c.$ Then $\ell(w)=2(n-1)$, by Lemma \ref{lemma11}.
\end{proof}

\begin{remark} \label{remark14.5} We have
\begin{equation}\label{equation14.5} \Gamma_{c} ( id)=\{s_{\beta}s_{\beta'}:\beta, \beta' \in   \Pi_{\text{aff}}^{\text{re},\,+}\, \text{such that}\,\, \beta+\beta'=c\},
\end{equation}
see the proof of Theorem \ref{thm14}.
\end{remark}

\begin{example} Suppose that the affine Weyl group $W_{\text{aff}}$ is associated to $A_{2}^{(1)}.$ Then $$\Gamma_{c}( id)=\{t_{\alpha_{1}},t_{\alpha_{2}},t_{\alpha_{1}+\alpha_{2}}, t_{-\alpha_{1}}, t_{-\alpha_{2}},t_{-(\alpha_{1}+\alpha_{2})}\}.$$ 
Moreover, for any $w\in \Gamma_{c}( id)$, $\ell(w)=2(n-1)=2\cdot 2=4$. 

\end{example}

\subsection{Curve Neighborhood $\Gamma_{c+\alpha}(id)$} \label{case of c plus alpha}

Here, we will state the result which inspires us to calculate the most general case. First, we will have some preliminary lemmas.

\begin{lemma}\label{lemma17.3} let $\alpha<c$ be a positive real root. Then $s_{\alpha}\cdot s_{\alpha}\cdot s_{c-\alpha}\leq s_{\alpha}\cdot s_{c-\alpha}\cdot s_{\alpha}$ and $s_{c-\alpha}\cdot s_{\alpha}\cdot s_{\alpha}\leq s_{\alpha}\cdot s_{c-\alpha}\cdot s_{\alpha}.$ 

\end{lemma}

\begin{proof} First, assume that $\alpha=p_{i,i}=\alpha_{i}$ for some $i.$ Then $s_{\alpha}=s_{i}$ and since $s_{i}\cdot s_{i}=s_{i}$ we get $s_{\alpha}\cdot s_{\alpha}\cdot s_{c-\alpha}=s_{i}\cdot s_{i}\cdot s_{c-\alpha}=s_{i}\cdot s_{c-\alpha}\leq s_{i}\cdot s_{c-\alpha}\cdot s_{i}=s_{\alpha}\cdot s_{c-\alpha}\cdot s_{\alpha}.$
Now, assume that $\alpha=p_{i,j}$ where $i\neq j.$ Then $s_{\alpha}=s_{i}\cdot s_{p_{i,j}-\alpha_{i}}\cdot s_{i}$ and $s_{p_{i,j}-\alpha_{i}}=s_{j}\cdot s_{p_{i,j}-\alpha_{i}-\alpha_{j}}\cdot s_{j}$. Note that $s_{\alpha}\cdot s_{i}=s_{\alpha}$ and similarly $s_{\alpha}\cdot s_{j}=s_{\alpha}.$ Also, by Lemma \ref{lemma5.7}, $s_{\alpha}$ Hecke commutes with $s_{p_{i,j}-\alpha_{i}}$ and with $s_{p_{i,j}-\alpha_{i}-\alpha_{j}}$ since $p_{i,j}-\alpha_{i}<\alpha$ and $p_{i,j}-\alpha_{i}-\alpha_{j}<\alpha$. We also have $s_{p_{i,j}-\alpha_{i}-\alpha_{j}}<s_{\alpha}$ by Corollary \ref{cor13.6}. Furthermore, $\text{supp}(p_{i,j}-\alpha_{i}-\alpha_{j})$ and $\text{supp}(c-p_{i,j})$ are disconnected so $s_{p_{i,j}-\alpha_{i}-\alpha_{j}}$ Hecke commutes with $s_{c-p_{i,j}}.$ Thus, 
$$\begin{array}{lll} s_{\alpha}\cdot s_{\alpha}\cdot s_{c-\alpha}&=&s_{\alpha}\cdot s_{i}\cdot s_{p_{i,j}-\alpha_{i}}\cdot s_{i}\cdot s_{c-\alpha}
=s_{\alpha}\cdot s_{p_{i,j}-\alpha_{i}}\cdot s_{i}\cdot s_{c-\alpha}\\
&=& s_{p_{i,j}-\alpha_{i}}\cdot s_{\alpha}\cdot s_{i}\cdot s_{c-\alpha}
=s_{p_{i,j}-\alpha_{i}}\cdot s_{\alpha}\cdot s_{c-\alpha}\\
&=&s_{j}\cdot s_{p_{i,j}-\alpha_{i}-\alpha_{j}}\cdot s_{j}\cdot s_{\alpha}\cdot s_{c-\alpha}
= s_{j}\cdot s_{p_{i,j}-\alpha_{i}-\alpha_{j}}\cdot s_{\alpha}\cdot s_{c-\alpha}\\
&=& s_{j}\cdot s_{\alpha}\cdot s_{p_{i,j}-\alpha_{i}-\alpha_{j}}\cdot s_{c-\alpha}
=s_{\alpha}\cdot s_{p_{i,j}-\alpha_{i}-\alpha_{j}}\cdot s_{c-\alpha} \\
&=& s_{\alpha}\cdot s_{c-\alpha}\cdot s_{p_{i,j}-\alpha_{i}-\alpha_{j}}
\leq  s_{\alpha}\cdot s_{c-\alpha}\cdot s_{\alpha}.\\
\end{array}$$
The proof of the inequality $s_{c-\alpha}\cdot s_{\alpha}\cdot s_{\alpha}\leq s_{\alpha}\cdot s_{c-\alpha}\cdot s_{\alpha}$ is similar. 
\end{proof}

\begin{lemma}\label{lemma17.71} Let $w=s_{\beta_{1}}\cdot s_{\beta_{2}}\cdot ...\cdot s_{\beta_{r}}$ where $\beta_{i}\in \Pi_{\text{aff}}^{\text{re},\,+}$ and $\beta_{i}<c$ for all $i.$ Also, assume that $\mathbf{d}=\sum_{i=1}^{r}\beta_{i}=\mathbf{d'}+\mathbf{d''}$ where $\mathbf{d'}$ is the biggest degree that $0<\mathbf{d'}\leq c$ and $\mathbf{d''}=\mathbf{d}-\mathbf{d'}>0$ may or may not be the degree of a root. Then there are some positive real roots, $\beta'_{1}, \beta'_{2},...,\beta'_{r'}$ such that $w\leq s_{\beta'_{1}}\cdot s_{\beta'_{2}}\cdot ...\cdot s_{\beta'_{r'}}$ where $ \sum_{i=1}^{t}\beta'_{i}=\mathbf{d'}$ and $ \sum_{i=t+1}^{r'}\beta'_{i}= \mathbf{d''}$ for some $t$ and $\beta'_{i}<c$ for all $i.$
\end{lemma}

\begin{proof}First, assume that there is an integer $b$ such that $\sum_{i=1}^{b}\beta_{i}=\mathbf{d'}.$ Then $ \sum_{i=b+1}^{r'}\beta'_{i}= \mathbf{d''}$ so we can take $t=b$ and $\beta_{i}=\beta'_{i}$ for all $i.$ Now, suppose that for any integer b such that $1\leq b \leq r$ we have $\sum_{i=1}^{b}\beta_{i}\neq \mathbf{d'}.$ We will prove the statement by induction on $\mathbf{d}.$ Here, for the base case we have $\mathbf{d}=2\alpha_{i},$ for some simple root $\alpha_{i}$ where $\mathbf{d'}=\mathbf{d''}=\alpha_{i}$. Then, we have to have $w=s_{\alpha_{i}}\cdot s_{\alpha_{i}}$ which satisfies the statement, so we are done with the base case. Now assume that the statement is true for all degrees that are smaller than $\mathbf{d}$. We will prove that it is also true for $\mathbf{d}.$ 

\begin{itemize}

\item[1)] Suppose that $\beta_{r}<\mathbf{d''}$. Then since $\mathbf{d}-\beta_{r}<\mathbf{d}$ and $\mathbf{d}-\beta_{r}=\mathbf{d'}+\mathbf{d''}-\beta_{r}$ where $\mathbf{d''}-\beta_{r}>0$ by the induction assumption we have some real positive roots, $\beta'_{1}, \beta'_{2},...,\beta'_{r'},$ such that 
$w=(s_{\beta_{1}}\cdot s_{\beta_{2}}\cdot ...\cdot s_{\beta_{r-1}})\cdot s_{\beta_{r}}\leq (s_{\beta'_{1}}\cdot s_{\beta'_{2}}\cdot ...\cdot s_{\beta'_{r'}})\cdot s_{\beta_{r}}$ where $ \sum_{i=1}^{t}\beta'_{i}=\mathbf{d'}$ and $ \sum_{i=t+1}^{r'}\beta'_{i}= \mathbf{d''}-\beta_{r}$ for some $t$ and $\beta'_{i}<c$ for all $i.$ This implies that we are done with this case since $ \sum_{i=1}^{t}\beta'_{i}=\mathbf{d'}$ and $( \sum_{i=t+1}^{r'}\beta'_{i})+\beta_{r}= \mathbf{d''}.$ 

\item[2)] Assume that $\beta_{1}<\mathbf{d''}$. Then again since  $\mathbf{d}-\beta_{1}<\mathbf{d}$ and $\mathbf{d}-\beta_{1}=\mathbf{d'}+\mathbf{d''}-\beta_{1}$ where $\mathbf{d''}-\beta_{1}>0$ by induction assumption we have some real positive roots, $\beta'_{1}, \beta'_{2},...,\beta'_{r'}$ such that 
$w=s_{\beta_{1}}\cdot (s_{\beta_{2}}\cdot ...\cdot s_{\beta_{r-1}}\cdot s_{\beta_{r}} )\leq s_{\beta_{1}}\cdot (s_{\beta'_{1}}\cdot s_{\beta'_{2}}\cdot ...\cdot s_{\beta'_{r'}})$ where $ \sum_{i=1}^{t}\beta'_{i}=\mathbf{d'}$ and $ \sum_{i=t+1}^{r'}\beta'_{i}= \mathbf{d''}-\beta_{1}$ for some $t$ and $\beta'_{i}<c$ for all $i.$ Now, observe that $\beta'_{r'}<\mathbf{d''}-\beta_{1}<\mathbf{d''}.$ Thus, this case follows by the previous case. 

\item[3)] Suppose that $\beta_{1} \cap \mathbf{d''}=\emptyset$. Then the statement will follow by the same arguments in the previous cases.

\item[4)] Assume that $\beta_{i}\cap \beta_{j}\neq \emptyset$ for some $i,j$ such that $2\leq i<j\leq r.$  Then  $\sum_{i=2}^{r}\beta_{i}=\mathbf{d}-\beta_{1}<\mathbf{d}$ can be written as $\mathbf{D}+\mathbf{D'}$ where $\mathbf{D'}$ is the biggest degree such that $0<\mathbf{D'}\leq \mathbf{d'}\leq  c$ and $0<\mathbf{D''}\leq \mathbf{d''}.$ Thus by the induction assumption there are some real positive roots, $\beta'_{1}, \beta'_{2},...,\beta'_{r'},$ such that 
$s_{\beta_{2}}\cdot s_{\beta_{3}}\cdot ...\cdot s_{\beta_{r}} \leq s_{\beta'_{1}}\cdot s_{\beta'_{2}}\cdot ...\cdot s_{\beta'_{r'}}$ where $ \sum_{i=1}^{t}\beta'_{i}=\mathbf{D'}$ and $ \sum_{i=t+1}^{r'}\beta'_{i}= \mathbf{D''}$ for some $t$ and $\beta'_{i}<c$ for all $i.$ This follows by
$w=s_{\beta_{1}}\cdot (s_{\beta_{2}}\cdot ...\cdot s_{\beta_{r}} )\leq s_{\beta_{1}} \cdot (s_{\beta'_{1}}\cdot s_{\beta'_{2}}\cdot ...\cdot s_{\beta'_{r'}}).$ 
Note that
$\beta'_{r'}\leq \mathbf{D''}\leq \mathbf{d''}$ thus we are done with this case.

\item[5)] Suppose that $\beta_{i}\cap \beta_{j}=\emptyset$ for any $i,j$ such that $2\leq i<j\leq r$. This implies that $\sum_{i=2}^{r}\beta_{i}\leq c$ so $\mathbf{d''}\leq \beta_{1}$ and $\mathbf{d}<2c$. We have two main cases here:\\

\begin{itemize}

\item[5a)] Assume that $\sum_{i=2}^{r}\beta_{i}< c.$ Now, for any $i$ such that $2\leq i\leq r-1$ if $\beta_{i}+\beta_{i+1}$ is a root then by Lemma \ref{lemma5.7}  we have $s_{\beta_{i}}\cdot s_{\beta_{i+1}}\leq s_{\beta_{i}+\beta_{i+1}}.$ So if we repeatedly combine consecutive roots in $s_{\beta_{2}}\cdot ...\cdot s_{\beta_{r}}$ which add up to a root we can assume that no two consecutive roots in the multiplication add up to a root. By the same lemma, this implies that any consecutive roots of the multiplication Hecke commute. So, we can assume that for any $i,j$ such that $2\leq i,j\leq r$,  $s_{\beta_{i}}\cdot s_{\beta_{j}}=s_{\beta_{j}}\cdot s_{\beta_{i}}.$ Thus, we can also suppose that $\beta_{1}\cap \beta_{2}\neq \emptyset.$ Now, if $\beta_{1}\cap \beta_{2}\leq \mathbf{d''}$ is not a root and $\beta_{1}\cap \beta_{2}=\gamma_{1}+\gamma_{2}$ for some positive real roots, $\gamma_{1}, \gamma_{2}$ we will have $s_{\beta_{1}}\cdot s_{\beta_{2}}=s_{\gamma_{2}}\cdot s_{\gamma_{1}}\cdot s_{\beta_{1}-\gamma_{1}-\gamma_{2}}\cdot s_{\beta_{2}}$ by Lemma \ref{lemma5.7}. So we get
$$w=(s_{\beta_{1}}\cdot s_{\beta_{2}})\cdot s_{\beta_{3}}\cdot ...\cdot s_{\beta_{r}}\leq (s_{\gamma_{2}}\cdot s_{\gamma_{1}}\cdot s_{\beta_{1}-\gamma_{1}-\gamma_{2}}\cdot s_{\beta_{2}})\cdot s_{\beta_{3}}\cdot...\cdot s_{\beta_{r}}.$$ Note that $\gamma_{2}<\mathbf{d''}$ so we are done with this case by case $2)$ above. Now assume that $\beta_{1}\cap \beta_{2}\leq \mathbf{d''}$ is a root. Let $\gamma:=\beta_{1}\cap \beta_{2}$. Then by Lemma \ref{lemma5.7} we have $s_{\beta_{1}}\cdot s_{\beta_{2}}=s_{\gamma}\cdot s_{\beta_{1}-\gamma}\cdot s_{\beta_{2}}.$ So
$w=(s_{\beta_{1}}\cdot s_{\beta_{2}})\cdot s_{\beta_{3}}\cdot ...\cdot s_{\beta_{r}}\leq (s_{\gamma}\cdot s_{\beta_{1}-\gamma}\cdot s_{\beta_{2}})\cdot s_{\beta_{3}}\cdot...\cdot s_{\beta_{r}}.$ Now, if $\gamma <\mathbf{d''}$ then again we are done by case $2)$ above. Assume that $\gamma=\mathbf{d''}$ which implies that for any $i\geq 3$ we have both $\beta_{1}\cap \beta_{i}=\emptyset$ and $\mathbf{d''}\cap \beta_{i}=\emptyset.$ Here, we will consider the multiplication, $s_{\beta_{1}}\cdot s_{\beta_{2}}\cdot s_{\beta_{3}}\cdot ...\cdot s_{\beta_{r}}.$ By Hecke commutation, we can assume that $s_{\beta_{2}}$ is the last reflection of the multiplication. Now, if the sum of $\beta_{1}$ and $\beta_{3}$ is not a root then they will Hecke commute by Lemma \ref{lemma5.7} and we will have $w=s_{\beta_{1}}\cdot s_{\beta_{3}}\cdot s_{\beta_{4}}\cdot ...\cdot s_{\beta_{r}}\cdot s_{\beta_{2}}=s_{\beta_{3}}\cdot s_{\beta_{1}}\cdot s_{\beta_{4}}\cdot ...\cdot s_{\beta_{r}}\cdot s_{\beta_{2}}.$ But this implies that we are done with this case by case $3)$ above since $\beta_{3}\cap \mathbf{d''}=\emptyset.$ Now, if the sum is a root then we will replace $s_{\beta_{1}}\cdot s_{\beta_{3}}$ with $s_{\beta_{1}+\beta_{3}}$ by the same lemma to get $w=s_{\beta_{1}}\cdot s_{\beta_{3}}\cdot s_{\beta_{4}}\cdot ...\cdot s_{\beta_{r}}\cdot s_{\beta_{2}}\leq s_{\beta_{1}+\beta_{3}}\cdot s_{\beta_{4}}\cdot ...\cdot s_{\beta_{r}}\cdot s_{\beta_{2}}.$ Next, we will apply the same argument to $\beta_{1}+\beta_{3}$ and $\beta_{4}$ etc. So either we will be done by case $3)$ above that is because $w$ can be written in a way that the first root which appears in it is $s_{\beta_{i}}$ such that $\beta_{i}\cap \mathbf{d''}=\emptyset$ or we will get $w=s_{\beta_{1}}\cdot s_{\beta_{3}}\cdot s_{\beta_{4}}\cdot ...\cdot s_{\beta_{r}}\cdot s_{\beta_{2}}\leq s_{\beta_{1}+\beta_{3}+\beta_{4}+...+\beta_{r}}\cdot s_{\beta_{2}}.$ Last, note that, $(\beta_{1}+\beta_{3}+\beta_{4}+...+\beta_{r})\cap  \beta_{2}=\gamma=\mathbf{d''}.$ By Lemma \ref{lemma5.7} we can write $s_{\beta_{1}+\beta_{3}+\beta_{4}+...+\beta_{r}}\cdot s_{\beta_{2}}=s_{\beta_{1}+\beta_{3}+\beta_{4}+...+\beta_{r}}\cdot s_{\beta_{2}-\gamma}\cdot s_{\gamma}.$ Hence, we have done with this case as well.

\item[5b)] Suppose that $\sum_{i=2}^{r}\beta_{i}= c$ then $\mathbf{d'}=c$ and $\beta_{1}=\mathbf{d''}.$ Now, note that, by Equation \ref{equation14.5} there is a positive real root, $\beta<c$ such that $s_{\beta_{2}}\cdot s_{\beta_{3}}\cdot ...\cdot s_{\beta_{r}}\leq s_{\beta}\cdot s_{c-\beta}.$ So we get 
$w=s_{\beta_{1}}\cdot (s_{\beta_{2}}\cdot ...\cdot s_{\beta_{r}})\leq s_{\beta_{1}}\cdot s_{\beta}\cdot s_{c-\beta}.$
Here, we can assume that $w=s_{\beta_{1}}\cdot s_{\beta}\cdot s_{c-\beta}$;

\begin{itemize}

\item[i)] If $\beta_{1}\cap \beta=\emptyset$ then $\beta_{1}\leq c-\beta.$ Now, first assume that $\beta_{1}+\beta=c$ then $c-\beta=\beta_{1}=\mathbf{d''}$ and this implies that $w$ is already in the desired form. Second, if $\beta_{1}+\beta<c$ is not a root then by Lemma \ref{lemma5.7} we have $s_{\beta_{1}}\cdot s_{\beta}=s_{\beta}\cdot s_{\beta_{1}}$. Also, by the same lemma we have $s_{\beta_{1}}\cdot s_{c-\beta}=s_{c-\beta}\cdot s_{\beta_{1}}$ since $\beta_{1}\leq c-\beta.$ Hence $w= s_{\beta_{1}}\cdot s_{\beta}\cdot s_{c-\beta}=s_{\beta}\cdot s_{c-\beta}\cdot s_{\beta_{1}}.$So we are done with this case as well. Last, suppose that $\beta_{1}+\beta<c$ is a root then by Lemma \ref{lemma5.7} we have $s_{\beta_{1}}\cdot s_{\beta}\leq s_{\beta_{1}+\beta}.$ So we get 
$w= s_{\beta_{1}}\cdot s_{\beta}\cdot s_{c-\beta}\leq s_{\beta_{1}+\beta}\cdot s_{c-\beta}.$ Now, note that $(\beta_{1}+\beta)\cap (c-\beta)=\beta_{1}$ and by the same lemma we have $s_{\beta_{1}+\beta}\cdot s_{c-\beta} =s_{\beta_{1}+\beta}\cdot s_{c-\beta-\beta_{1}}\cdot s_{\beta_{1}}.$ Thus the statement follows.

\item[ii)] If $\beta_{1}\cap \beta\neq \emptyset$ then we have several cases here; first, suppose if $\beta_{1}=\beta$ then by Lemma \ref{lemma17.3} we have $w=s_{\beta_{1}}\cdot s_{\beta}\cdot s_{c-\beta}=s_{\beta_{1}}\cdot s_{\beta_{1}}\cdot s_{c-\beta_{1}}\leq s_{\beta_{1}}\cdot s_{c-\beta_{1}}\cdot s_{\beta_{1}}.$ Thus, we are done with case. Next, if $\beta_{1}\neq \beta$ then either $\beta_{1}<\beta$, $\beta_{1}>\beta$ or $\beta_{1}$ is incomparable with $\beta.$ Now, if $\beta_{1}<\beta$ or $\beta_{1}>\beta$ then $s_{\beta_{1}}$ and $s_{\beta}$ will Hecke commute by Lemma \ref{lemma5.7} so have $w=s_{\beta_{1}}\cdot s_{\beta}\cdot s_{c-\beta}=s_{\beta}\cdot s_{\beta_{1}}\cdot s_{c-\beta}.$ Now, if $\beta_{1}<\beta$ then $\beta_{1}+c-\beta<c$ so we are done with this case by case $5a)$ above. Now, if $\beta_{1}>\beta$ then $\beta_{1}\cap (c-\beta)\neq \emptyset$ so the statement follows by case $4)$ above. Last, assume that $\beta_{1}$ is incomparable with $\beta.$ Then $\beta_{1}$ intersects with both $\beta$ and $c-\beta$ which implies that $\beta_{1}\cap \beta<\mathbf{d''}.$ Then by Lemma \ref{lemma5.7} we can replace the multiplication, $s_{\beta_{1}}\cdot s_{\beta}$ in $w$ with $s_{\gamma_{2}}\cdot s_{\gamma_{1}}\cdot s_{\beta_{1}-\gamma_{1}-\gamma_{2}}\cdot s_{\beta}$ if $\beta_{1}\cap \beta$ is not a root and $\beta_{1}\cap \beta=\gamma_{1}+\gamma_{2}$ for some positive real roots, $\gamma_{1},\gamma_{2}$ and with $s_{\gamma}\cdot s_{\beta_{1}-\gamma}\cdot s_{\beta}$ if $\beta_{1}\cap \beta$ is a root and $\beta_{1}\cap \beta=\gamma.$ So, the first root that appears in the permutation $w$ is either $\gamma_{2}$ or $\gamma$ which are strictly smaller than $\mathbf{d''}.$ But then the statement follows by case $2)$ above.

\end{itemize}

\end{itemize}

\end{itemize} \end{proof}

\begin{lemma}\label{lemma17.8} Let $\mu,\mu'$ and $\alpha \in \Pi_{\text{aff}}^{\text{re},\,+}$ such that $\mu+\mu'=c$ and $\alpha<c.$ Then there are some positive real roots $\beta$ and $\beta'$ such that $s_{\mu}\cdot s_{\mu'}\cdot s_{\alpha}\leq s_{\beta}\cdot s_{\beta'}\cdot s_{\alpha}$ where $\beta+\beta'=c$ and either $\beta'\leq c-\alpha$ or both $\beta'>\alpha$ and $\beta'\perp \alpha.$
\end{lemma}

\begin{proof} We have three cases;
\begin{itemize}
\item[1)] If $\mu' \leq c-\alpha$ then take $\beta'=\mu'.$
\item[2)] If $\mu'>c-\alpha$ then $\mu' \cap \alpha \neq \emptyset.$ Note that, $\mu'\cap \alpha\neq \mu', \alpha.$ We have two cases here;

\begin{itemize}

\item[a)] If $\mu' \cap \alpha$ is a root then let $\gamma:=\mu' \cap \alpha$. Now, by Lemma \ref{lemma5.7} we get $s_{\mu}\cdot s_{\mu'}\cdot s_{\alpha}=s_{\mu}\cdot s_{\gamma}\cdot s_{\mu'-\gamma}\cdot s_{\alpha}.$
Here, note that $\mu+\gamma$ is a root since $\mu+\gamma=c-\mu'+\gamma=c-(\mu'-\gamma)$ and $\mu'-\gamma$ is a root. Thus $s_{\mu}\cdot s_{\mu'}\cdot s_{\alpha}=s_{\mu}\cdot s_{\gamma}\cdot s_{\mu'-\gamma}\cdot s_{\alpha}\leq s_{\mu+\gamma}\cdot s_{\mu'-\gamma}\cdot s_{\alpha}.$
Note that, $\mu'-\gamma=c-\alpha$ and $(\mu+\gamma)+(\mu'-\gamma)=c$ so we will take $\beta'=\mu'-\gamma$ in this case.

\item[b)] If $\mu' \cap \alpha$ is not a root then by Lemma \ref{lemma5.7} we have  $\mu' \cap \alpha=\gamma_{1}+\gamma_{2}$ for some positive real roots $\gamma_{1}$ and $\gamma_{2}$ and $s_{\mu}\cdot s_{\mu'}\cdot s_{\alpha}=s_{\mu}\cdot s_{\gamma_{1}}\cdot s_{\gamma_{2}}\cdot s_{\mu'-\gamma}\cdot s_{\alpha}.$ 
Again, here $\mu+\gamma$ is a root since both $\mu+\alpha=c-\mu'+\gamma=c-(\mu'-\gamma)$ and $\mu'-\gamma$ are roots. Thus 
$s_{\mu}\cdot s_{\mu'}\cdot s_{\alpha}=s_{\mu}\cdot s_{\gamma_{1}}\cdot s_{\gamma_{2}}\cdot s_{\mu'-\gamma}\cdot s_{\alpha}\leq s_{\mu+\gamma_{1}+\gamma_{2}}\cdot s_{\mu'-\gamma}\cdot s_{\alpha}.$
Now, observe that, $\mu'-\gamma=c-\alpha$ and $(\mu+\gamma_{1}+\gamma_{2})+(\mu'-\gamma)=c$ so again we will take $\beta'=\mu'-\gamma$ in this case.

\end{itemize}

\item[3)] If $\mu'$ is incomparable with $c-\alpha$ then again $\mu' \cap \alpha \neq \emptyset.$ But then we will either have the same two cases in the second case which then we will be done by the arguments in the case or $\mu' \leq \alpha$ or $\mu'> \alpha.$ 
\begin{enumerate}
\item[a)] Assume that $\mu'\leq \alpha.$ If $\mu'=\alpha$ then $s_{\mu}\cdot s_{\mu'}\cdot s_{\alpha} =s_{c-\alpha}\cdot s_{\alpha}\cdot s_{\alpha}.$ Now, by Lemma \ref{lemma17.3}, we have $s_{c-\alpha}\cdot s_{\alpha}\cdot s_{\alpha} \leq s_{\alpha}\cdot s_{c-\alpha}\cdot s_{\alpha}$. So we can take $\beta'=c-\alpha$ in this case. Now, suppose that $\mu'< \alpha.$ Then $s_{\mu'}$ and $s_{\alpha}$ Hecke commute by Lemma \ref{lemma5.7} part $a).$ Also,  $\mu \cap \alpha \neq \emptyset$. Note that the intersection of two positive roots is either a root or the sum of two positive roots. Now, assume that  $\mu \cap \alpha$ is a root. Let  $\gamma:=\mu \cap \alpha.$ Then again by Lemma \ref{lemma5.7} we have $s_{\mu}\cdot s_{\alpha}=s_{\gamma}\cdot s_{\mu-\gamma}\cdot s_{\alpha}.$ Thus $s_{\mu}\cdot s_{\mu'}\cdot s_{\alpha} =s_{\mu}\cdot s_{\alpha}\cdot s_{\mu'}=s_{\gamma}\cdot s_{\mu-\gamma}\cdot s_{\alpha}\cdot s_{\mu'}=s_{\gamma}\cdot s_{\mu-\gamma}\cdot s_{\mu'}\cdot s_{\alpha}.$ Here, note that $\mu+\mu'-\gamma=c-\gamma$ so $s_{\mu-\gamma}\cdot s_{\mu'}\leq s_{c-\gamma}$ by Corollary \ref{cor13.6}. So $s_{\gamma}\cdot s_{\mu-\gamma}\cdot s_{\mu'}\cdot s_{\alpha}\leq s_{\gamma}\cdot s_{c-\gamma}\cdot s_{\alpha}.$ Note that, $c-\gamma>c-\alpha$, hence this case falls into the case $2)$ above. So we are done with this case. If $\mu \cap \alpha$ is not a root then the proof is similar.

\item[b)] Assume that $\mu'> \alpha$. Now, if $\mu' \perp \alpha$ then we can take $\beta'=\mu'.$ Now, assume that $\mu'$ is not perpendicular to $\alpha.$ Then $<\mu',\alpha>=1$ which implies $s_{\alpha}(\mu')=\mu'-<\mu',\alpha^{\vee}>\alpha=\mu'-\alpha.$ So $\mu'-\alpha$ is a root. Also, $s_{\mu'}$ and $s_{\alpha}$ Hecke commute by Lemma \ref{lemma5.7}. Furthermore, $\mu+\alpha=c-\mu'+\alpha=c-(\mu'-\alpha)$ is also a root and $(\mu+\alpha)\cap \mu'=\alpha$ since $\mu$ and $\mu'$ are disjoint and $\mu'>\alpha.$ We also have $s_{\mu}\cdot s_{\alpha}\leq s_{\mu+\alpha}$ by Corollary \ref{cor13.6}. Thus  $s_{\mu}\cdot s_{\mu'}\cdot s_{\alpha} =s_{\mu}\cdot s_{\alpha}\cdot s_{\mu'}\leq s_{\mu+\alpha}\cdot s_{\mu'}.$ Here, note that $s_{\mu+\alpha}\cdot s_{\mu'}=s_{\mu+\alpha}\cdot s_{\mu'-\alpha}\cdot s_{\alpha}$ by case $2)$ in Lemma \ref{lemma5.7}. Now, observe that $(\mu+\alpha)+(\mu'-\alpha)=c$ and $\mu'-\alpha\leq c-\alpha.$ Thus, we can take $\beta'=\mu'-\alpha$ in this case.

\end{enumerate}
\end{itemize} \end{proof}

We will state the result in the following theorem.

\begin{thm}\label{thm17.9} Let $\alpha\in \Pi_{\text{aff}}^{\text{re},\,+}$ be such that $\alpha<c.$ Then

\begin{enumerate}
\item[1)] $\Gamma_{c+\alpha}( id)=\{t_{\beta'} s_{\alpha}:\beta'\in \Pi_{\text{aff}}^{\text{re},\,+}(\alpha)\} \cup \{t_{\beta'-c} s_{\alpha}:\beta'\in \Pi_{\text{aff}}^{\text{re},\,+}(\alpha)\}$
\item[2)] $|\Gamma_{c+\alpha}( id)|=|\{\beta': \beta' \in \Pi_{\text{aff}}^{\text{re},\,+}(\alpha)\}|$
\item[3)] For all $w\in \Gamma_{c+\alpha}( id)$, $\ell(w)=2(n-1)+\ell{(s_{\alpha})}$

\end{enumerate}

where $ \Pi_{\text{aff}}^{\text{re},\,+}(\alpha):=\{\beta' \in \Pi_{\text{aff}}^{\text{re},\,+}: \beta'<c \,\,\text{and either}\,\,  \beta'\leq c-\alpha \,\, \text{or both} \,\,\beta'>\alpha \,\,\text{and} \,\, \beta'\perp \alpha \}$.

\end{thm}

\begin{proof} $1)$ Let $id \stackrel{\beta_{1}} \longrightarrow s_{\beta_{1}} \stackrel{\beta_{2}} \longrightarrow s_{\beta_{1}}s_{\beta_{2}} ...\stackrel{\beta_{r}} \longrightarrow w=s_{\beta_{1}}s_{\beta_{2}}...s_{\beta_{r}}$
be a path in the moment graph such that $\sum_{i=1}^{r} \beta_{i} \leq c+\alpha$. Note that, by Remark \ref{remark12.9} we can assume that $w=s_{\beta_{1}} \cdot s_{\beta_{2}}\cdot...\cdot s_{\beta_{r}}.$ where $\sum_{i=1}^{r} \beta_{i} = c+\alpha$ and $\beta_{i}<c$ for all $i$. Now, by Lemma \ref{lemma17.71}, we can also suppose that $ \sum_{i=1}^{t}\beta_{i}=c$ and $ \sum_{i=t+1}^{r}\beta_{i}= \alpha$ for some $t.$ Then  
$$w= (s_{\beta_{1}}\cdot...\cdot s_{\beta_{t}})\cdot(s_{\beta_{t+1}}\cdot s_{\beta_{t+2}}\cdot ...\cdot s_{\beta_{r}})\leq (s_{\mu}\,s_{\mu'})\cdot s_{\alpha}=s_{\mu} \cdot s_{\mu'}\cdot s_{\alpha}$$  for some $\mu, \mu' \in \Pi_{\text{aff}}^{\text{re},\,+}$ such that $\mu+\mu'=c$ by Equation \ref{equation14.5} and Corollary  \ref{cor5.9}. Furthermore, $s_{\mu}\cdot s_{\mu'}\cdot s_{\alpha}\leq s_{\beta}\cdot s_{\beta'}\cdot s_{\alpha}$ where $\beta+\beta'=c$ and either $\beta'\leq c-\alpha$ or both $\beta'>\alpha$ and $\beta'\perp \alpha,$ by Lemma \ref{lemma17.8}. Now, by Corollary \ref{cor12.8} $s_{\beta} s_{\beta'}s_{\alpha}$ is reduced so $s_{\beta}\cdot s_{\beta'}\cdot s_{\alpha}=s_{\beta} s_{\beta'}s_{\alpha}.$ Now, also note that $s_{\beta}s_{\beta'}=t_{\beta'}$ if $\alpha_{0}\notin  \text{supp}(\beta')$ and $s_{\beta}s_{\beta'}=t_{\beta'-c}$ if $\alpha_{0}\in \text{supp}(\beta')$, see the proof of Lemma \ref{lemma10}.

$2)$ Let $s_{\beta}s_{\beta'}s_{\alpha}$ and  $s_{\nu}s_{\nu'}s_{\alpha}\in \Gamma_{c+\beta}( id)$ be given. We need to show that $s_{\beta}s_{\beta'}s_{\alpha}\neq s_{\nu}s_{\nu'}s_{\alpha}$ if $\beta'\neq \nu'.$ Now assume that $\beta'\neq \nu'$ but $s_{\beta}s_{\beta'}s_{\alpha}= s_{\nu}s_{\nu'}s_{\alpha}$. By Corollary \ref{cor12.8}, both $s_{\beta}s_{\beta'}s_{\alpha}$ and $s_{\nu}s_{\nu'}s_{\alpha}$ are reduced so the equality, $s_{\beta}s_{\beta'}s_{\alpha}= s_{\nu}s_{\nu'}s_{\alpha}$ implies that $s_{\beta}s_{\beta'}= s_{\nu}s_{\nu'}$ but this is a contradiction by Lemma \ref{lemma10}. 
 
$3)$ Let $w=s_{\beta}s_{\beta'} s_{\alpha}\in \Gamma_{c+\alpha}( id).$ Then by Corollary \ref{cor12.8}, $\ell(w)=2(n-1)+\ell{(s_{\alpha})}$.
\end{proof}

\begin{remark} \label{remark17.91} Let $\alpha\in \Pi_{\text{aff}}^{\text{re},\,+}$ be such that $\alpha<c.$ We have
\begin{equation}\label{equation17.92} \Gamma_{c+\alpha} (id)=\{s_{\beta}s_{\beta'}s_{\alpha}:\beta' \in  \Pi_{\text{aff}}^{\text{re},\,+}(\alpha)\,\, \text{such that}\,\, \beta+\beta'=c\},
\end{equation}
see the proof of Theorem \ref{thm17.9}.
\end{remark}

\begin{example} \label{example17.93} Let $W_{\text{aff}}$ be the Weyl group of type $A_{4}^{(1)}.$ We compute $\Gamma_{c+\alpha}( id)$ where $\alpha=\alpha_{0}+\alpha_{4}.$ We have six positive real roots which are smaller than $c-\alpha=\alpha_{1}+\alpha_{2}+\alpha_{3};$ $\beta'_{1}=\alpha_{1}, \beta'_{2}=\alpha_{2}, \beta'_{3}=\alpha_{3}, \beta'_{4}=\alpha_{1}+\alpha_{2}, \beta'_{5}=\alpha_{2}+\alpha_{3} , \beta'_{6}=\alpha_{1}+\alpha_{2}+\alpha_{3}$ and only one positive real root which  is smaller than $c$, strictly bigger than $\alpha$ and perpendicular to $\alpha$; $\beta'_{7}=\alpha_{0}+\alpha_{1}+\alpha_{3}+\alpha_{4}.$ So $ \Gamma_{c+\alpha}( id)=\{t_{\beta'_{1}}s_{\alpha}, t_{\beta'_{2}}s_{\alpha},t_{\beta'_{3}}s_{\alpha},t_{\beta'_{4}}s_{\alpha},t_{\beta'_{5}}s_{\alpha},t_{\beta'_{6}}s_{\alpha},t_{\beta'_{7}-c}s_{\alpha} \}$ by Theorem \ref{thm17.9}.
Moreover, for any $w\in \Gamma_{c+\alpha}( id),$ we have $\ell(w)=2(n-1)+\ell{(s_{\alpha})}=2(n-1)+2\, \text{supp}(\alpha)-1=2\cdot4+2\cdot2-1=11.$

\end{example}

\subsection{Curve Neighborhood $\Gamma_{mc+\alpha}(id)$} \label{case of mc plus alpha}

This section is devoted to a generalization of the result in the previous section. We will begin with some lemmas.

\begin{lemma}\label{lemma17.95} Let $\beta,\beta',\nu $ and $\nu' \in \Pi_{\text{aff}}^{\text{re},\,+}$ such that $\nu+\nu'=\beta+\beta'=c$ and $\nu'\leq \beta'.$ Then  $s_{\nu}\cdot s_{\nu'}\cdot s_{\beta}\cdot s_{\beta'}\leq (s_{\nu}\cdot s_{\nu'})^2.$
\end{lemma}

\begin{proof} Here we have two main cases;
\begin{itemize}
\item[1)] $\nu'+\beta$ is a root. Then $c-(\nu'+\beta)=c-\nu'-\beta=\nu-\beta=\beta'-\nu'$ is also a root. We have $s_{\nu'}\cdot s_{\beta}\leq s_{\nu'+\beta}$ and $(\nu'+\beta)\cap \beta'=\nu'$ so by Lemma \ref{lemma5.7} part $2a)$ we get 
$$s_{\nu}\cdot s_{\nu'}\cdot s_{\beta}\cdot s_{\beta'}\leq s_{\nu}\cdot s_{\nu'+\beta}\cdot s_{\beta'}\leq s_{\nu}\cdot s_{\nu'+\beta}\cdot s_{\beta'-\nu'}\cdot s_{\nu'}.$$ Also, we have $\nu \cap (\nu'+\beta)=\beta$ and $\beta+\beta'-\nu'=c-\nu'=\nu$. Then by Lemma \ref{lemma5.7} part $2a)$ we get 
$s_{\nu}\cdot s_{\nu'+\beta}\cdot s_{\beta'-\nu'}\cdot s_{\nu'}\leq s_{\nu}\cdot s_{\nu'}\cdot s_{\beta}\cdot s_{\beta'-\nu'}\cdot s_{\nu'}\leq s_{\nu}\cdot s_{\nu'}\cdot s_{\nu}\cdot s_{\nu'}.$

\item[2)] $\nu'+\beta$ is not a root. Note that, since $\nu'+\beta<c$ and $\beta<\nu$ we have $s_{\nu'} \cdot s_{\beta}=s_{\beta}\cdot s_{\nu'}$ and $s_{\nu} \cdot s_{\beta}=s_{\beta} \cdot s_{\nu}$ which implies that 
$s_{\nu}\cdot s_{\nu'}\cdot s_{\beta}\cdot s_{\beta'}=s_{\beta}\cdot s_{\nu} \cdot s_{\nu'}\cdot s_{\beta'}.$ Also note that, $\nu'+\beta=c-\nu+c-\beta'=2c-(\nu+\beta')$ and since $\nu'+\beta$ is not a root $\nu+\beta'$ is also not a root. Then $s_{\nu} \cdot s_{\beta'}=s_{\beta'} \cdot s_{\nu}$. Furthermore, $s_{\nu'} \cdot s_{\beta'}=s_{\beta'} \cdot s_{\nu'}$ since $\nu'<\beta'$. We get 
$s_{\beta}\cdot s_{\nu} \cdot s_{\nu'}\cdot s_{\beta'}=s_{\beta}\cdot s_{\nu} \cdot s_{\beta'}\cdot s_{\nu'}.$ Note that, $\nu \cap \beta' \neq \emptyset.$ Here we have two cases;

\begin{itemize}
\item[a)] If $\gamma:=\nu \cap \beta'$ is a root then $\nu-\gamma$ and $\beta'-\gamma$ are also roots. By Lemma \ref{lemma5.7} we get $s_{\beta} \cdot s_{\nu} \cdot s_{\beta'} \cdot s_{\nu'}=s_{\beta} \cdot s_{\nu} \cdot s_{\beta'-\gamma}\cdot s_{\gamma} \cdot s_{\nu'}.$ Also note that, $\nu'=\beta'-\gamma$ and $\beta+\gamma=c-\beta'+\gamma=c-(\beta'-\gamma)=c-\nu'=\nu.$ Again, if we also use the facts that $s_{\beta} \cdot s_{\nu'}=s_{\nu'} \cdot s_{\beta}$ and $s_{\beta} \cdot s_{\nu} =s_{\nu} \cdot s_{\beta}$ we will obtain 

$\begin{array}{lll}s_{\beta} \cdot s_{\nu} \cdot s_{\beta'-\gamma} \cdot s_{\gamma} \cdot s_{\nu'}&=&s_{\beta} \cdot s_{\nu} \cdot s_{\nu'} \cdot s_{\gamma} \cdot s_{\nu'}
=s_{\nu} \cdot s_{\beta} \cdot s_{\nu'} \cdot s_{\gamma} \cdot s_{\nu'}\\
&=&s_{\nu} \cdot s_{\nu'} \cdot s_{\beta} \cdot s_{\gamma} \cdot s_{\nu'} \leq s_{\nu} \cdot s_{\nu'} \cdot s_{\nu} \cdot s_{\nu'}

\end{array}$ 

since $s_{\beta}\cdot s_{\gamma}\leq s_{\beta+\gamma}= s_{\nu}.$ 

\item[b)]If $\nu \cap \beta'$ is not a root then there are some roots $\gamma_{1}$ and $\gamma_{2}$ such that $\nu \cap \beta'=\gamma_{1}+\gamma_{2}$. Now, by Lemma \ref{lemma5.7} part $2b)$   we have $s_{\nu}\cdot s_{\beta'}=s_{\nu}\cdot s_{\beta'-\gamma_{1}-\gamma_{2}}\cdot s_{\gamma_{1}} \cdot s_{\gamma_{2}}$. We also have $\beta'-\gamma_{1}-\gamma_{2}=\nu'$ and $\beta+\gamma_{1}+\gamma_{2}=\nu$ which implies that 

$\begin{array}{lll}s_{\nu}\cdot s_{\nu'}\cdot s_{\beta}\cdot s_{\beta'}&=&s_{\beta}\cdot s_{\nu}\cdot s_{\beta'}\cdot s_{\nu'}=s_{\beta}\cdot s_{\nu}\cdot s_{\beta'-\gamma_{1}-\gamma_{2}}\cdot s_{\gamma_{1}} \cdot s_{\gamma_{2}}\cdot s_{\nu'}\\
&=&s_{\beta}\cdot s_{\nu}\cdot s_{\nu'} \cdot s_{\gamma_{1}} \cdot s_{\gamma_{2}}\cdot s_{\nu'}

\end{array}$ 

Again, $s_{\beta}$ commutes with $s_{\nu}$ and $s_{\nu'}$. Thus, 

$\begin{array}{lll}s_{\beta}\cdot s_{\nu}\cdot s_{\nu'} \cdot s_{\gamma_{1}} \cdot s_{\gamma_{2}}\cdot s_{\nu'}&=& s_{\nu}\cdot s_{\nu'} \cdot s_{\beta} \cdot s_{\gamma_{1}} \cdot s_{\gamma_{2}}\cdot s_{\nu'}\leq s_{\nu}\cdot s_{\nu'} \cdot s_{\beta+\gamma_{1}+\gamma_{2}}\cdot s_{\nu'}\\
&=&s_{\nu}\cdot s_{\nu'} \cdot s_{\nu}\cdot s_{\nu'}

\end{array}$ 

due to the fact that $ s_{\beta} \cdot s_{\gamma_{1}} \cdot s_{\gamma_{2}}\leq s_{\beta+\gamma_{1}+\gamma_{2}}=s_{\nu}.$ 
\end{itemize}
\end{itemize} \end{proof}

\begin{lemma}\label{lemma17.98} Let $\beta,\beta',\nu $ and $\nu' \in \Pi_{\text{aff}}^{\text{re},\,+}$ such that $\nu+\nu'=\beta+\beta'=c$. Also, assume that $\nu'>\beta$ and $\nu'\perp \beta.$ Then  $s_{\nu}\cdot s_{\nu'}\cdot s_{\beta}\cdot s_{\beta'}\leq (s_{\beta}\cdot s_{\beta'})^2.$ \end{lemma}

\begin{proof}First, note that $s_{\nu'}$ and $s_{\beta}$ Hecke commute by part $1)$ in Lemma \ref{lemma5.7}. Also, by the assumptions on $\nu'$ and $\beta$, the support of $\nu$ and $\beta$ are disconnected which implies that $s_{\nu}$ and $s_{\beta}$ Hecke commute. So we get $s_{\nu}\cdot s_{\nu'}\cdot s_{\beta}\cdot s_{\beta'}= s_{\beta}\cdot s_{\nu}\cdot s_{\nu'} \cdot s_{\beta'}.$ Here, the fact that the support of $\nu$ and $\beta$ are disconnected forces us to have $\nu'\cap \beta'\neq \emptyset$ and $\nu'\cap \beta'$ not to be a root. But then $\nu'\cap \beta'=\gamma_{1}+\gamma_{2}$ for some $\gamma_{1},\gamma_{2}\in \Pi_{\text{aff}}^{\text{re},\,+}.$ By part $2)$ in Lemma \ref{lemma5.7}, we get $s_{\nu'}\cdot s_{\beta'}=s_{\gamma_{1}}\cdot s_{\gamma_{2}}\cdot s_{\nu'-\gamma_{1}-\gamma_{2}}\cdot s_{\beta'}.$ Note that $\nu+\gamma_{1}+\gamma_{2}=\beta'$ and $\nu'-\gamma_{1}-\gamma_{2}=\beta.$ Thus 
$$s_{\beta}\cdot s_{\nu}\cdot s_{\nu'} \cdot s_{\beta'}=s_{\beta}\cdot s_{\nu}\cdot s_{\gamma_{1}}\cdot s_{\gamma_{2}}\cdot s_{\nu'-\gamma_{1}-\gamma_{2}}\cdot s_{\beta'}\leq s_{\beta}\cdot s_{\nu+\gamma_{1}+\gamma_{2}}\cdot s_{\beta}\cdot s_{\beta'}= s_{\beta}\cdot s_{\beta'}\cdot  s_{\beta}\cdot s_{\beta'}.$$ \end{proof}

\begin{lemma} \label{lemma17.99} Let $\beta,\beta',\nu $ and $\nu' \in \Pi_{\text{aff}}^{\text{re},\,+}$ such that $\nu+\nu'=\beta+\beta'=c$. Also, assume that $\nu'\leq \beta'$, $\beta'>\alpha$, and $\beta'\perp \alpha.$ Then,  for any positive integer $m,$  $s_{\nu}\cdot s_{\nu'}\cdot (s_{\beta}\cdot s_{\beta'})^m\cdot s_{\alpha}\leq (s_{\mu}\cdot s_{\mu'})^{m+1}\cdot s_{\alpha}$ for some positive real roots $\mu, \mu'$ such that $\mu+\mu'=c$ and either $\mu' \leq c-\alpha$ or both $\mu'>\alpha$ and $\mu' \perp \alpha$.

\end{lemma}

\begin{proof} First, observe that $\beta'>\alpha$ and $\beta'\perp \alpha$ imply that $c-\alpha>\beta$ and $(c-\alpha)\perp \beta$. Also, $\text{supp}(\beta)$ and $\text{supp}(\alpha)$ are disconnected. So $\alpha$ Hecke commute with both $\beta'$ and $\beta$. We have several cases here;
\begin{itemize}

\item[1)] If $\nu'=\beta'$ then take $\mu'=\beta'.$

\item[2)] If $\nu'\cap \alpha=\emptyset$ then $\nu'\leq c-\alpha.$ Note that, by Lemma \ref{lemma17.95}, we have $s_{\nu}\cdot s_{\nu'}\cdot s_{\beta}\cdot s_{\beta'}\leq (s_{\nu}\cdot s_{\nu'})^2$ since $\nu'\leq \beta'$ which follows by 
$ s_{\nu}\cdot s_{\nu'}\cdot (s_{\beta}\cdot s_{\beta'})^{m}\cdot s_{\alpha}\leq (s_{\nu}\cdot s_{\nu'})^{m+1}\cdot s_{\alpha}.$ So we can take $\mu'=\nu'.$

\item[3)] If $\nu'\cap \alpha \neq \emptyset$ then we have several cases;

\item If we have both $\nu'>\alpha$ and $\nu' \perp \alpha$ then by the same arguments in case $2)$ we can take $\mu'=\nu'.$

\item If $v'=\alpha$ then $s_{\nu}\cdot s_{\nu'}\cdot s_{\alpha}=s_{c-\alpha}\cdot s_{\alpha}\cdot s_{\alpha}$. By Lemma \ref{lemma17.3}, $s_{c-\alpha}\cdot s_{\alpha}\cdot s_{\alpha}\leq s_{\alpha}\cdot s_{c-\alpha}\cdot s_{\alpha}.$ Thus $s_{\nu}\cdot s_{\nu'}\cdot (s_{\beta}\cdot s_{\beta'})^{m}\cdot s_{\alpha}=s_{\nu}\cdot s_{\nu'}\cdot s_{\alpha}\cdot (s_{\beta}\cdot s_{\beta'})^{m}\leq s_{\alpha}\cdot s_{c-\alpha}\cdot s_{\alpha}\cdot (s_{\beta}\cdot s_{\beta'})^{m}=s_{\alpha}\cdot s_{c-\alpha}\cdot (s_{\beta}\cdot s_{\beta'})^{m}\cdot s_{\alpha}.$ Here, observe that by Lemma \ref{lemma17.98}, $s_{\alpha}\cdot s_{c-\alpha}\cdot s_{\beta}\cdot s_{\beta'}\leq (s_{\beta}\cdot s_{\beta'})^{2}$ since $c-\alpha>\beta$ and $(c-\alpha) \perp \beta$ which follows by $s_{\alpha}\cdot s_{c-\alpha}\cdot (s_{\beta}\cdot s_{\beta'})^{m}\cdot s_{\alpha}\leq (s_{\beta}\cdot s_{\beta'})^{m+1}\cdot s_{\alpha}.$ Hence we can take $\mu'=\beta'$ in this case.

\item If $\nu'<\alpha$ then $\nu'$ and $\alpha$ Hecke commute by part $1)$ in Lemma \ref{lemma5.7}. Also, note that $\nu\cap \alpha\neq \emptyset.$ Now, assume that $\nu\cap \alpha$ is a root, let $\gamma:=\nu\cap \alpha.$ Then, by part $2)$ in Lemma \ref{lemma5.7} we have $s_{\nu}\cdot s_{\alpha}=s_{\gamma}\cdot s_{\nu-\gamma}\cdot s_{\alpha}.$ Here, note that $\nu-\gamma+\nu'=c-\gamma$ is a root and by the same lemma $s_{\nu-\gamma}\cdot s_{\nu'}\leq s_{c-\gamma}.$ Hence $s_{\nu}\cdot s_{\nu'}\cdot s_{\alpha}=s_{\nu}\cdot s_{\alpha}\cdot s_{\nu'}=s_{\gamma}\cdot s_{\nu-\gamma}\cdot s_{\alpha}\cdot s_{\nu'}=s_{\gamma}\cdot s_{\nu-\gamma}\cdot s_{\nu'}\cdot s_{\alpha}\leq s_{\gamma}\cdot s_{c-\gamma}\cdot s_{\alpha}.$ Moreover, $c-\gamma>c-\alpha>\beta$ and $(c-\gamma)\perp \beta$ since $(c-\alpha )\perp \beta.$ Then by Lemma \ref{lemma17.98}, $s_{\gamma}\cdot s_{c-\gamma}\cdot s_{\beta}\cdot s_{\beta'}\leq (s_{\beta}\cdot s_{\beta'})^{2}$ which follows by $s_{\nu}\cdot s_{\nu'}\cdot (s_{\beta}\cdot s_{\beta'})^{m}\cdot s_{\alpha}\leq s_{\nu}\cdot s_{\nu'}\cdot s_{\alpha}\cdot (s_{\beta}\cdot s_{\beta'})^m\leq  s_{\gamma}\cdot s_{c-\gamma}\cdot s_{\alpha}\cdot (s_{\beta}\cdot s_{\beta'})^m=  s_{\gamma}\cdot s_{c-\gamma}\cdot (s_{\beta}\cdot s_{\beta'})^m \cdot s_{\alpha}\leq (s_{\beta}\cdot s_{\beta'})^{m+1} \cdot s_{\alpha}.$ Hence we can take $\mu'=\beta'$ in this case. Now, if $\nu\cap \alpha$ is not a root then one can show that we can take $\mu'=\beta'$, by the same arguments in this case.

\item If $\nu'>\alpha$ but $\nu'$ is not perpendicular to $\alpha$ then $<\nu',\alpha^{\vee}>=1$ so $s_{\alpha}(\nu')=\nu'-\alpha$. So $\nu'-\alpha$ is a root which implies that $\nu+\alpha=c-\nu'+\alpha=c-(\nu'-\alpha)$ is also a root. Thus  $s_{\nu}\cdot s_{\nu'}\cdot s_{\alpha}=s_{\nu}\cdot s_{\alpha}\cdot s_{\nu'}=s_{\nu+\alpha}\cdot s_{\nu'}$ by Lemma \ref{lemma5.7}. Here, note that $(\nu+\alpha)\cap \nu'=\alpha$. So by the same lemma we get $s_{\nu+\alpha} \cdot s_{\nu'}=s_{\nu+\alpha}\cdot s_{\nu'-\alpha}\cdot s_{\alpha}.$ Now, observe that $\nu'-\alpha\leq \beta'$ since both $\nu'$ and $\alpha$ are smaller than $\beta'.$ Then by Lemma \ref{lemma17.95}, $s_{\nu+\alpha}\cdot s_{\nu'-\alpha}\cdot s_{\beta}\cdot s_{\beta'} \leq (s_{\nu+\alpha}\cdot s_{\nu'-\alpha})^2$ which follows by $s_{\nu}\cdot s_{\nu'}\cdot (s_{\beta}\cdot s_{\beta'})^m\cdot s_{\alpha}=s_{\nu}\cdot s_{\nu'}\cdot s_{\alpha}\cdot(s_{\beta}\cdot s_{\beta'})^m \leq s_{\nu+\alpha}\cdot s_{\nu'-\alpha}\cdot s_{\alpha}\cdot (s_{\beta}\cdot s_{\beta'})^m  \leq s_{\nu+\alpha}\cdot s_{\nu'-\alpha}\cdot (s_{\beta}\cdot s_{\beta'})^m\cdot s_{\alpha}\leq (s_{\nu+\alpha}\cdot s_{\nu'-\alpha})^{m+1}\cdot s_{\alpha}.$ Last, note that $\nu'-\alpha\leq c-\alpha$ so we can take $\mu'=\nu'-\alpha$ in this case. 

\item Assume that $\nu'\cap \alpha\neq \emptyset$ and also $\nu'\cap \alpha \neq \nu',\alpha.$ Then $\nu'\cap \alpha$ is a root since $\nu'+\alpha$ would be bigger than $c$ otherwise, and that would contradict the facts that $\beta'>\alpha$ and $\beta'>\nu'$. Now, let $\gamma:=\nu'\cap \alpha.$
Then $\nu'-\gamma$ is a root and $s_{\nu'}\cdot s_{\alpha}=s_{\gamma}\cdot s_{\nu'-\gamma}\cdot s_{\alpha}$ by Lemma \ref{lemma5.7}. Also, note that $\nu+\gamma=c-\nu'+\gamma=c-(\nu'-\gamma)$ is a root so by the same lemma we have $s_{\nu}\cdot s_{\gamma}\leq s_{\nu+\gamma}$. Thus $s_{\nu}\cdot s_{\nu'}\cdot s_{\alpha}=s_{\nu}\cdot s_{\gamma}\cdot s_{\nu'-\gamma}\cdot s_{\alpha}\leq s_{\nu+\gamma}\cdot s_{\nu'-\gamma}\cdot s_{\alpha}.$ Here, note that $\nu'-\gamma\leq \beta'$ since $\nu'\leq \beta'$. Then by Lemma \ref{lemma17.95}, $s_{\nu+\gamma}\cdot s_{\nu'-\gamma}\cdot s_{\beta}\cdot s_{\beta'} \leq (s_{\nu+\gamma}\cdot s_{\nu'-\gamma})^2$ which follows by $s_{\nu}\cdot s_{\nu'}\cdot (s_{\beta}\cdot s_{\beta'})^m\cdot s_{\alpha}=s_{\nu}\cdot s_{\nu'}\cdot s_{\alpha}\cdot(s_{\beta}\cdot s_{\beta'})^m \leq s_{\nu+\gamma}\cdot s_{\nu'-\gamma}\cdot s_{\alpha}\cdot (s_{\beta}\cdot s_{\beta'})^m  \leq s_{\nu+\gamma}\cdot s_{\nu'-\gamma}\cdot (s_{\beta}\cdot s_{\beta'})^m\cdot s_{\alpha}\leq (s_{\nu+\gamma}\cdot s_{\nu'-\gamma})^{m+1}\cdot s_{\alpha}.$ Last, note that $\nu'-\gamma\leq c-\alpha$ so we can take $\mu'=\nu'-\gamma$ in this case.

\end{itemize} \end{proof}

The result of this section is:

\begin{thm}\label{thm18} Let $\alpha<c$ be an affine positive real root and let $m$ be a positive integer. Then

\begin{enumerate}
\item[1)] $\Gamma_{mc+\alpha}(id)=\{t_{m\beta'} s_{\alpha}:\beta'\in \Pi_{\text{aff}}^{\text{re},\,+}(\alpha)\} \cup \{t_{m(\beta'-c)} s_{\alpha}:\beta'\in \Pi_{\text{aff}}^{\text{re},\,+}(\alpha)\}$
 
\item[2)] $|\Gamma_{mc+\alpha}( id)|=|\{\beta': \beta' \in \Pi_{\text{aff}}^{\text{re},\,+}(\alpha)\}|$

\item[3)] For all $w\in \Gamma_{mc+\alpha}( id)$, $\ell(w)=2m(n-1)+\ell{(s_{\alpha})}$

\end{enumerate}
where $ \Pi_{\text{aff}}^{\text{re},\,+}(\alpha):=\{\beta' \in \Pi_{\text{aff}}^{\text{re},\,+}: \beta'<c \,\,\text{and either}\,\,  \beta'\leq c-\alpha \,\, \text{or both} \,\,\beta'>\alpha \,\,\text{and} \,\, \beta'\perp \alpha \}$.

\end{thm}

\begin{proof} $1)$ We will prove the statement by induction on $m.$ Assume that $m=1.$ Then the statement is true by Theorem \ref{thm17.9}. Now, assume that the statement is true for $m=k.$ We will prove that it is also true for $m=k+1.$ Let $id \stackrel{\beta_{1}} \longrightarrow s_{\beta_{1}} \stackrel{\beta_{2}} \longrightarrow s_{\beta_{1}}s_{\beta_{2}} ...\stackrel{\beta_{r}} \longrightarrow w=s_{\beta_{1}}s_{\beta_{2}}...s_{\beta_{r}}$
be a path in the moment graph such that $\sum_{i=1}^{r} \beta_{i} \leq (k+1)c+\alpha$. Note that, by Remark \ref{remark12.9} we can assume that $w=s_{\beta_{1}} \cdot s_{\beta_{2}}\cdot...\cdot s_{\beta_{r}}$ where $\sum_{i=1}^{r} \beta_{i} = (k+1)c+\alpha$ and $\beta_{i}<c$ for all $i$. It is enough to show that $w\leq (s_{\beta} s_{\beta'})^{k+1}s_{\alpha}$ for some $\beta, \beta' \in \Pi_{\text{aff}}^{\text{re},\,+}$ such that $\beta+\beta'=c$ where either $\beta' \leq c-\alpha$ or both $\beta'>\alpha$ and $\beta'\perp \alpha$ since $t_{\beta'}=s_{\beta}s_{\beta'}$ if $\alpha_{0}\in \text{supp}(\beta)$ and $t_{\beta'-c}=s_{\beta}s_{\beta'}$ if $\alpha_{0}\notin \text{supp}(\beta)$, see the proof of Lemma \ref{lemma10}.  Moreover, by Lemma \ref{lemma17.71} we can make the assumption that there is an integer $p$ such that $\sum_{i=1}^{p}\beta_{i}=c$ and $\sum_{i=p+1}^{r}\beta_{i}=kc+\alpha$. Then by Equation \ref{equation14.5} we have $s_{\beta_{1}}\cdot s_{\beta_{2}}\cdot ...\cdot s_{\beta_{p}}\leq s_{\gamma} s_{\gamma'}$ for some $\gamma, \gamma' \in \Pi_{\text{aff}}^{\text{re},\,+}$ such that $\gamma+\gamma'=c$ and $s_{\beta_{p+1}}\cdot ...\cdot s_{\beta_{r}}\leq (s_{\beta} s_{\beta'})^{k}s_{\alpha}$ for some $\beta, \beta' \in \Pi_{\text{aff}}^{\text{re},\,+}$ such that $\beta+\beta'=c$ where either $\beta' \leq c-\alpha$ or both $\beta'>\alpha$ and $\beta'\perp \alpha$ by the induction assumption. Here, note that, $s_{\gamma} s_{\gamma'}$ is reduced by Lemma  \ref{lemma11}. Moreover, $ (s_{\beta} s_{\beta'})^{k}s_{\alpha}$ is also reduced by Corollary \ref{cor12.8}. Thus, 
$$w=(s_{\beta_{1}}\cdot s_{\beta_{2}}\cdot ...\cdot s_{\beta_{p}})\cdot (s_{\beta_{p+1}}\cdot ...\cdot s_{\beta_{r}})\leq (s_{\gamma} s_{\gamma'})\cdot ((s_{\beta}s_{\beta'})^{k}s_{\alpha})=(s_{\gamma}\cdot s_{\gamma'})\cdot ((s_{\beta}\cdot s_{\beta'})^{k}\cdot s_{\alpha}).$$
Now note that by Lemma \ref{lemma17.8} there are $\nu, \nu' \in \Pi_{\text{aff}}^{\text{re},\,+}$ such that $s_{\gamma}\cdot s_{\gamma'}\cdot s_{\beta}\leq s_{\nu}\cdot s_{\nu'}\cdot s_{\beta}$ where $\nu+\nu'=c$ and either $\nu' \leq c-\beta=\beta'$ or both $\nu'>\beta$ and $\nu'\perp \beta.$ 

\begin{itemize}

\item If $\nu'\leq c-\beta=\beta'$ then by Lemma \ref{lemma17.95}, we get $s_{\nu}\cdot s_{\nu'}\cdot s_{\beta}\cdot s_{\beta'}\leq (s_{\nu}\cdot s_{\nu'})^2$ which follows by 
$w\leq (s_{\gamma}\cdot s_{\gamma'})\cdot (s_{\beta}\cdot s_{\beta'})^{k}\cdot s_{\alpha}\leq s_{\nu}\cdot s_{\nu'}\cdot (s_{\beta}\cdot s_{\beta'})^{k}\cdot s_{\alpha}\leq (s_{\nu}\cdot s_{\nu'})^{k+1}\cdot s_{\alpha}.$ Now, if we also have $\beta'\leq c-\alpha$ then $c-\alpha \geq \beta'\geq \nu'$ so we are done with this case. Now if  $\beta' >\alpha$ and $\beta'\perp \alpha$ then by Lemma \ref{lemma17.99}, $s_{\nu}\cdot s_{\nu'}\cdot (s_{\beta}\cdot s_{\beta'})^m\cdot s_{\alpha}\leq (s_{\mu}\cdot s_{\mu'})^{m+1}\cdot s_{\alpha}$ for some positive real roots, $\mu, \mu'$ such that $\mu+\mu'=c$ and either $\mu' \leq c-\alpha$ or both $\mu'>\alpha$ and $\mu' \perp \alpha$. Hence we are done with this case too.

\item If $\nu'>\beta$ and $\nu'\perp \beta$ then by Lemma \ref{lemma17.98} we get $s_{\nu}\cdot s_{\nu'}\cdot s_{\beta}\cdot s_{\beta'}\leq (s_{\beta}\cdot s_{\beta'})^2$ which follows by 
$w\leq (s_{\gamma}\cdot s_{\gamma'})\cdot (s_{\beta}\cdot s_{\beta'})^{k}\cdot s_{\alpha}\leq s_{\nu}\cdot s_{\nu'}\cdot (s_{\beta}\cdot s_{\beta'})^{k}\cdot s_{\alpha}\leq (s_{\beta}\cdot s_{\beta'})^{k+1}\cdot s_{\alpha}.$ Thus we are done with all cases. 
\end{itemize}
Moreover, by Corollary \ref{cor12.8},  $(s_{\beta}s_{\beta'})^{k+1}s_{\alpha} $ is reduced so $(s_{\beta}\cdot s_{\beta'})^{k+1}\cdot s_{\alpha}=(s_{\beta}s_{\beta'})^{k+1}s_{\alpha}.$ 

$2)$ Let $(s_{\beta}s_{\beta'})^{m}s_{\alpha}$ and $ (s_{\nu}s_{\nu'})^{m}s_{\alpha}\in \Gamma_{mc+\alpha}( id).$ We need to show that $(s_{\beta}s_{\beta'})^{m}s_{\alpha}\neq (s_{\nu}s_{\nu'})^{m}s_{\alpha}$ if $\beta'\neq \nu'.$ Now assume that $\beta'\neq \nu'$ but 
\begin{equation}\label{equation18.5}(s_{\beta}s_{\beta'})^{m}s_{\alpha}= (s_{\nu}s_{\nu'})^{m}s_{\alpha}.\end{equation} By Corollary \ref{cor12.8} both $(s_{\beta}s_{\beta'})^{m}s_{\alpha}$ and $(s_{\nu}s_{\nu'})^{m}s_{\alpha}$ are reduced, so Equation \ref{equation18.5} implies that $(s_{\beta}s_{\beta'})^{m}= (s_{\nu}s_{\nu'})^{m}$ but this is a contradiction by Lemma \ref{lemma10}.

$3)$ Let $w=(s_{\beta}s_{\beta'})^{m} s_{\alpha}\in  \Gamma_{mc+\alpha}( id)$. Then $\ell(w)=2m(n-1)+\ell{(s_{\alpha})}$, by Corollary \ref{cor12.8}.

\end{proof}

\begin{remark} \label{remark18.7} Let $\alpha\in \Pi_{\text{aff}}^{\text{re},\,+}$ be such that $\alpha<c$ and $m$ be a positive integer. Note that \begin{equation}\label{equation18.8} \Gamma_{mc+\alpha} (id)=\{(s_{\beta}s_{\beta'})^{m}s_{\alpha}:\beta' \in  \Pi_{\text{aff}}^{\text{re},\,+}(\alpha)\,\, \text{such that}\,\, \beta+\beta'=c\},
\end{equation}
see the proof of Theorem \ref{thm18}.

\end{remark}

\begin{example}Suppose that the affine Weyl group $W_{\text{aff}}$ is of type $A_{4}^{(1)}.$ We compute $\Gamma_{12c+\alpha}( id)$ where $\alpha=\alpha_{0}+\alpha_{4}.$  Then, by Theorem \ref{thm18}, we get 
$$ \Gamma_{12c+\alpha}( id)=\{t_{12\beta'_{1}}s_{\alpha}, t_{12\beta'_{2}}s_{\alpha},t_{12\beta'_{3}}s_{\alpha},t_{12\beta'_{4}}s_{\alpha},t_{12\beta'_{5}}s_{\alpha},t_{12\beta'_{6}}s_{\alpha},t_{12(\beta'_{7}-c)}s_{\alpha}\}$$
where $\beta'_{1}=\alpha_{1}, \beta'_{2}=\alpha_{2}, \beta'_{3}=\alpha_{3}, \beta'_{4}=\alpha_{1}+\alpha_{2}, \beta'_{5}=\alpha_{2}+\alpha_{3} , \beta'_{6}=\alpha_{1}+\alpha_{2}+\alpha_{3}$ and $\beta'_{7}=\alpha_{0}+\alpha_{1}+\alpha_{3}+\alpha_{4}$, see Example \ref{example17.93}. Moreover, for any $w\in \Gamma_{12c+\alpha}( id),$ we have $\ell(w)=2m(n-1)+\ell{(s_{\alpha})}=2m(n-1)+2\, \text{supp}(\alpha)-1=2\cdot12\cdot4+2\cdot2-1=99.$

\end{example}

\subsection{Curve Neighborhood $\Gamma_{mc}(id)$} \label{case of mc}

\begin{thm} \label{thm19} Let $m\geq 2$ be a positive integer.  Then we have 

\begin{enumerate}

\item[1)] $\Gamma_{mc} ( id)=\{t_{m\gamma}:\gamma \in \Pi \}$. 

\item[2)]  $|\Gamma_{mc}( id)|=n(n-1)$.

\item[3)] For all $w\in \Gamma_{mc}( id)$, $\ell(w)=2m(n-1)$.

\end{enumerate}

\end{thm}

\begin{proof} $1)$ Let $id \stackrel{\beta_{1}} \longrightarrow s_{\beta_{1}} \stackrel{\beta_{2}} \longrightarrow s_{\beta_{1}}s_{\beta_{2}} ...\stackrel{\beta_{r}} \longrightarrow w=s_{\beta_{1}}s_{\beta_{2}}...s_{\beta_{r}}$
be a path in the moment graph such that $\sum_{i=1}^{r} \beta_{i} \leq mc$. Note that, by Remark \ref{remark12.9} we can assume that $w=s_{\beta_{1}} \cdot s_{\beta_{2}}\cdot...\cdot s_{\beta_{r}}.$ where $\sum_{i=1}^{r} \beta_{i}=mc$ and $\beta_{i}<c$ for all $i$. Now, note that $\sum_{i=1}^{r-1}\beta_{i}=mc-\beta_{r}=(m-1)c+c-\beta_{r}$. Let $\alpha:=c-\beta_{r}$ and $\alpha':=\beta_{r}$. Then by Remark \ref{remark18.7} we have $s_{\beta_{1}}\cdot s_{\beta_{2}}\cdot ...\cdot s_{\beta_{r-1}}\leq (s_{\beta}s_{\beta'})^{m-1}s_{\alpha}$ for some positive real roots $\beta$ and $\beta'$ such that $\beta+\beta'=c$ and either $\beta'\leq c-\alpha$ or both $\beta'>\alpha$ and $\beta'\perp \alpha.$ Then, we get
$$w= (s_{\beta_{1}}\cdot s_{\beta_{2}}\cdot ...\cdot s_{\beta_{r-1}})\cdot s_{\beta_{r}}\leq (s_{\beta}s_{\beta'})^{m-1}s_{\alpha}\cdot s_{\alpha'}.$$
Now, note that $(s_{\beta}s_{\beta'})^{m-1}s_{\alpha}=(s_{\beta}\cdot s_{\beta'})^{m-1}\cdot s_{\alpha},$ since $=(s_{\beta}s_{\beta'})^{m-1}s_{\alpha}$ is reduced by Corollary \ref{cor12.8}. Furthermore, if $\beta'\leq c-\alpha$ then $s_{\beta}\cdot s_{\beta'}\cdot s_{\alpha}\cdot s_{\alpha'}\leq (s_{\beta}\cdot s_{\beta'})^2$ by Lemma \ref{lemma17.95} which follows by $(s_{\beta}s_{\beta'})^{m-1}s_{\alpha}\cdot s_{\alpha'}= (s_{\beta}\cdot s_{\beta'})^{m-1}\cdot s_{\alpha}\cdot s_{\alpha'} \leq (s_{\beta}\cdot s_{\beta'})^m.$ If $\beta'>\alpha$ and $\beta'\perp \alpha$ then $s_{\beta}\cdot s_{\beta'}\cdot s_{\alpha}\cdot s_{\alpha'}\leq (s_{\alpha}\cdot s_{\alpha'})^2$ by Lemma \ref{lemma17.98} which follows by $(s_{\beta}s_{\beta'})^{m-1}s_{\alpha}\cdot s_{\alpha'}= (s_{\beta}\cdot s_{\beta'})^{m-1}\cdot s_{\alpha}\cdot s_{\alpha'} \leq (s_{\alpha}\cdot s_{\alpha'})^m.$ We also know that $(s_{\alpha}s_{\alpha'})^m=t_{m\gamma}$ for some $\gamma \in \Pi,$ see the proof of Lemma \ref{lemma10}.

$2)$ Let $\beta,\beta',\nu,\nu' \in \Pi_{\text{aff}}^{\text{re},\,+}$ such that $\nu+\nu'=\beta+\beta'=c.$ Then $(s_{\beta}s_{\beta'})^{m}\neq (s_{\nu}s_{\nu'})^{m}$ for any positive integer $m$ if $\beta\neq \nu$ by Lemma \ref{lemma10}. Also, note that  $|\{\beta \in \Pi_{\text{aff}}^{\text{re},\,+}: \beta<c\}|=n(n-1)$.\end{proof}

$3)$ Let $w=(s_{\beta}s_{\beta'})^{m} \in \Gamma_{mc}( id)$ for some positive real roots $\beta,\beta'$ such that $\beta+\beta'=c.$ Then $\ell{(w)=2m(n-1)}$ by Lemma \ref{lemma11}.

\begin{remark} Let $m$ be a positive integer. Note that, we have \begin{equation}\label{equation19.2} \Gamma_{mc} ( id)=\{(s_{\beta}s_{\beta'})^{m}:\beta, \beta' \in   \Pi_{\text{aff}}^{\text{re},\,+}\, \text{such that}\,\, \beta+\beta'=c\},
\end{equation}
see the proof of Theorem \ref{thm19}.
\end{remark}

\begin{example} Suppose that the affine Weyl group $W_{\text{aff}}$ is associated to $A_{2}^{(1)}.$ Then $$\Gamma_{10c}( id)=\{t_{10\alpha_{1}},t_{10\alpha_{2}},t_{10(\alpha_{1}+\alpha_{2})}, t_{-10\alpha_{1}}, t_{-10\alpha_{2}},t_{-10(\alpha_{1}+\alpha_{2})}\}.$$ 
Moreover, for any $w\in \Gamma_{10c}( id)$, $\ell(w)=2m(n-1)=2\cdot 10 \cdot 2=40$. 

\end{example}

\subsection{Curve Neighborhood $\Gamma_{\mathbf{d}}(id)$ for any $\mathbf{d}>c$} \label{the most general case}

Finally, we will discuss the most general case in this section. In order to prove the result we will first consider two lemmas.  

\begin{lemma} \label{lemma19.3} Let $\alpha,\alpha' \in \Pi_{\text{aff}}^{\text{re},\,+}$ such that $\alpha+\alpha'=c$ Also, let $\gamma_{1},\gamma_{2},...,\gamma_{k} \in \Pi_{\text{aff}}^{\text{re},\,+}$ such that $\text{supp}(\gamma_{i})$ and $\text{supp}(\gamma_{j})$ are disconnected for any $1\leq i,j\leq k$ such that $i\neq j.$ Then $s_{\alpha}\cdot s_{\alpha'}\cdot s_{\gamma_{1}}\cdot s_{\gamma_{2}}\cdot ...\cdot s_{\gamma_{k}} \leq s_{\beta}\cdot s_{\beta'}\cdot s_{\gamma_{1}}\cdot s_{\gamma_{2}}\cdot ...\cdot s_{\gamma_{k}}$ for some affine positive real roots, $\beta,\beta'  \in \Pi_{\text{aff}}^{\text{re},\,+}$ such that either $\beta'\cap \gamma_{i}=\emptyset$ or both $\beta'>\gamma_{i}$ and $\beta'\perp \gamma_{i}$ for any $i=1,2,...,k.$

\end{lemma}

\begin{proof} We will prove the statement by the induction on $k.$ Let $k=1$. Then by Equation \ref{equation17.92},  $ (s_{\alpha}\cdot s_{\alpha'})\cdot s_{\gamma_{1}}\leq (s_{\beta}\cdot s_{\beta'})\cdot s_{\gamma_{1}}$ for some $\beta, \beta' \in \Pi_{\text{aff}}^{\text{re},\,+}$ such that $\beta+\beta'=c$ where either $\beta'\leq c- \gamma_{1}$ i.e $\beta'\cap \gamma_{1}=\emptyset$ or both $\beta'>\gamma_{1}$ and $\beta'\perp \gamma_{1}$. So we are done with this case. Now we will assume that the statement is true for $k$ and prove that it is also true for $k+1$ i.e we will show that $ (s_{\alpha}\cdot s_{\alpha'})\cdot s_{\gamma_{1}}\cdot s_{\gamma_{2}}\cdot ...\cdot s_{\gamma_{k+1}}\leq (s_{\beta}\cdot s_{\beta'})\cdot s_{\gamma_{1}}\cdot s_{\gamma_{2}}\cdot ...\cdot s_{\gamma_{k+1}}$ for some affine positive real roots, $\beta,\beta'  \in \Pi_{\text{aff}}^{\text{re},\,+}$ such that either $\beta'\cap \gamma_{i}=\emptyset$ or both $\beta'>\gamma_{i}$ and $\beta'\perp \gamma_{i}$ for any $i=1,2,...,k+1.$  Now note that $ (s_{\alpha}\cdot s_{\alpha'})\cdot s_{\gamma_{1}}\cdot s_{\gamma_{2}}\cdot ...\cdot s_{\gamma_{k}}\leq (s_{\nu} \cdot s_{\nu'})\cdot s_{\gamma_{1}}\cdot s_{\gamma_{2}}\cdot ...\cdot s_{\gamma_{k}}$ for some $\nu, \nu' \in \Pi_{\text{aff}}^{\text{re},\,+}$ such that $\nu+\nu'=c$ where $\nu'\cap \gamma_{i}=\emptyset$ or both $\nu'>\gamma_{i}$ and $\nu'\perp \gamma_{i}$ for any $i=1,2,...,k$ by the induction assumption. So there is a subset $I$ of $\{\gamma_{1},\gamma_{2},...,\gamma_{k}\}$ such that $\nu'\cap \gamma_{i}=\emptyset$ for all $\gamma_{i}\in I$ and $\nu'>\gamma_{j}$ and $\nu'\perp \gamma_{j}$   for all $\gamma_{j}\in \{\gamma_{1},\gamma_{2},...,\gamma_{k}\} \setminus I$. Thus $ (s_{\alpha}\cdot s_{\alpha'})\cdot s_{\gamma_{1}}\cdot s_{\gamma_{2}}\cdot ...\cdot s_{\gamma_{k}}\cdot s_{\gamma_{k+1}}\leq (s_{\nu}\cdot  s_{\nu'}\cdot s_{\gamma_{1}}\cdot s_{\gamma_{2}}\cdot ...\cdot s_{\gamma_{k}})\cdot s_{\gamma_{k+1}}.$ Now, if we also have $\nu'\cap \gamma_{k+1}=\emptyset$ or both $\nu'>\gamma_{k+1}$ and $\nu'\perp \gamma_{k+1}$ then we are done. Assume not i.e assume that $\nu'\cap \gamma_{k+1}\neq \emptyset$ and $\nu'$ is not strictly bigger than or not perpendicular to $\gamma_{k+1}.$ Note that $s_{\gamma_{i}}$ and $s_{\gamma_{j}}$ Hecke commute for any $1\leq i,j\leq k+1$  so $s_{\nu}\cdot s_{\nu'}\cdot s_{\gamma_{1}}\cdot s_{\gamma_{2}}\cdot ...\cdot s_{\gamma_{k}}\cdot s_{\gamma_{k+1}}=s_{\nu} \cdot s_{\nu'}\cdot s_{\gamma_{k+1}}\cdot s_{\gamma_{1}}\cdot s_{\gamma_{2}}\cdot ...\cdot s_{\gamma_{k}}.$ Here we have several cases;

a) Assume that $\nu'\cap \gamma_{k+1} \neq \nu',\gamma_{k+1}$ and $\nu'\cap \gamma_{k+1}$ is a root. Let $\gamma:=\nu'\cap \gamma_{k+1}.$
Then $\nu'-\gamma$ is a root and $s_{\nu'}\cdot s_{\gamma_{k+1}}=s_{\gamma}\cdot s_{\nu'-\gamma}\cdot s_{\gamma_{k+1}}$ by Lemma \ref{lemma5.7}. Also, note that $\nu+\gamma=c-\nu'+\gamma=c-(\nu'-\gamma)$ is a root so by the same lemma we have $s_{\nu}\cdot s_{\gamma}\leq s_{\nu+\gamma}$. Thus $s_{\nu}\cdot s_{\nu'}\cdot s_{\gamma_{k+1}}=s_{\nu}\cdot s_{\gamma}\cdot s_{\nu'-\gamma}\cdot s_{\gamma_{k+1}}\leq s_{\nu+\gamma}\cdot s_{\nu'-\gamma}\cdot s_{\gamma_{k+1}}$ which follows by $s_{\nu}\cdot s_{\nu'}\cdot s_{\gamma_{k+1}}\cdot s_{\gamma_{1}}\cdot s_{\gamma_{2}}\cdot ...\cdot s_{\gamma_{k}}\leq s_{\nu+\gamma}\cdot s_{\nu'-\gamma}\cdot s_{\gamma_{k+1}}\cdot s_{\gamma_{1}}\cdot s_{\gamma_{2}}\cdot ...\cdot s_{\gamma_{k}}.$ Here, note that $(\nu'-\gamma) \cap \gamma_{k+1}=\emptyset$ and $(\nu'-\gamma)\cap \gamma_{i}=\emptyset$ for all $\gamma_{i}\in I$ and  $(\nu'-\gamma)>\gamma_{j}$ and $(\nu'-\gamma)\perp \gamma_{j}$   for all $\gamma_{j}\in \{\gamma_{1},\gamma_{2},...,\gamma_{k}\} \setminus I$ since $\nu'-\gamma< \nu'$ and $\text{supp}(\gamma)$ and $\text{supp}(\gamma_{i})$ are disconnected for any $i=1,2,...,k$. So we can take $\beta'=\nu'-\gamma$ in this case. The case where $\nu'\cap \gamma_{k+1}$ is not a root is similar.

b) If $\nu'\cap \gamma_{k+1}=\gamma_{k+1}$ then $\gamma_{k+1}\leq \nu'$. 

\begin{itemize}

\item Assume that $v'=\gamma_{k+1}$. Then $s_{\nu}\cdot s_{\nu'}\cdot s_{\gamma_{k+1}}=s_{c-\gamma_{k+1}}\cdot s_{\gamma_{k+1}}\cdot s_{\gamma_{k+1}}$. By Lemma \ref{lemma17.3}, $s_{c-\gamma_{k+1}}\cdot s_{\gamma_{k+1}}\cdot s_{\gamma_{k+1}}\leq s_{\gamma_{k+1}}\cdot s_{c-\gamma_{k+1}}\cdot s_{\gamma_{k+1}}$ which follows by $s_{\nu}\cdot s_{\nu'}\cdot s_{\gamma_{k+1}}\cdot s_{\gamma_{1}}\cdot s_{\gamma_{2}}\cdot ...\cdot s_{\gamma_{k}}\leq s_{\gamma_{k+1}}\cdot s_{c-\gamma_{k+1}}\cdot s_{\gamma_{k+1}}\cdot s_{\gamma_{1}}\cdot s_{\gamma_{2}}\cdot ...\cdot s_{\gamma_{k}}.$  Here, observe that $(c-\gamma_{k+1})\cap \gamma_{k+1}=\emptyset.$ Also, $(c-\gamma_{k+1})>\gamma_{i}$ and  $(c-\gamma_{k+1})\perp \gamma_{i}$ for all $i=1,2,...,k$ since $\text{supp}(\gamma_{k+1})$ and $\text{supp}(\gamma_{i})$ are disconnected for any $i=1,2,...,k$. Hence we can take $\beta'=c-\gamma_{k+1}$ in this case.

\item  Assume that $\gamma_{k+1}<\nu'.$ Now since $\nu'$ is not perpendicular to $\gamma_{k+1}$ we have  $<\nu',\gamma_{k+1}^{\vee}>=1$ so $s_{\gamma_{k+1}}(\nu')=\nu'-\gamma_{k+1}$. So $\nu'-\gamma_{k+1}$ is a root which implies that $\nu+\gamma_{k+1}=c-\nu'+\gamma_{k+1}=c-(\nu'-\gamma_{k+1})$ is also a root. Thus  $s_{\nu}\cdot s_{\nu'}\cdot s_{\gamma_{k+1}}=s_{\nu}\cdot s_{\gamma_{k+1}}\cdot s_{\nu'}=s_{\nu+\gamma_{k+1}}\cdot s_{\nu'}$ by Lemma \ref{lemma5.7}. Here, note that $(\nu+\gamma_{k+1})\cap \nu'=\gamma_{k+1}$. So by the same lemma we get $s_{\nu+\gamma_{k+1}} \cdot s_{\nu'}=s_{\nu+\gamma_{k+1}}\cdot s_{\nu'-\gamma_{k+1}}\cdot s_{\gamma_{k+1}}$ which follows by $s_{\nu}\cdot s_{\nu'}\cdot s_{\gamma_{k+1}}\cdot s_{\gamma_{1}}\cdot s_{\gamma_{2}}\cdot ...\cdot s_{\gamma_{k}}\leq s_{\nu+\gamma_{k+1}}\cdot s_{\nu'-\gamma_{k+1}}\cdot s_{\gamma_{k+1}}\cdot s_{\gamma_{1}}\cdot s_{\gamma_{2}}\cdot ...\cdot s_{\gamma_{k}}.$  Here, observe that $(\nu'-\gamma_{k+1})\cap \gamma_{k+1}=\emptyset.$ Also, $(\nu'-\gamma_{k+1})\cap \gamma_{i}=\emptyset$ for all $\gamma_{i}\in I$ and  $(\nu'-\gamma_{k+1})>\gamma_{j}$ and $(\nu'-\gamma_{k+1})\perp \gamma_{j}$ for all $\gamma_{j}\in \{\gamma_{1},\gamma_{2},...,\gamma_{k}\} \setminus I$ since $\nu'-\gamma_{k+1}< \nu'$ and $\text{supp}(\gamma_{k+1})$ and $\text{supp}(\gamma_{i})$ are disconnected for any $i=1,2,...,k$. So we can take $\beta'=\nu'-\gamma_{k+1}$ in this case.

\end{itemize}

c) If $\nu'\cap \gamma_{k+1}=\nu'$ then $\nu'\leq \gamma_{k+1}$. We already consider the case where $\nu'= \gamma_{k+1}$ in case b) above so assume that $\nu'< \gamma_{k+1}.$ Then $s_{\nu'}$ and $s_{\gamma_{k+1}}$ Hecke commute by part $1)$ in Lemma \ref{lemma5.7}. Also, note that $\nu\cap \gamma_{k+1}\neq \emptyset.$ Now, assume that $\nu\cap \gamma_{k+1}$ is a root; let $\gamma:=\nu\cap \gamma_{k+1}.$ Then, by part $2)$ in Lemma \ref{lemma5.7} we have $s_{\nu}\cdot s_{\gamma_{k+1}}=s_{\gamma}\cdot s_{\nu-\gamma}\cdot s_{\gamma_{k+1}}.$ Here, note that $\nu-\gamma+\nu'=c-\gamma$ is a root and by the same lemma $s_{\nu-\gamma}\cdot s_{\nu'}\leq s_{c-\gamma}.$ Hence $s_{\nu}\cdot s_{\nu'}\cdot s_{\gamma_{k+1}}=s_{\nu}\cdot s_{\gamma_{k+1}}\cdot s_{\nu'}=s_{\gamma}\cdot s_{\nu-\gamma}\cdot s_{\gamma_{k+1}}\cdot s_{\nu'}=s_{\gamma}\cdot s_{\nu-\gamma}\cdot s_{\nu'}\cdot s_{\gamma_{k+1}}\leq s_{\gamma}\cdot s_{c-\gamma}\cdot s_{\gamma_{k+1}}$ which follows by $s_{\nu}\cdot s_{\nu'}\cdot s_{\gamma_{k+1}}\cdot s_{\gamma_{1}}\cdot s_{\gamma_{2}}\cdot ...\cdot s_{\gamma_{k}}\leq s_{\gamma}\cdot s_{c-\gamma} \cdot s_{\gamma_{k+1}}\cdot s_{\gamma_{1}}\cdot s_{\gamma_{2}}\cdot ...\cdot s_{\gamma_{k}}.$ Moreover, $(c-\gamma )\cap \gamma_{k+1}\neq \emptyset$ and $(c-\gamma) \cap \gamma_{k+1}\neq c-\gamma,\gamma_{k+1}$ since $\gamma< \gamma_{k+1}.$  Also, observe that $(c-\gamma)> \gamma_{i}$ and $(c-\gamma)\perp \gamma_{i}$ for all $i=1,2,...,k$ since $\text{supp}(\gamma_{k+1})$ and $\text{supp}(\gamma_{i})$ are disconnected for any $i=1,2,...,k$. Thus we can consider this case under case a) above. Now, if $\nu\cap \gamma_{k+1}$ is not a root then the proof is similar.

\end{proof}

\begin{lemma} \label{lemma19.5} Let $\nu, \nu' ,\mu, \mu' \in \Pi_{\text{aff}}^{\text{re},\,+}$ such that $\nu+\nu'=\mu+\mu'=c$ and either $\nu'\leq \mu'$ or both $\nu'>\mu$ and $\nu'\perp \mu$. Also, let $\gamma_{1},\gamma_{2},...,\gamma_{k} \in \Pi_{\text{aff}}^{\text{re},\,+}$ such that $\text{supp}(\gamma_{i})$ and $\text{supp}(\gamma_{j})$ are disconnected for any $1\leq i,j\leq k$ such that $i\neq j$. Furthermore, suppose that either $\mu'\cap \gamma_{i}=\emptyset$ or both $\mu'>\gamma_{i}$ and $\mu'\perp \gamma_{i}$ for any $1\leq i\leq k.$ Then $s_{\nu}\cdot s_{\nu'}\cdot (s_{\mu}\cdot s_{\mu'})^m\cdot s_{\gamma_{1}}\cdot s_{\gamma_{2}}\cdot ...\cdot s_{\gamma_{k}} \leq (s_{\beta}\cdot s_{\beta'})^{m+1}\cdot s_{\gamma_{1}}\cdot s_{\gamma_{2}}\cdot ...\cdot s_{\gamma_{k}}$ for some affine positive real roots, $\beta,\beta'  \in \Pi_{\text{aff}}^{\text{re},\,+}$ such that either $\beta'\cap \gamma_{i}=\emptyset$ or both $\beta'>\gamma_{i}$ and $\beta'\perp \gamma_{i}$ for any $i=1,2,...,k$ where $m\geq 1$ is an integer.

\end{lemma}

\begin{proof} First, if $\nu'=\mu'$ then we can take $\beta'=\mu'.$ We will assume that $\nu'\neq \mu'.$ If both $\nu'>\mu$ and $\nu'\perp \mu$ are true then by Lemma \ref{lemma17.98} we get $s_{\nu}\cdot s_{\nu'}\cdot (s_{\mu}\cdot s_{\mu'})^m\leq (s_{\mu}\cdot s_{\mu'})^{m+1}$ which follows by $s_{\nu}\cdot s_{\nu'}\cdot (s_{\mu}\cdot s_{\mu'})^m\cdot  s_{\gamma_{1}}\cdot s_{\gamma_{2}}\cdot ...\cdot s_{\gamma_{k}} \leq (s_{\mu}\cdot s_{\mu'})^{m+1}\cdot s_{\gamma_{1}}\cdot s_{\gamma_{2}}\cdot ...\cdot s_{\gamma_{k}}$. Thus we can take $\beta'=\mu'$ in this case. Suppose that $\nu'<\mu'$. Now, observe that there is a subset, $I$ of $\{\gamma_{1},\gamma_{2},...,\gamma_{k}\}$ such that $\nu'\cap \gamma_{i}=\emptyset$ for all $\gamma_{i}\in I$ and $\nu'>\gamma_{j}$ and $\nu'\perp \gamma_{j}$   for all $\gamma_{j}\in \{\gamma_{1},\gamma_{2},...,\gamma_{k}\} \setminus I$ since we have either $\mu'\cap \gamma_{i}=\emptyset$ or both $\mu'>\gamma_{i}$ and $\mu'\perp \gamma_{i}$ for any $1\leq i\leq k.$ Then $\nu'\cap \gamma_{i}=\emptyset$ for all $\gamma_{i}\in I$ since $\nu'< \mu'.$ Here, we have two cases;

a)  If we have either $\nu'\cap \gamma_{j}=\emptyset$ or $\nu'>\gamma_{j}$ and $\nu'\perp \gamma_{j}$ for any $\gamma_{j}$ such that $\gamma_{j}\in \{\gamma_{1},\gamma_{2},...,\gamma_{k}\} \setminus I$ then we can take $\beta'=\nu'$ since $s_{\nu}\cdot s_{\nu'}\cdot (s_{\mu}\cdot s_{\mu'})^m\leq (s_{\nu}\cdot s_{\nu'})^{m+1}$ by Lemma \ref{lemma17.95} which follows by $s_{\nu}\cdot s_{\nu'}\cdot (s_{\mu}\cdot s_{\mu'})^m\cdot z_{d-(m+1)c}\leq (s_{\nu}\cdot s_{\nu'})^{m+1}\cdot z_{\mathbf{d}-(m+1)c}.$

b) Assume that we have neither $\nu'\cap \gamma_{j}=\emptyset$ nor both $\nu'>\gamma_{j}$ and $\nu'\perp \gamma_{j}$  for all $\gamma_{j}\in B$ for some  $B \subseteq \{\gamma_{1},\gamma_{2},...,\gamma_{k}\} \setminus I$.  Note that $s_{\gamma_{i}}$ and $s_{\gamma_{j}}$ Hecke commutes for any $1\leq i,j\leq k$ so we can assume that $\gamma_{1}\in B.$ Now, note that $\text{supp}(\mu) $ and $\text{supp}(\gamma_{1}) $ are disconnected since  $\mu'>\gamma_{1}$ and $\mu'\perp \gamma_{1}$ so $\gamma_{1}$ Hecke commute with $s_{\mu}$ and $s_{\mu'}.$ Thus $s_{\nu}\cdot s_{\nu'}\cdot (s_{\mu}\cdot s_{\mu'})^m\cdot s_{\gamma_{1}}\cdot s_{\gamma_{2}}\cdot ...\cdot s_{\gamma_{k}} = s_{\nu}\cdot s_{\nu'}\cdot s_{\gamma_{1}}\cdot (s_{\mu}\cdot s_{\mu'})^m  s_{\gamma_{2}}\cdot ...\cdot s_{\gamma_{k}}$. Here, we have several cases;

\begin{itemize}

\item If $\nu'=\gamma_{1}$ then $s_{\nu}\cdot s_{\nu'}\cdot s_{\gamma_{1}}=s_{c-\gamma_{1}}\cdot s_{\gamma_{1}}\cdot s_{\gamma_{1}}$. By Lemma \ref{lemma17.3}, $s_{c-\gamma_{1}}\cdot s_{\gamma_{1}}\cdot s_{\gamma_{1}}\leq s_{\gamma_{1}}\cdot s_{c-\gamma_{1}}\cdot s_{\gamma_{1}}$ which follows by $s_{\nu}\cdot s_{\nu'}\cdot s_{\gamma_{1}}\cdot (s_{\mu}\cdot s_{\mu'})^m \cdot  s_{\gamma_{2}}\cdot ...\cdot s_{\gamma_{k}} \leq s_{\gamma_{1}}\cdot s_{c-\gamma_{1}}\cdot s_{\gamma_{1}}\cdot (s_{\mu}\cdot s_{\mu'})^m \cdot s_{\gamma_{2}}\cdot ...\cdot s_{\gamma_{k}}=s_{\gamma_{1}}\cdot s_{c-\gamma_{1}}\cdot  (s_{\mu}\cdot s_{\mu'})^m \cdot s_{\gamma_{1}}\cdot s_{\gamma_{2}}\cdot ...\cdot s_{\gamma_{k}}.$  Here, observe that $(c-\gamma_{1})>\mu$ and  $(c-\gamma_{1})\perp \mu$ since $\mu'>\gamma_{1}$ and $\mu'\perp \gamma_{1}$. So by Lemma  \ref{lemma17.98} we get $s_{\gamma_{1}}\cdot s_{c-\gamma_{1}}\cdot (s_{\mu}\cdot s_{\mu'})^m\leq (s_{\mu}\cdot s_{\mu'})^{m+1}$ which follows by $s_{\gamma_{1}}\cdot s_{c-\gamma_{1}}\cdot  (s_{\mu}\cdot s_{\mu'})^m \cdot s_{\gamma_{1}}\cdot s_{\gamma_{2}}\cdot ...\cdot s_{\gamma_{k}}. \leq (s_{\mu}\cdot s_{\mu'})^{m+1}\cdot s_{\gamma_{1}}\cdot s_{\gamma_{2}}\cdot ...\cdot s_{\gamma_{k}}$. Thus we can take $\beta'=\mu'$ in this case.

\item If $\nu'<\gamma_{1}$ then $\nu\cap \gamma_{1}\neq \emptyset$ since $\nu+\nu'=c.$ Assume that  $\nu\cap \gamma_{1}$ is a root and let $\gamma:= \nu\cap \gamma_{1}.$ Then, by part $2)$ in Lemma \ref{lemma5.7} we have $s_{\nu}\cdot s_{\gamma_{1}}=s_{\gamma}\cdot s_{\nu-\gamma}\cdot s_{\gamma_{1}}.$ Here, note that $\nu-\gamma+\nu'=c-\gamma$ is a root and by the same lemma $s_{\nu-\gamma}\cdot s_{\nu'}\leq s_{c-\gamma}.$ Hence $s_{\nu}\cdot s_{\nu'}\cdot s_{\gamma_{1}}=s_{\nu}\cdot s_{\gamma_{1}}\cdot s_{\nu'}=s_{\gamma}\cdot s_{\nu-\gamma}\cdot s_{\gamma_{1}}\cdot s_{\nu'}=s_{\gamma}\cdot s_{\nu-\gamma}\cdot s_{\nu'}\cdot s_{\gamma_{1}}\leq s_{\gamma}\cdot s_{c-\gamma}\cdot s_{\gamma_{1}}$ which follows by $s_{\nu}\cdot s_{\nu'}\cdot s_{\gamma_{1}}\cdot (s_{\mu}\cdot s_{\mu'})^m \cdot  s_{\gamma_{2}}\cdot ...\cdot s_{\gamma_{k}} \leq s_{\gamma}\cdot s_{c-\gamma}\cdot s_{\gamma_{1}}\cdot (s_{\mu}\cdot s_{\mu'})^m \cdot s_{\gamma_{2}}\cdot ...\cdot s_{\gamma_{k}}=s_{\gamma}\cdot s_{c-\gamma}\cdot  (s_{\mu}\cdot s_{\mu'})^m \cdot s_{\gamma_{1}}\cdot s_{\gamma_{2}}\cdot ...\cdot s_{\gamma_{k}}.$  Here, observe that $(c-\gamma)>\mu$ and  $(c-\gamma)\perp \mu$ since $\gamma<\gamma_{1}$ where $\mu'>\gamma_{1}$ and $\mu'\perp \gamma_{1}$. So by Lemma  \ref{lemma17.98} we get $s_{\gamma}\cdot s_{c-\gamma}\cdot (s_{\mu}\cdot s_{\mu'})^m\leq (s_{\mu}\cdot s_{\mu'})^{m+1}$ which follows by $s_{\gamma}\cdot s_{c-\gamma}\cdot  (s_{\mu}\cdot s_{\mu'})^m \cdot s_{\gamma_{1}}\cdot s_{\gamma_{2}}\cdot ...\cdot s_{\gamma_{k}}. \leq (s_{\mu}\cdot s_{\mu'})^{m+1}\cdot s_{\gamma_{1}}\cdot s_{\gamma_{2}}\cdot ...\cdot s_{\gamma_{k}}$. Again we can take $\beta'=\mu'$ in this case. If $\nu\cap \gamma_{1}$ is not a root then the proof is similar.

\item If $\gamma_{1}<\nu'$ then since $\nu'$ is not perpendicular to $\gamma_{1}$ we have  $<\nu',\gamma_{1}^{\vee}>=1$ so $s_{\gamma_{1}}(\nu')=\nu'-\gamma_{1}$. So $\nu'-\gamma_{1}$ is a root which implies that $\nu+\gamma_{1}=c-\nu'+\gamma_{1}=c-(\nu'-\gamma_{1})$ is also a root. Thus  $s_{\nu}\cdot s_{\nu'}\cdot s_{\gamma_{1}}=s_{\nu}\cdot s_{\gamma_{1}}\cdot s_{\nu'}=s_{\nu+\gamma_{1}}\cdot s_{\nu'}$ by Lemma \ref{lemma5.7}. Here, note that $(\nu+\gamma_{1})\cap \nu'=\gamma_{1}$. So by the same lemma we get $s_{\nu+\gamma_{1}} \cdot s_{\nu'}=s_{\nu+\gamma_{1}}\cdot s_{\nu'-\gamma_{1}}\cdot s_{\gamma_{1}}$ which follows by $s_{\nu}\cdot s_{\nu'}\cdot s_{\gamma_{1}}\cdot (s_{\mu}\cdot s_{\mu'})^m \cdot  s_{\gamma_{2}}\cdot ...\cdot s_{\gamma_{k}} \leq s_{\nu+\gamma_{1}}\cdot s_{\nu'-\gamma_{1}}\cdot s_{\gamma_{1}} \cdot (s_{\mu}\cdot s_{\mu'})^m \cdot s_{\gamma_{2}}\cdot ...\cdot s_{\gamma_{k}}=s_{\nu+\gamma_{1}}\cdot s_{\nu'-\gamma_{1}}\cdot (s_{\mu}\cdot s_{\mu'})^m \cdot s_{\gamma_{1}}\cdot s_{\gamma_{2}}\cdot ...\cdot s_{\gamma_{k}}.$ Here, observe that $(\nu'-\gamma_{1})\cap \gamma_{1}=\emptyset.$ Now, if $B=\{\gamma_{1}\}$ then for any $1\leq i\leq k$ we have either $(\nu'-\gamma_{1})\cap \gamma_{i}=\emptyset$ or both $(\nu'-\gamma_{1})>\gamma_{i}$ and $(\nu'-\gamma_{1})\perp \gamma_{i}$. Also, note that $(\nu'-\gamma_{1})<\mu'$ since $\nu'<\mu'$ which implies that $s_{\nu+\gamma_{1}}\cdot s_{\nu'-\gamma_{1}}\cdot (s_{\mu}\cdot s_{\mu'})^m\leq (s_{\nu+\gamma_{1}}\cdot s_{\nu'-\gamma_{1}})^{m+1}$ which follows by $s_{\nu+\gamma_{1}}\cdot s_{\nu'-\gamma_{1}}\cdot  (s_{\mu}\cdot s_{\mu'})^m \cdot s_{\gamma_{1}}\cdot s_{\gamma_{2}}\cdot ...\cdot s_{\gamma_{k}}. \leq (s_{\nu+\gamma_{1}}\cdot s_{\nu'-\gamma_{1}})^{m+1}\cdot s_{\gamma_{1}}\cdot s_{\gamma_{2}}\cdot ...\cdot s_{\gamma_{k}}$. Thus we can take $\beta'=\nu'-\gamma_{1}$ in this case. Now, suppose that $|B|> 1$. Then $B$ can have at most two elements since we have neither $\nu'\cap \gamma_{j}=\emptyset$ nor both $\nu'>\gamma_{j}$ and $\nu'\perp \gamma_{j}$  for all $\gamma_{j}\in B$ which implies that $\nu'$ intersect with $\gamma_{j}$ and it is not strictly bigger than and not perpendicular to $\gamma_{j}$ for any $\gamma_{j}\in B$ and also by the fact that $\text{supp}(\gamma_{i})$ and $\text{supp}(\gamma_{j})$ are disconnected for any $1\leq i,j\leq k$. Now, let $\gamma_{q}\in B$ such that $\gamma_{q}\neq \gamma_{1}.$ Then one can show that $s_{\nu+\gamma_{1}}\cdot s_{\nu'-\gamma_{1}}\cdot  (s_{\mu}\cdot s_{\mu'})^m \cdot s_{\gamma_{1}}\cdot s_{\gamma_{2}}\cdot ...\cdot s_{\gamma_{k}}. \leq (s_{\beta}\cdot s_{\beta'})^{m+1}\cdot s_{\gamma_{1}}\cdot s_{\gamma_{2}}\cdot ...\cdot s_{\gamma_{k}}$ for some affine positive real roots $\beta,\beta'  \in \Pi_{\text{aff}}^{\text{re},\,+}$ such that either $\beta'\cap \gamma_{i}=\emptyset$ or both $\beta'>\gamma_{i}$ and $\beta'\perp \gamma_{i}$ for any $i=1,2,...,k$ by considering the relationships between $\nu'-\gamma_{1}$ and $\gamma_{q}$ which are identical with those between $\nu'$ and $\gamma_{1}$ above under case b).

\end{itemize}

\end{proof}

The result for the most general case is:

\begin{thm} \label{thm20} Let $\mathbf{d}=(d_0,d_1,...,d_{n-1})>c$ be a degree and $m=\text{min}\{d_0,d_1,...,d_{n-1}\}$. Also, assume that $z_{\mathbf{d}-mc}=s_{\gamma_{1}} s_{\gamma_{2}}...s_{\gamma_{k}}$ for some k where this expression is obtained in Theorem \ref{thm12.7}. Then 

\begin{enumerate}

\item[1)] $\Gamma_{\mathbf{d}}( id)=\{t_{m\beta'}z_{\mathbf{d}-mc}: \beta' \in \Pi_{\text{aff}}^{\text{re},+}(\mathbf{d}-mc)\}\cup \{t_{m(\beta'-c)}z_{\mathbf{d}-mc}: \beta' \in \Pi_{\text{aff}}^{\text{re},+}(\mathbf{d}-mc)\}$

\item[2)] $|\Gamma_{\mathbf{d}}( id)|=|\{\beta': \beta' \in \Pi_{\text{aff}}^{\text{re},\,+}(\mathbf{d}-mc)\}|$

\item[3)] For all $w\in \Gamma_{\mathbf{d}}(id)$, $\ell(w)=2m(n-1)+\ell(z_{\mathbf{d}-mc})$ 

\end{enumerate}

where $ \Pi_{\text{aff}}^{\text{re},\,+}(\mathbf{d}-mc)$ is the set of $\beta' \in \Pi_{\text{aff}}^{\text{re},\,+}$ such that $ \beta'<c$ and either  $\beta' \cap \gamma_{i}=\emptyset $ or both $\beta'>\gamma_{i} $ and $\beta'\perp \gamma_{i}$ for any $ i=1,...,k$. 
\end{thm}

\begin{proof}  $1)$ First, note that we have either $\text{supp}(\gamma_{i})$ and $\text{supp}(\gamma_{j})$ are disconnected or both $\gamma_{i}\perp \gamma_{j}$ and $\gamma_{i},\gamma_{j}$ are comparable, for any $1\leq i,j\leq k$ such that $i\neq j$; see Theorem \ref{thm12.7}. Now, we define $B:=\{\gamma_{i_{1}},\gamma_{i_{2}},...,\gamma_{i_{t}}\} \subseteq \{\gamma_{1},\gamma_{2},...,\gamma_{k}\}$ such that for any $1\leq i\leq k$, we have $\gamma_{i}\leq \gamma_{i_{l}}$ for some $\gamma_{i_{l}}\in B.$ Note that, any two elements of $B$ have disconnected supports. Now, observe that for a root $\gamma_{i_{j}}\in B$ and $\beta' \in \Pi_{\text{aff}}^{\text{re},+}$ such that  $\beta'\cap \gamma_{i_{j}}=\emptyset$ or both $\beta'>\gamma_{i_{j}}$ and $\beta'\perp \gamma_{i_{j}}$  implies that we have either  $\beta'\cap \gamma_{i}=\emptyset$ or both $\beta'>\gamma_{i}$ and $\beta'\perp \gamma_{i}$  for all $1\leq i \leq k$ such that $\gamma_{i} \leq \gamma_{i_{j}},$ respectively. Also, for any $1\leq i \leq k$ such that  $\beta'\cap \gamma_{i}=\emptyset$ or both $\beta'>\gamma_{i}$ and $\beta'\perp \gamma_{i}$ implies that we have either  $\beta'\cap \gamma_{i_{j}}=\emptyset$ or both $\beta'>\gamma_{i_{j}}$ and $\beta'\perp \gamma_{i_{j}}$ where $\gamma_{i} \leq \gamma_{i_{j}},$ respectively. So we may assume that $B=\{\gamma_{1},\gamma_{2},...,\gamma_{k}\}$ i.e we will suppose that $\text{supp}(\gamma_{i})$ and $\text{supp}(\gamma_{j})$ are disconnected for any $1\leq i,j\leq k$ such that $i\neq j$.

Let $id \stackrel{\beta_{1}} \longrightarrow s_{\beta_{1}} \stackrel{\beta_{2}} \longrightarrow s_{\beta_{1}}s_{\beta_{2}} ...\stackrel{\beta_{r}} \longrightarrow w=s_{\beta_{1}}s_{\beta_{2}}...s_{\beta_{r}}$ be a path in the moment graph such that $\sum_{i=1}^{r} \beta_{i} \leq \mathbf{d}$. Note that by Remark \ref{remark12.9} we can assume that $w=s_{\beta_{1}} \cdot s_{\beta_{2}}\cdot...\cdot s_{\beta_{r}}$ where $\sum_{i=1}^{r} \beta_{i} = \mathbf{d}$ and $\beta_{i}<c$ for all $i$. Now, note that $t_{m\beta'}$ and $t_{m(\beta'-c)}$ are given by $(s_{\beta} s_{\beta'})^m$ for some $\beta,\beta'  \in \Pi_{\text{aff}}^{\text{re},\,+}$ such that $\beta+\beta'=c$; see the proof of Lemma \ref{lemma10}. So it is enough to show that $w\leq (s_{\beta} s_{\beta'})^{m}z_{\mathbf{d}-mc}$ where we have either $\beta'\cap \gamma_{i}=\emptyset$ or both $\beta'>\gamma_{i}$ and $\beta'\perp \gamma_{i}$ for any $i=1,2,...,k.$ Furthermore, we can assume that there is an integer $p$ such that $\sum_{i=1}^{p}\beta_{i}=c$ and $\sum_{i=p+1}^{r}\beta_{i}=\mathbf{d}-c$, by Lemma \ref{lemma17.71}.  We will prove the statement by the induction on $m.$ 

First, suppose that $m=1.$ Then by Equation \ref{equation14.5}, $s_{\beta_{1}}\cdot s_{\beta_{2}}\cdot ...\cdot s_{\beta_{p}}\leq s_{\alpha}\cdot s_{\alpha'}$ for some $\alpha, \alpha' \in \Pi_{\text{aff}}^{\text{re},\,+}$ such that $\alpha+\alpha'=c$. Moreover, $s_{\beta_{p+1}}\cdot ...\cdot s_{\beta_{r}}\leq z_{\mathbf{d}-c}$, by Theorem \ref{thm13}. Thus 
$$w=(s_{\beta_{1}}\cdot s_{\beta_{2}}\cdot ...\cdot s_{\beta_{p}})\cdot (s_{\beta_{p+1}}\cdot ...\cdot s_{\beta_{r}})\leq (s_{\alpha}\cdot s_{\alpha'})\cdot z_{\mathbf{d}-c}.$$ Now, by Lemma \ref{lemma19.3} we have $s_{\alpha}\cdot s_{\alpha'}\cdot s_{\gamma_{1}}\cdot s_{\gamma_{2}}\cdot ...\cdot s_{\gamma_{k}} \leq s_{\beta}\cdot s_{\beta'}\cdot s_{\gamma_{1}}\cdot s_{\gamma_{2}}\cdot ...\cdot s_{\gamma_{k}}$ for some affine positive real roots, $\beta,\beta'  \in \Pi_{\text{aff}}^{\text{re},\,+}$ such that we have either $\beta'\cap \gamma_{i}=\emptyset$ or both $\beta'>\gamma_{i}$ and $\beta'\perp \gamma_{i}$ for any $i=1,2,...,k.$

Now, we will suppose that the statement is true for $m$ and prove that it is also true for $m+1.$ Again, by Equation \ref{equation14.5} we have $s_{\beta_{1}}\cdot s_{\beta_{2}}\cdot ...\cdot s_{\beta_{p}}\leq s_{\alpha}\cdot s_{\alpha'}$ for some $\alpha, \alpha' \in \Pi_{\text{aff}}^{\text{re},\,+}$ such that $\alpha+\alpha'=c$ and $s_{\beta_{p+1}}\cdot ...\cdot s_{\beta_{r}}\leq (s_{\mu}\cdot s_{\mu'})^m \cdot z_{\mathbf{d}-(m+1)c}$ for some $\mu,\mu' \in \Pi_{\text{aff}}^{\text{re},\,+}$ such that $\mu+\mu'=c$ where either $\mu'\cap \gamma_{i}=\emptyset$ or $\mu'>\gamma_{j}$ and $\mu'\perp \gamma_{j}$ by induction assumption. Thus  
$$w=(s_{\beta_{1}}\cdot s_{\beta_{2}}\cdot ...\cdot s_{\beta_{p}})\cdot (s_{\beta_{p+1}}\cdot ...\cdot s_{\beta_{r}})\leq (s_{\alpha}\cdot s_{\alpha'}) \cdot(s_{\mu}\cdot s_{\mu'})^m \cdot z_{\mathbf{d}-(m+1)c}.$$  Now, observe that $s_{\alpha}\cdot s_{\alpha'}\cdot s_{\mu}\leq s_{\nu}\cdot s_{\nu'}\cdot s_{\mu}$ for some $\nu, \nu' \in \Pi_{\text{aff}}^{\text{re},\,+}$ such that $\nu+\nu'=c$ and either $\nu'\leq \mu'$ or both $\nu'>\mu$ and $\nu'\perp \mu$ by Equation \ref{equation17.92}. Hence $s_{\alpha}\cdot s_{\alpha'}\cdot (s_{\mu}\cdot s_{\mu'})^m\cdot z_{\mathbf{d}-(m+1)c}\leq (s_{\nu}\cdot s_{\nu'})\cdot (s_{\mu}\cdot s_{\mu'})^m\cdot z_{\mathbf{d}-(m+1)c}.$ By Lemma \ref{lemma19.5} we have $s_{\nu}\cdot s_{\nu'}\cdot (s_{\mu}\cdot s_{\mu'})^m\cdot s_{\gamma_{1}}\cdot s_{\gamma_{2}}\cdot ...\cdot s_{\gamma_{k}} \leq (s_{\beta}\cdot s_{\beta'})^{m+1}\cdot s_{\gamma_{1}}\cdot s_{\gamma_{2}}\cdot ...\cdot s_{\gamma_{k}}$ for some affine positive real roots, $\beta, \beta'$ such that $\beta+\beta'=c$ where either $\beta'\cap \gamma_{i}=\emptyset$ or both $\beta'>\gamma_{i}$ and $\beta'\perp \gamma_{i}$ for any $i=1,2,...,k.$ Furthermore, by Lemma \ref{lemma12.71} $, z_{\mathbf{d}-(m+1)c}=s_{\gamma_{1}} s_{\gamma_{2}}...s_{\gamma_{k}}$ is reduced so $s_{\gamma_{1}} s_{\gamma_{2}}...s_{\gamma_{k}}=s_{\gamma_{1}} \cdot s_{\gamma_{2}}\cdot ...\cdot s_{\gamma_{k}}$ which follows by $(s_{\beta}\cdot s_{\beta'})^{m+1}\cdot s_{\gamma_{1}}\cdot s_{\gamma_{2}}\cdot ...\cdot s_{\gamma_{k}}=(s_{\beta}\cdot s_{\beta'})^{m+1}\cdot (s_{\gamma_{1}}s_{\gamma_{2}}...s_{\gamma_{k}})=(s_{\beta}\cdot s_{\beta'})^{m+1}\cdot z_{\mathbf{d}-(m+1)c}$. Now by Lemma \ref{lemma12.72}, the element $(s_{\beta}s_{\beta'})^{m+1} z_{\mathbf{d}-(m+1)c}$ is reduced so $(s_{\beta}\cdot s_{\beta'})^{m+1}\cdot z_{\mathbf{d}-(m+1)c}=(s_{\beta}s_{\beta'})^{m+1} z_{\mathbf{d}-(m+1)c}.$

$2)$ Let $(s_{\beta}s_{\beta'})^{m}z_{\mathbf{d}-mc}$ and $ (s_{\nu}s_{\nu'})^{m}z_{\mathbf{d}-mc}\in \Gamma_{\mathbf{d}}( id).$ We need to show that $(s_{\beta}s_{\beta'})^{m}z_{\mathbf{d}-mc}\neq (s_{\nu}s_{\nu'})^{m}z_{\mathbf{d}-mc}$ if $\beta'\neq \nu'.$ Now assume that $\beta'\neq \nu'$ but 
\begin{equation}\label{equation21}(s_{\beta}s_{\beta'})^{m}s_{\gamma_{1}} s_{\gamma_{2}}...s_{\gamma_{k}}= (s_{\nu}s_{\nu'})^{m}s_{\gamma_{1}} s_{\gamma_{2}}...s_{\gamma_{k}}.\end{equation} By Lemma \ref{lemma12.72} both $(s_{\beta}s_{\beta'})^{m}s_{\gamma_{1}} s_{\gamma_{2}}...s_{\gamma_{k}}$ and $ (s_{\nu}s_{\nu'})^{m}s_{\gamma_{1}} s_{\gamma_{2}}...s_{\gamma_{k}}$ are reduced, so Equation \ref{equation21} implies that $(s_{\beta}s_{\beta'})^{m}= (s_{\nu}s_{\nu'})^{m}$ but this is a contradiction by Lemma \ref{lemma10}.

$3)$ Let $w\in \Gamma_{\mathbf{d}}( id)$. Then by Lemma \ref{lemma12.72} we have $\ell(w)=2m(n-1)+\ell(z_{\mathbf{d}-mc})$.

\end{proof}

\begin{remark} Let $m$ is a positive integer. We get 
\begin{equation}\label{equation18.8} \Gamma_{\mathbf{d}} (id)=\{(s_{\beta}s_{\beta'})^{m} z_{\mathbf{d}-mc}:\beta' \in  \Pi_{\text{aff}}^{\text{re},\,+}(\mathbf{d}-mc)\,\, \text{such that}\,\, \beta+\beta'=c\},
\end{equation}
see the proof of Theorem \ref{thm20}.

\end{remark}

\begin{example}  Let $W_{\text{aff}}$ be the affine Weyl group associated to $A_{3}^{(1)}$ and $\mathbf{\mathbf{d}}=(6,5,8,5).$ Then $\mathbf{d}=5c+(1,0,3,0)$ so $m=5$ and $\mathbf{d}-5c=(1,0,3,0)$. Also, $z_{\mathbf{d}-5c}=s_{\alpha_{0}}\cdot s_{\alpha_{2}}\cdot s_{\alpha_{2}}\cdot s_{\alpha_{2}}=s_{\alpha_{0}}s_{\alpha_{2}}.$ Hence by Theorem \ref{thm0.62} we get 
\begin{equation*} \displaystyle \Gamma_{\mathbf{d}}( id)= \{t_{5\beta'_{1}} z_{\mathbf{d}-5c}, t_{5\beta'_{2}}z_{\mathbf{d}-5c},t_{5\beta'_{4}} z_{\mathbf{d}-5c},t_{5(\beta'_{3}-c)} z_{\mathbf{d}-5c} \}
\end{equation*}
where $\beta'_{1}=\alpha_{1}$, $\beta'_{2}=\alpha_{3}$, $\beta'_{3}=\alpha_{0}+\alpha_{1}+\alpha_{3}$, $\beta'_{4}=\alpha_{1}+\alpha_{2}+\alpha_{3}$. Moreover, for all $w\in \Gamma_{\mathbf{d}}( id)$, $\ell(w)=2m(n-1)+\ell(z_{\mathbf{d}-mc})=2\cdot 5\cdot3+2=32$.

\end{example}

\subsection{Reduction from $\Gamma_{\mathbf{d}}(w)$ to $\Gamma_{\mathbf{d}}(id)$} \label{reduction}

In this section we will discuss our last result which indicates that to calculate $\Gamma_{\mathbf{d}}(w)$ for any given degree $\mathbf{d}$ and $w \in W_{\text{aff}}$ one only needs to calculate $ \Gamma_{\mathbf{d}}(id).$

\begin{thm} \label{thm22} Let $w \in W_{\text{aff}}$ and $\mathbf{d}$ be any degree. Then 
$$\Gamma_{\mathbf{d}}(w)=\text{max}\left \{w\cdot u: u \in \Gamma_{\mathbf{d}}( id) \right \}.$$\end{thm}
\begin{proof} Let
$$w \stackrel{\beta_{1}} \longrightarrow ws_{\beta_{1}} \stackrel{\beta_{2}} \longrightarrow ws_{\beta_{1}}s_{\beta_{2}} ...\stackrel{\beta_{k}} \longrightarrow ws_{\beta_{1}}s_{\beta_{2}}...s_{\beta_{k}}$$
be a path in the moment graph such that $\sum_{i=1}^{k} \beta_{i}\leq \mathbf{d}$. Then 
$$ws_{\beta_{1}}s_{\beta_{2}}...s_{\beta_{k}}\leq w\cdot (s_{\beta_{1}}s_{\beta_{2}}...s_{\beta_{k}}) \leq w\cdot u$$
for some $u\in \Gamma_{\mathbf{d}}( id)$. To complete the proof we need to show that the element $w\cdot u $ can be reached by a path which starts with an element that is smaller than and equal to $w$ and has a degree at most $\mathbf{d}$. Now, note that $w\cdot u=vu$ for some $v\in W_{\text{aff}}$ such that $v\leq w$ by property $e)$ of the Hecke product in Section \ref{hecke product}.   In addition, $u=s_{\beta'_{1}}s_{\beta'_{2}}...s_{\beta'_{t}}$ for some positive affine real roots $\beta'_{i}$, $i=1,2,...,t$ such that   $\sum_{i=1}^{t} \beta'_{i}\leq \mathbf{d}$ since $u\in \Gamma_{\mathbf{d}}(id)$.  Thus 
$$v \stackrel{\beta'_{1}} \longrightarrow vs_{\beta'_{1}} \stackrel{\beta'_{2}} \longrightarrow vs_{\beta'_{1}}s_{\beta'_{2}} ...\stackrel{\beta'_{t}} \longrightarrow vs_{\beta'_{1}}s_{\beta'_{2}}...s_{\beta'_{k}}=vu$$
is the path we are looking for.

 \end{proof}


\addcontentsline{toc}{chapter}{Biblography}
\bibliographystyle{amsplain} \bibliography{biblio}
%



\end{document}